\documentclass{amsart}

\usepackage{amsthm, amsmath, amssymb, enumitem, geometry, caption, graphicx, hyperref, fontenc, inputenc}

\usepackage{mathrsfs}
\usepackage{stmaryrd}
\usepackage[all,2cell]{xy} 
\UseAllTwocells

\DeclareSymbolFont{EulerScriptc}{U}{eus}{m}{n}
\SetSymbolFont{EulerScriptc}{bold}{U}{eus}{b}{n} 
\DeclareSymbolFontAlphabet\matheusm{EulerScriptc}
\newcommand\eusm{\matheusm}

\theoremstyle{plain}
\newtheorem{theo}{Theorem}[section]
\newtheorem{conj}[theo]{Conjecture}

\newtheorem{lemm}[theo]{Lemma}
\newtheorem{prop}[theo]{Proposition}
\newtheorem{coro}[theo]{Corollary}
\newtheorem{defi}[theo]{Definition}

\theoremstyle{definition}
\newtheorem{rema}[theo]{Remark}

\renewcommand{\llbracket}{[\kern-.08pc[}
\renewcommand{\rrbracket}{]\kern-.08pc]}

\newcommand{\AnDA}{\mathbf{AnDA}}
\newcommand{\rat}{\mathsf{rat}}
\newcommand{\op}{\mathrm{op}}
\newcommand{\ra}{\rightarrow}
\newcommand{\Id}{\mathrm{Id}}
\newcommand{\Hom}{\mathrm{Hom}}
\newcommand{\Osheaf}{\mathscr{O}}
\newcommand{\bbQ}{\mathbf{Q}}
\newcommand{\bbZ}{\mathbf{Z}}
\newcommand{\bbC}{\mathbf{C}}
\newcommand{\bbD}{\mathbb{D}}
\newcommand{\bbG}{\mathbf{G}}
\newcommand{\Spec}{\mathop{\mathrm{Spec}}}
\newcommand{\Perv}{\mathscr P}
\newcommand{\Cb}{\mathrm{C}^{\mathrm{b}}}
\newcommand{\Db}{\mathrm{D}^{\mathrm{b}}}
\newcommand{\Dbc}{\mathrm{D}^{\mathrm{b}}_{\mathrm{c}}}
\newcommand{\Dbh}{\mathrm{D}^{\mathrm{b}}_{\mathrm{h}}}
\newcommand{\Dbm}{\mathrm{D}^{\mathrm{b}}_{\mathrm{m}}}
\newcommand{\Hc}{\mathsf{H}}
\newcommand{\pHc}{{}^{\mathsf{p}}\Hc}
\newcommand{\Ker}{\mathop{\mathrm{Ker}}\nolimits}
\newcommand{\Img}{\mathop{\mathrm{Im}}\nolimits}
\newcommand{\colim}{\mathop{\mathrm{colim}}}
\newcommand{\anc}{\mathrm{an}}
\newcommand{\Sch}{\mathrm{Sch}}
\newcommand{\Sm}{\mathrm{Sm}}
\newcommand{\sM}{\mathscr{M}}
\newcommand{\sN}{\mathscr N}
\newcommand{\sA}{\mathscr A}
\newcommand{\sP}{\mathscr P}
\newcommand{\sPm}{{\mathscr P}_{\mathrm{m}}}
\newcommand{\Mod}{\mathbf{Mod}}
\newcommand{\bbA}{\mathbf{A}}
\newcommand{\pH}{{}^{p}H}
\newcommand{\MM}{\mathbf{MM}}
\newcommand{\ct}{\mathsf{ct}}
\newcommand{\bbN}{\mathbb{N}}
\newcommand{\MHS}{\mathbf{MHS}}
\newcommand{\xra}{\xrightarrow}
\newcommand{\D}{\mathrm{D}}
\newcommand{\Bti}{\mathrm{Bti}}
\newcommand{\real}{\mathrm{real}}
\newcommand{\eff}{\mathrm{eff}}
\newcommand{\SH}{\mathbf{SH}}
\newcommand{\dM}{\mathfrak{M}}
\newcommand{\HM}{\mathbf{HM}}
\newcommand{\Pairs}{\mathbf{Pairs}}
\newcommand{\smalllrcarre}{\pmb{\lrcorner}}
\newcommand{\smallucarre}{\pmb{\ulcorner}}
\newcommand{\smallllcarre}{\pmb{\llcorner}}
\newcommand{\doublecarre}{\pmb{\boxbar}}
\newcommand{\carre}{\pmb{\boxempty}}
\newcommand{\AR}{\mathbf{R}}
\newcommand{\AL}{\mathbf{L}}
\newcommand{\Ab}{\mathbf{A}}
\newcommand{\DA}{\mathbf{DA}}
\newcommand{\Path}{\mathbf{P}}
\newcommand{\ad}{\mathrm{ad}}
\newcommand{\tr}{\mathrm{tr}}
\newcommand{\qv}{\mathrm{qv}}
\newcommand{\Log}{{\mathscr L}og}
\newcommand{\un}{\mathrm{un}}
\newcommand{\Ind}{\mathrm{Ind}}
\newcommand{\eusmCb}{\matheusm{C}^{\mathrm{b}}_{\mathrm{dg}}}
\newcommand{\eusmDb}{\matheusm{D}^{\mathrm{b}}_{\mathrm{dg}}}
\newcommand{\sMct}{{}^{\ct}\kern-.3em\sM}
\newcommand\iso{\xra{\sim}}
\newcommand\Scat{\matheusm{S}}
\newcommand{\Ret}{\mathfrak{R}^{\mathrm{\acute{e}t}}}

\DeclareMathOperator{\Ob}{Ob}
\DeclareMathOperator{\Chow}{Chow}
\DeclareMathOperator\Ext{Ext}
 
\numberwithin{equation}{section}

\newenvironment{acknowledgements}{%
  \begin{abstract}
}{%
  \end{abstract}
}

\begin{document}

\title[Operations on perverse mixed motives]{The four operations on perverse motives}
\author{Florian Ivorra}
\address{Institut de recherche math\'ematique de Rennes\\ UMR 6625 du CNRS\\ Universit\'e de Rennes 1\\
Campus de Beaulieu\\
35042 Rennes cedex (France)\\
}
\email{florian.ivorra@univ-rennes1.fr}
\author{Sophie Morel}
\address{UMPA UMR CNRS 5669\\
ENS Lyon Site Monod\\
46 Allée d'Italie\\
69364 Lyon Cedex 07}
\email{sophie.morel@ens-lyon.fr}



\begin{abstract}
Let $k$ be a field of characteristic zero with a fixed embedding $\sigma:k\hookrightarrow \bbC$ into the field of complex numbers. Given a $k$-variety $X$, we use the triangulated category of \'etale motives with rational coefficients on $X$  to construct an abelian category $\sM(X)$ of perverse mixed motives.  We show that over $\Spec(k)$ the category obtained is canonically equivalent to the usual category of Nori motives and that the derived categories $\Db(\sM(X))$ are equipped with the four operations of Grothendieck (for morphisms of quasi-projective $k$-varieties) as well as nearby and vanishing cycles functors and a formalism of weights.

In particular, as an application, we show that many classical constructions done with perverse sheaves, such as intersection cohomology groups or Leray spectral sequences, are motivic and therefore compatible with Hodge theory. This recovers and strengthens work by Zucker, Saito, Arapura and de Cataldo--Migliorini and provide an arithmetic proof of the pureness of intersection cohomology with coefficients in a geometric variation of Hodge structures.
\end{abstract}

\maketitle


\setcounter{tocdepth}{2}
\tableofcontents

\section*{Introduction}

Let $k$ be a field of characteristic zero with a fixed embedding $\sigma:k \hookrightarrow \bbC$ into the field of complex numbers. A $k$-variety is a separated $k$-scheme of finite type. Unless otherwise specified, we will only consider quasi-projective $k$-varieties.

In the present work, we construct the four operations on the derived categories of perverse Nori motives. In order to combine the tools provided by Ayoub in~\cite{AyoubI,AyoubII} and Beilinson 
in~\cite{MR923134,MR923133} in a most efficient way, we define the abelian category of perverse Nori motives on a given $k$-variety as a byproduct of the triangulated category of constructible \'etale motives on the same variety.  Over the base field the category obtained still coincides with the usual category of Nori motives but now, as we show, it is possible to equip the derived categories of these abelian categories with the four operations of Grothendieck as well as nearby and vanishing cycles functors. However, we leave the construction of the tensor product and internal $\Hom$ operations on these categories to a later paper.

In particular, as an application, we show that many classical constructions done with perverse sheaves, such as intersection cohomology groups or Leray spectral sequences, are motivic and therefore compatible with Hodge theory. This recovers and strengthens works by Zucker~\cite{MR534758}, Saito~\cite{MHMII}, Arapura~\cite{MR2178703} and de Cataldo--Migliorini~\cite{MR2680404}. Moreover it provides an arithmetic proof via reduction to positive characteristic and the Weil conjectures of the pureness of the Hodge structure on intersection cohomology with coefficients in a geometric variation of Hodge structures.

\subsubsection*{Conjectural picture and some earlier works}

Before going into more detail about the content of this paper, let us discuss perverse motives from the perspective of perverse sheaves and recall parts of the conjectural picture and related earlier works. 

For someone interested in perverse sheaves, perverse motives can be thought of as perverse sheaves of geometric origin. However, the classical definition of these perverse sheaves as a full subcategory of the category of all perverse sheaves is not entirely satisfactory. Indeed, this category contains too many morphisms and consequently, as we take kernels and cokernels of morphisms which shouldn't be considered, too many objects. For example, perverse sheaves of geometric origin should define mixed Hodge modules and therefore any morphism between them should also be a morphism of mixed Hodge modules. Therefore, one expects the category of perverse motives/perverse sheaves of geometric origin to be an abelian category endowed with a faithful -- but not full -- exact functor into the category of perverse sheaves.

According to Grothendieck, there should exist a $\bbQ$-linear abelian category $\MM(k)$ whose objects are called mixed motives. Given an embedding $\sigma:k \hookrightarrow \bbC$, the category $\MM(k)$ should come with a faithful exact functor
\[\MM(k)\ra\MHS\]
to the category of (polarizable) mixed $\bbQ$-Hodge structures $\MHS$, called the realization functor. The mixed Hodge structure on the $i$-th Betti cohomology group $H^i(X)$ of a given $k$-variety $X$ should come via the realization functor from a mixed motive $H_{\mathbf{M}}^i(X)$. 
The appealing beauty of this picture lies in the expected properties of this category, in particular, the conjectural relations between extension groups and algebraic cycles (see e.g. Jannsen's~\cite{MR1265533}), or the relation with period rings and motivic Galois groups (see e.g. for a survey Ayoub's~\cite{MR3202399}).

As part of Grothendieck's more general cohomological program, the category $\MM(k)$ should underlie a system of coefficients. For any $k$-variety $X$, there should exist an abelian category $\MM(X)$ of mixed motives along with a realization functor into the category of mixed Hodge modules (or simply of sheaves of $\bbQ$-vector spaces) on the associated analytic space $X^\anc$,
and their derived categories should satisfy a formalism of (adjoint) triangulated functors
\[\xymatrix{{\Db(\MM(X))}\ar@<-1ex>[r]_-{f^{{\mathbf M}}_*} & {\Db(\MM(Y))}\ar@<-1ex>[l]_-{f_{{\mathbf M}}^*}\ar@<1ex>[r]^-{f_{\mathbf{M}}^!} & {\Db(\MM(X)),}\ar@<1ex>[l]^-{f_!^{{\mathbf M}}} }\]
a formalism which has been at the heart of Grothendieck's approach to every cohomology theory. Then, for a $k$-variety $a:X\ra\Spec k$,
the motive $H^i_{\mathbf M}(X)$ would be given as the $i$-th cohomology of  the image under $a^{\mathbf M}_*$ of a complex of mixed motives $\bbQ^{\mathbf M}_X$ that should realize to the standard constant sheaf $\bbQ_X$ on $X^\anc$.
Grothendieck was looking for abelian categories modeled after the categories of constructible sheaves, but as pointed out by Be{\u\i}linson and Deligne one could/should also look for categories modeled after perverse sheaves (see e.g. Deligne's~\cite{MR1265528}).

Many attempts have been made to carry out  at least partially but unconditionally Grothendieck's program.

The most successful attempt in constructing the triangulated category of mixed motives (that is, conjecturally, the derived category of $\MM(X)$) stems from Morel--Voevodsky's stable homotopy theory of schemes. The best candidate so far is the triangulated category $\DA_\ct(X)$ of constructible \'etale motivic sheaves (with rational coefficients) extensively studied by Ayoub in~\cite{AyoubI,AyoubII,AyoubEtale}. The theory developed in \cite{AyoubI,AyoubII} provides these categories with the Grothendieck four operations and, as shown by Voevodsky in~\cite{MR1883180}, Chow groups of  smooth algebraic $k$-varieties can be computed as extension groups in the category $\DA_\ct(k)$.

On the abelian side, Nori has constructed a candidate for the abelian category of mixed motives over $k$. The construction of Nori's abelian category $\HM(k)$ is tannakian in essence and, since it is a category of comodules over some Hopf algebra, it comes with a built-in motivic Galois group. Moreover any Nori motive has a canonical weight filtration and Arapura has shown in~\cite[Theorem 6.4.1]{MR2995668} that the full subcategory of pure motives coincides with the semi-simple abelian category defined by Andr\'e in~\cite{MR1423019} using motivated algebraic cycles (see also Theorem~10.2.5 in the book~\cite{BookNori}
of Huber and M\"uller--Stach).
More generally, attempts have been made to define Nori motives over $k$-varieties. Arapura has defined a constructible variant in~\cite{MR2995668} and the first author a perverse variant in~\cite{IvorraPNM}. However, the Grothendieck four operations have not been constructed (at least in their full extent) in those contexts. For example in~\cite{MR2995668}, the direct image functor is only available for structural morphisms or projective morphisms and no extraordinary inverse image is defined.

Note that the two different attempts should not be unrelated. One expects the triangulated category $\DA_\ct(X)$ to possess a special $t$-structure (called the motivic $t$-structure) whose heart should be the abelian category of mixed motives. This is a very deep conjecture, even for $X=\Spec k$, which implies for example the Lefschetz and K\"unneth type standard conjectures  (see Beilinsons's~\cite{MR2953406}). As of now, the extension groups in Nori's abelian category of mixed motives are known to be related with algebraic cycles only very poorly.

However, striking unconditional relations between the two different approaches have still been obtained. In particular, in~\cite{MR3649230}, Gallauer--Choudhury have shown that the motivic Galois group constructed by Ayoub in~\cite{AyoubCrelleI,AyoubCrelleII} using the triangulated category of \'etale motives is isomorphic to the motivic Galois group obtained by Nori's construction.

\subsubsection*{Content of this paper}

Let us now describe more precisely the content of our paper. Given a $k$-variety $X$, consider the bounded derived category $\Dbc(X,\bbQ)$ of sheaves of $\bbQ$-vector spaces with algebraically constructible cohomology on the analytic space $X^\anc$ associated with the base change of $X$ along $\sigma$ and the category of perverse sheaves $\sP(X)$ which is the heart of the self-dual perverse $t$-structure on $\Dbc(X,\bbQ)$ introduced in the book~\cite{BBD} by Beilinson--Bernstein--Deligne. Let $\DA_\ct(X)$ be the triangulated category of constructible \'etale motivic sheaves (with rational coefficients) which is a full triangulated subcategory of the $\bbQ$-linear counterpart of the stable homotopy category of schemes $\SH(X)$ introduced by Morel and Voevodsky (see~\cite{MR1813224, MR1648048} and Jardine's~\cite{MR1787949}). 
This category has been extensively studied by Ayoub in~\cite{AyoubI,AyoubII,AyoubEtale} and comes with a realization functor 
\[\Bti_X^*:\DA_\ct(X)\ra\Dbc(X,\bbQ)\]
(see Ayoub's~\cite{AyoubBetti}) and thus, by composing with the perverse cohomology functor, with a homological functor $\pH^0_\sP$ with values in $\Perv(X)$.

 The category of perverse motives considered in the present paper is defined (see the first author's~\ref{sec:PerverseMotives}) as the universal factorization  
\[\DA_\ct(X)\xra{\pH^0_\sM} \sM(X)\xra{\rat^\sM_X}\Perv(X)\]
of $\pH^0_\sP$, where $\sM(X)$ is an abelian category, $\pH^0_\sM$ is a homological functor and $\rat^\sM_X$ is a faithful exact functor. This kind of universal construction goes back to Freyd and is recalled in  \ref{sec:CatPre}. As we see in \ref{sec:weights}, $\ell$-adic perverse sheaves can also be used to defined the category of perverse motives (see \ref{defi:perversel} and \ref{prop:independancel}).

 Given a morphism of $k$-varieties $f:X\ra Y$, the four functors
\begin{equation}\label{Eq:FourFunDbc}
\xymatrix{{\Dbc(X,\bbQ)}\ar@<-1ex>[r]_-{f^\sP_*} & {\Dbc(Y,\bbQ)}\ar@<-1ex>[l]_-{f^*_\sP}\ar@<-1ex>[r]_-{f_\sP^!} & {\Dbc(X,\bbQ)}\ar@<-1ex>[l]_-{f_!^\sP} }
\end{equation}
where developed by Verdier~\cite{MR1610971} (see also Kashiwara--Schapira's book~\cite{MR1299726}) on the model of the theory developed by Artin, Grothendieck et al. for \'etale and $\ell$-adic sheaves in~\cite{MR0354654}. The nearby and vanishing cycles functors 
\[
\Psi_g:\Dbc(X_\eta,\bbQ)\ra\Dbc(X_\sigma,\bbQ)\qquad \Phi_g:\Dbc(X,\bbQ)\ra\Dbc(X_\sigma,\bbQ)\]
associated with a morphism $g:X\ra\bbA^1_k$ were constructed  by Grothendieck in~\cite{MR0354657} (here $X_\eta$ denotes the generic fiber and $X_\sigma$ the special fiber). By a theorem of Gabber, the functors $\psi_g:=\Psi_g[-1]$ and $\phi_g:=\Phi_g[-1]$ are $t$-exact for the perverse $t$-structures and thus induce exact functors
\begin{equation}\label{Eq:PsiFunDbc}
\psi_g:\Perv(X_\eta)\ra\Perv(X_\sigma)\qquad \phi_g:\Perv(X)\ra\Perv(X_\sigma).
\end{equation}
 In this work, we prove that $\sM(k)$ is canonically equivalent to the abelian category $\HM(k)$ of Nori motives (see \ref{prop:CompNori}) and that the four operations \eqref{Eq:FourFunDbc} (for morphisms of quasi-projective $k$-varieties) and the functors \eqref{Eq:PsiFunDbc} can be canonically lifted along the functors 
 \[\Db(\sM(X))\xra{\rat^\sM_X}\Db(\Perv(X))\xra{\real}\Dbc(X,\bbQ),\]
where $\real$ is the realization functor of Section~3.1 of the book~\cite{BBD} by Beilinson--Bernstein--Deligne, which has been shown to be an equivalence in Beilinson's~\cite{MR923133}, to (adjoint) triangulated functors
\[\xymatrix{{\Db(\sM(X))}\ar@<-1ex>[r]_-{f^{\sM}_*} & {\Db(\sM(Y))}\ar@<-1ex>[l]_-{f_{\sM}^*}\ar@<-1ex>[r]_-{f_{\sM}^!} & {\Db(\sM(X))}\ar@<-1ex>[l]_-{f_!^{\sM}} }\]
and to exact functors
\[\psi_f^\sM:\sM(X_\eta)\ra\sM(X_\sigma)\qquad \phi_f^\sM:\sM(X)\ra\sM(X_\sigma).\]

Relying on Ayoub's~\cite{AyoubI,AyoubII} and on the compatibility of the Betti realization with the four operations, our strategy consists in establishing enough of the formalism to show that the categories $\Db(\sM(X))$ underlie a stable homotopical $2$-functor in the sense of~\cite{AyoubI} (see \ref{theo:maintheo}), so that the rest of the formalism is obtained from~\cite{AyoubI,AyoubII}. 
The existence of the direct image by a closed immersion or the inverse image by a smooth morphism are obtained immediately via the universal property (see \ref{sec:PerverseMotives}). However, to construct the inverse image by a closed immersion (see \ref{sec:pullback}), we need to develop analogues, for \'etale motives, of the functors and gluing exact sequences obtained by Be{\u\i}linson in~\cite{MR923134}. This is done in \ref{sec:nearbycycles} and uses derivators, and  the logarithmic specialization system of Ayoub, see~\cite{AyoubII,AyoubEtale}. The proof of the main theorem is carried out in \ref{sec:MainTheo} and the most important step is the proof of the existence of the direct image by the projection of the affine line $\bbA^1_X$ onto its base $X$ (see \ref{prop:rightadjoint}). We conclude this section by the aforementioned applications to intersection cohomology and Leray spectral sequences.

In  \ref{sec:weights}, we show that perverse motives can also be defined using $\ell$-adic perverse sheaves and that they admit a notion of weights. We deduce the existence
of the weight filtration from the properties of Bondarko's Chow weight structure and from
the Weil conjectures (cf. Th\'eor\`eme 2 of Deligne's paper~\cite{WeilII}).
Then, using the strict support decomposition of pure objects to reduce to the case of a point, we
show that the category of pure objects of a given weight is semi-simple. As
an application, we get the existence of a weight structure on the derived category of $\sM(X)$ and an arithmetic proof of Zucker's theorem~\cite[Theorem p.416]{MR534758} for geometric variations of Hodge structures (see \ref{theo:weightIC} and \ref{coro:weightIC}).

\bigskip

\begin{acknowledgements}
The authors are very grateful to the referees for their valuable comments and suggestions that helped improve this work. The first author also expresses his thanks to A. Huber and S. Kebekus for the wonderful semester spent in Freiburg and the numerous conversations on Nori motives and Hodge theory, and the
second author would like to thank J. Ayoub, A. Huber and M. A. de Cataldo for useful conversations, and D. Hansen for pointing out some mistakes in an earlier version of the manuscript.

The present paper was partly written while the first author was a Marie-Curie FRIAS COFUND fellow at the Freiburg Institute of Advanced Studies and the second author was a professor at Princeton University and an
invited professor at the \'Ecole Normale Sup\'erieure de Lyon and the Universit\'e Lyon 1. 
They would like to thank these institutions for their hospitality and support during the academic year 2017-2018.
Also, the work of the second author was supported by the ANR HQDIAG project
no ANR-21-CE40-0015.

\end{acknowledgements}

\section{Categorical preliminaries}\label{sec:CatPre}

Let us recall in this section a few universal constructions related to abelian and triangulated categories. They date back to Freyd's construction of the abelian hull of an additive category, see his paper~\cite{MR0209333} and have been considered in many different forms in various works (see e.g. Verdier's book~\cite{MR1453167}, Krause's~\cite{MR1487973}, Prest's~\cite{MR2791358} and the paper~\cite{BVP}
by Barbieri-Viale and Prest).

Let $\matheusm S$ be an additive category. Let $\Mod(\eusm S)$ be the category of right $\eusm S$-modules, that is, the category of additive functors from $\eusm S^\op$ to the category $\mathbf{Ab}$ of abelian groups. The category $\Mod(\eusm S)$ is abelian and a sequence of right $\eusm S$-modules
\[0\ra F'\ra F\ra F''\ra 0\]
is exact if and only if for every $s\in\eusm S$ the sequence of abelian groups
\[0\ra F'(s)\ra F(s)\ra F''(s)\ra 0\]
is exact.

 A right $\eusm S$-module $F$ is said to be of finite presentation if there exist objects $s,t$ in $\eusm S$ and an exact sequence 
\[\eusm S(-,s)\ra\eusm S(-,t)\ra F\ra 0\]
in $\Mod(\eusm S)$.

\begin{defi} Let $\eusm S$ be an additive category. We denote by $\AR(\eusm S)$ the full subcategory of $\Mod(\eusm S)$ consisting of right $\eusm S$-modules of finite presentation.
\end{defi}

The category $\AR(\eusm S) $ is an additive category with cokernels (the cokernel of a morphism of right $\eusm S$-modules of finite presentation is of finite presentation) and the Yoneda functor 
\[h_{\eusm S}:\eusm S\ra\AR(\eusm S) \]
is a fully faithful additive functor. Recall that, given a morphism $t\ra s$ in $\eusm S$, a morphism $r\ra t$ is called a pseudo-kernel if the sequence
\[\eusm S(-,r)\ra\eusm S(-,t)\ra\eusm S(-,s)\]
is exact in $\Mod(\eusm S)$. The category $\AR(\eusm S)$ is abelian if and only if $\eusm S$ has pseudo-kernels (see \cite[Theorem~1.4]{MR0209333} and
\cite[Lemma~2.2]{MR1487973}).
 It also satisfies the following universal property.

\begin{prop}(\cite[2.1 and Lemma~2.6]{MR1487973}).
Let $\eusm S$ be an additive category. Let $\eusm A$ be an additive category with cokernels and $F:\eusm S\ra\eusm A$ be an additive functor, then there exists, up to a natural isomorphism, a unique right exact functor $\AR(\eusm S)\ra \eusm A$ that extends $F$. Moreover, if $\eusm S$ and $\eusm A$ admit pseudo-kernels,
then this functor is exact if and only if $F$ preserves pseudo-kernels. \end{prop}

Note that the construction can be dualized so that there is a universal way to add kernels to an additive category. One simply set $\AL(\eusm S):=\AR(\eusm S^\op)^\op$. 
The two constructions can be combined to add both cokernels and kernels at the same time. Let $\eusm S$ be an additive category and let 
\[\Ab^\ad(\eusm S):=\AL(\AR(\eusm S)).\]
Then the functor $h:\eusm S\ra \Ab^\ad(\eusm S)$ is a fully faithful additive functor and $\Ab^\ad(\eusm S)$ is an abelian category which enjoys the following universal property (this is Freyd's abelian hull).
\begin{prop}
Let $\eusm A$ be an abelian category and $F:\eusm S\ra\eusm A$ be an additive functor, then there exists, up to a natural isomorphism, a unique exact functor $\Ab^\ad(\eusm S)\ra \eusm A$ that extends $F$. 
\end{prop}

Note also that the category $\Ab^\ad(\eusm S)$ is canonically equivalent to $\AR(\AL(\eusm S))$.

This construction can be used to provide an alternative description of Nori's category (see the paper~\cite{BVP} by Barbieri-Viale and Prest). Let $\eusm Q$ be a quiver, $\eusm A$ be an abelian category and $T:\eusm Q\ra \eusm A$ be a representation. Let $\Path(\eusm Q)$ be the path category and $\Path(\eusm Q)^\oplus$ be its additive completion obtained by adding finite direct sums.
Then, up to natural isomorphisms, we have a commutative diagram
\[\xymatrix{{\eusm Q}\ar[r]\ar[rd]_-{T} & {\Path(\eusm Q)^\oplus}\ar[r]\ar[d]^-{\varrho_T} &{\Ab^{\ad}(\Path(\eusm Q)^\oplus)=:\Ab^{\qv}(\eusm Q)}\ar[ld]^-{\rho_T}\\
{} & {\eusm A} & {}}\]
where $\varrho_T $ is an additive functor and $\rho_T $ an exact functor. The kernel of $\rho_T$ is a thick subcategory of $\Ab^{\qv}(\eusm Q)$ and we define the abelian category $\Ab^{\qv}(\eusm Q,T)$ to be the quotient of $\Ab^{\qv}(\eusm Q)$ by this kernel. By construction, the functor $\rho_T$ has a canonical factorization 
\[\Ab^{\qv}(\eusm Q)\xra{\pi_T}\Ab^{\qv}(\eusm Q,H)\xra{r_T}\eusm A\]
where $\pi_T$ is an exact functor and $r_T$ is a faithful exact functor. If we denote by  $\overline{T}$ the composition of the representation $\eusm Q\ra\Ab^{\qv}(\eusm Q)$ and the functor $\pi_T:\Ab^{\qv}(\eusm Q)\ra\Ab^{\qv}(\eusm Q,T)$, it provides a canonical factorization of $T$:
\[\eusm Q\xra{\overline{T}}\Ab^{\qv}(\eusm Q,T)\xra{r_T}\eusm A\]
where $\overline{T}$ is a representation and $r_T$ is a faithful exact functor. It is easy to see that the above factorization is universal among all factorizations of $T$ of the form
\[\eusm Q\xra{R}\eusm B\xra{s}\eusm A\]
where $\eusm B$ is an abelian category, $R$ is a representation and $s$ is a faithful exact functor. In particular, whenever Nori's construction is available, e.g. if $\eusm A$ is Noetherian, Artinian and has finite dimensional $\Hom$-groups over $\bbQ$ (see the paper~\cite{IvorraPNM} by the first author), then the category $\Ab^{\qv}(\eusm Q, T)$ is equivalent to Nori's abelian category associated with the quiver representation $T$.

Let us consider the case when $\eusm Q$ is an additive category and $T$ is an additive functor. Then, up to natural isomorphisms, we have a commutative diagram
\[\xymatrix{{\eusm Q}\ar[r]\ar[rd]_-{T} & {\Ab^\ad(\eusm Q)}\ar[d]^-{T^*} \\
{} & {\eusm A} & {}}\]
where $T^*$ is an exact functor. The kernel of $T^*$ is a thick subcategory of $\Ab^\ad(\eusm Q)$ and we define the abelian category $\Ab^\ad(\eusm Q,T)$ to be the quotient of $\Ab^\ad(\eusm Q)$ by this kernel. 

\begin{lemm}\label{LemmCompA}
Let $\eusm Q$ and $\eusm A$ be additive categories. Then, for every additive functor $T:\eusm Q\ra \eusm A$, the categories $\Ab^{\qv}(\eusm Q,T)$ and $\Ab^\ad(\eusm Q,T)$ are canonically equivalent.
\end{lemm}
\begin{proof}
 To see this, it suffices to check that the factorization
\[\eusm Q\ra\Ab^\ad(\eusm Q,T)\ra\eusm A\]
satisfies the universal property that defines $\Ab^{\qv}(\eusm Q,T)$. Consider a factorization of the representation $T$ of the quiver $\eusm Q$
\[\eusm Q\xra{R}\eusm B\xra{s}\eusm A\]
where $\eusm B$ is an abelian category, $R$ is a representation and $s$ is a faithful exact functor. Since $s$ is faithful, $R$ must be an additive functor. Therefore, we get a commutative diagram (up to natural isomorphisms)
\[\xymatrix{{\eusm Q}\ar@/_2em/[rdd]_-{T}\ar[rd]_-{R}\ar[rr] & {} & {\Ab^{\ad}(\eusm Q)}\ar@{.>}[ld]_-{\textrm{exact}}\ar@/^2em/[ldd]^-{T^*}\\
{} & {\eusm B}\ar[d]^-{s} & {}\\
{} &{\eusm A.} &{}}\]
The exactness and the faithfulness of $s$ imply that the above diagram can be further completed into a 
commutative diagram (up to natural isomorphisms)
\[\xymatrix{{\eusm Q}\ar@/_2em/[rdd]_-{T}\ar[rd]_-{R}\ar[rr] & {} & {\Ab^\ad(\eusm Q)}\ar[ld]_-{\textrm{exact}}\ar@/^6em/[ldd]^-{T^*}\ar[d]\\
{} & {\eusm B}\ar[d]^-{s} & {\Ab^\ad(\eusm Q,T)}\ar@{.>}[l]\ar[ld]\\
{} &{\eusm A.} &{}}\]
This shows the desired universal property.
\end{proof}

Let us finally consider the special case when $\eusm S$ is a triangulated category. In that case the additive category $\eusm S$ has pseudo-kernels and pseudo-cokernels, in particular, 
the category $\Ab^\tr(\eusm S):=\AR(\eusm S)$
is an abelian category.\footnote{This is the abelian category denoted by  $A(\eusm S)$ in Chapter~V of Neeman's book~\cite{MR1812507}.} 
The Yoneda embedding $h_{\eusm S}:\eusm S\ra \Ab^\tr(\eusm S)$ is a homological functor and is universal for this property (see \cite[Theorem 5.1.18]{MR1812507}). In particular, if $\eusm A$ is an abelian category, any homological functor $H:\eusm S\ra\eusm A$ admits a canonical factorization
\[\xymatrix{{\eusm S}\ar[r]^-{h_{\eusm S}} & {\Ab^\tr(\eusm S)}\ar[r]^-{\rho_H} & {\eusm A}}\]
where $\rho_H$ is an exact functor. This factorization of $H$ is universal among all such factorizations.

The kernel of $\rho_H$ is a thick subcategory of $\Ab^\tr(\eusm S)$ and we define the abelian category $\Ab^\tr(\eusm S,H)$ to be the quotient of $\Ab^\tr(\eusm S)$ by this kernel. By construction, the functor $\rho_H$ has a canonical factorization 
\[\Ab^{\tr}(\eusm S)\xra{\pi_H}\Ab^{\tr}(\eusm S,H)\xra{r_H}\eusm A\]
where $\pi_H$ is an exact functor and $r_H$ is a faithful exact functor. Setting $H_{\eusm S}:=\pi_H\circ h_{\eusm S}$, it provides a canonical factorization of $H$:
\[\eusm S\xra{H_{\eusm S}}\Ab^\tr(\eusm S,H)\xra{r_H}\eusm A\]
where $H_{\eusm S}$ is a homological functor and $r_H$ a faithful exact functor. It is easy to see that the above factorization is universal among all factorizations of $H$ of the form
\[\eusm S\xra{L}\eusm B\xra{s}\eusm A\]
where $L$ is a homological functor and $s $ is a faithful exact functor.

We can also see the triangulated category $\eusm S$ simply as a quiver (resp. an additive category) and the homological functor $H:\eusm S\ra\eusm A$ simply as a representation (resp. an additive functor). In particular, we have at our disposal the universal factorizations of the representation $H$:
\[\eusm S\ra\Ab^{\qv}(\eusm S,H)\ra \eusm A\]
and
\[\eusm S\ra\Ab^{\ad}(\eusm S,H)\ra \eusm A\]
where the arrows on the right are exact and faithful functors. 

\begin{lemm}\label{LemmCompB}
Let $\eusm S$ be a triangulated category, $\eusm A$ be an abelian category and $H:\eusm S\ra\eusm A$ be a homological functor. Then, the three abelian categories $\Ab^{\qv}(\eusm S,H)$, $\Ab^\ad(\eusm S,H) $  and $\Ab^\tr(\eusm S,H)$ are canonically equivalent.
\end{lemm}

\begin{proof}
We have seen in \ref{LemmCompA} that $\Ab^\qv(\eusm S,H) $ and $\Ab^\ad(\eusm S,H)$ are canonically equivalent.  Let us prove that so do $\Ab^\ad(\eusm S,H) $ and $\Ab^\tr(\eusm S,H)$. It suffices to check that the factorization
\[\eusm S\ra\Ab^\tr(\eusm S,H)\ra\eusm A\]
satisfies the universal property that defines $\Ab^{\ad}(\eusm S,H)$. Consider a factorization of the additive functor $H$:
\[\eusm Q\xra{R}\eusm B\xra{s}\eusm A\]
where $\eusm B$ is an abelian category, $R$ is an additive functor and $s$ is a faithful exact functor. Since $s$ is faithful, $R$ must be homological. Therefore, we get a commutative diagram (up to natural isomorphisms)
\[\xymatrix{{\eusm S}\ar@/_2em/[rdd]_-{H}\ar[rd]_-{R}\ar[rr] & {} & {\Ab^\tr(\eusm S)}\ar@{.>}[ld]_-{\textrm{exact}}\ar@/^2em/[ldd]\\
{} & {\eusm B}\ar[d]^-{s} & {}\\
{} &{\eusm A.} &{}}\]
The exactness and the faithfulness of $s$ imply that the above diagram can be further completed into a 
commutative diagram (up to natural isomorphisms)
\[\xymatrix{{\eusm S}\ar@/_2em/[rdd]_-{H}\ar[rd]_-{R}\ar[rr] & {} & {\Ab^\tr(\eusm S)}\ar[ld]_-{\textrm{exact}}\ar@/^6em/[ldd]\ar[d]\\
{} & {\eusm B}\ar[d]^-{s} & {\Ab^\tr(\eusm S,T)}\ar@{.>}[l]\ar[ld]\\
{} &{\eusm A.} &{}}\]
This shows the desired universal property.
\end{proof}

\section{Perverse motives}\label{sec:PerverseMotives}

We fix a field $k$ that admits an embedding $\sigma:k\ra\bbC$.
Unless otherwise specified, we will only consider quasi-projective
$k$-varieties in this article.

\subsection{Definition}\label{subsec:defPerverseMotives}

Let $X$ be a quasi-projective $k$-variety. We denote by $X^\anc$ the complex
analytic space associated with the base change of $X$ along
$\sigma$, by $\Dbc(X,\bbQ)$ the category of complexes of sheaves of
$\bbQ$-vector spaces on $X^\anc$ with bounded algebraically
constructible cohomology,
by $\sP(X)$ the heart of the perverse t-structure on $\Dbc(X,\bbQ)$
introduced in Section~2 of the book~\cite{BBD} of Beilinson--Bernstein--Deligne for the self-dual perversity and
by $\DA_\ct(X)$ the triangulated category of constructible \'etale
motivic sheaves with rational coefficients (see for example
Section~3 of Ayoub's paper~\cite{AyoubEtale}).
By \cite[Theorem 16.2.18]{CD}, this last category is equivalent
to the category of constructible Be\u{\i}linson motives studied in Cisinski and
D\'eglise's book~\cite{CD}, and the equivalence commutes with the operations
we will consider later (direct and inverse images and tensor product).
So we will use reference to Ayoub's articles or to the book~\cite{CD},
as convenient.

To construct the abelian category of perverse motives $\sM(X)$ used in the present work, we take $\eusm S$ to be the triangulated category $\DA_{\ct}(X)$ and $H$ to be the homological functor $\pH^0_\sP$ obtained by composing of the Betti realization
\[\Bti^*_X:\DA_\ct(X)\ra \Dbc(X,\bbQ)\]
constructed by Ayoub in~\cite{AyoubBetti} and the perverse cohomology functor $\pH^0:\Dbc(X,\bbQ)\ra\sP(X)$. 

\begin{defi}
Let $X$ be a $k$-variety. The abelian category of perverse motives is the abelian category
\[\sM(X):=\Ab^\tr(\eusm S,H)=\Ab^\tr(\DA_\ct(X),\pH^0_\sP).\]
\end{defi}
By construction the functor $\pH^0_\sP$ has a factorization
\[\DA_\ct(X)\xra{\pH^0_\sM}\sM(X)\xra{\rat^\sM_X}\sP(X)\]
where $\rat^\sM_X$ is a faithful exact functor and $\pH^0_\sM$ is a homological functor. Let us recall the two consequences (denoted by \textbf{P1} and \textbf{P2} below) of the universal property of the factorization
\[\xymatrix{{\DA_\ct(X)}\ar[r]^-{\pH^0_\sM}\rrlowertwocell_{\pH^0_\sP}{\rho_X} &{\sM(X)}\ar[r]^-{\rat^\sM_X} &{\sP(X).}}\]
The property {\textbf{P1}} below is proved in~\cite[Proposition 6.6]{IvorraPNM}. A proof of property \textbf{P2} can be found in Proposition~2.5 of the paper~\cite{IvorraYamazaki} by
Ivorra--Yamazaki.

{\textbf{P1}} For every commutative diagram
\[\xymatrix@C=.5cm@R=.5cm{{\DA_\ct(X)}\ar[r]^-{\pH^0_\sP}\ar[d]^-{\mathsf{F}}\xtwocell[1,1]{}\omit{_\alpha} &{\sP(X)}\ar[d]^-{\mathsf{G}}\\
{\DA_\ct(Y)}\ar[r]_-{\pH^0_\sP} & {\sP(Y)}}\]
where $\mathsf{F} $ is a triangulated functor, $\mathsf{G}$ is an exact functor and $\alpha$ is an invertible natural transformation, there exists a commutative diagram
\[
\xymatrix@C=.5cm@R=.5cm{{\DA_\ct(X)}\xtwocell[1,1]{}\omit{_\beta}\ar[r]\ar[d]^-{\mathsf{F}} & {\sM(X)}\ar[r]\ar[d]^-{\mathsf{E}}\xtwocell[1,1]{}\omit{_\gamma} & {\sP(X)}\ar[d]^-{\mathsf{G}}\\
{\DA_\ct(Y)}\ar[r] & {\sM(Y)}\ar[r] & {\sP(Y)}}
\]
where $\mathsf{E}$ is an exact functor and $\beta,\gamma $ are invertible natural transformations such that the diagram
\[\xymatrix@R=.1cm@C=.1cm{{\DA_\ct(X)}\ar[ddd]^-{\mathsf{F}}\ar@{=}[rrrr]\ar[rd]^(.7){\pH^0_\sM}\xtwocell[4,1]{}\omit{_\beta} &{} & {} &{}\xtwocell[2,0]{}\omit{^\rho_X} &{\DA_\ct(X)}\ar[rrdd]^-{\pH^0_\sP}\ar@{.>}[ddd]^-{\mathsf{F}}\xtwocell[5,2]{}\omit{_\alpha} &{} &{}\\
{} &{\sM(X)}\ar[rd]^-{\rat^\sM_X}\ar[ddd]^-{\mathsf{E}}\xtwocell[4,1]{}\omit{_\gamma} &{} &{} & {}&{} &{}\\
{} &{} &{\sP(X)}\ar@{=}[rrrr]\ar[ddd]^-{\mathsf{G}} &{} & {}&{} &{\sP(X)}\ar[ddd]^-{\mathsf{G}}\\
{\DA_\ct(Y)}\ar@{:}[rrrr]\ar[rd]^(.6){\pH^0_\sM}  &{} &{} &{}\xtwocell[2,0]{}\omit{^\rho_Y} &{\DA_\ct(Y)}\ar@{.>}[rrdd]^-{\pH^0_\sP} &{} &{}\\
{} &{\sM(Y)}\ar[rd]^{\rat^\sM_Y} &{} &{} &{} &{} &{}\\
{} &{} &{\sP(Y)}\ar@{=}[rrrr] &{} & {}&{} &{\sP(Y)}
}\]
is commutative.

{\textbf{P2}}
Let 
\[\xymatrix@R=.3cm@C=.3cm{{\DA_\ct(X)}\ar[rr]^-{\pH^0_\sP}\ar@{=}[dd]\ar[rd]^-{\mathsf{F}_1}\xtwocell[1,3]{}\omit{_\alpha_1}  & {} & {\sP(X)}\ar[rd]^-{\mathsf{G}_1}\ar@{:}[dd] & {}\\
{}\xtwocell[0,1]{}\omit{<2>_\lambda}& {\DA_\ct(Y)}\ar[rr]^-{\pH^0_\sP}\ar@{=}[dd] & {}\xtwocell[0,1]{}\omit{<2>_\mu} & {\sP(Y)}\ar@{=}[dd]\\
{\DA_\ct(X)}\ar[rd]^-{\mathsf{F}_2}\ar@{.>}[rr]^(.7){\pH^0_\sP}\xtwocell[1,3]{}\omit{_\alpha_2} & {} & {\sP(X)}\ar@{.>}[rd]^-{\mathsf{G}_2} & {}\\
{} & {\DA_\ct(Y)}\ar[rr]^-{\pH^0_\sP} & {} & {\sP(Y)}}\]
be a commutative diagram in which $\mathsf{F}_1,\mathsf{F}_2$ are triangulated functors, $\mathsf{G}_1,\mathsf{G}_2$  are exact functors, $\alpha_1,\alpha_2 $ are invertible natural transformations and $\lambda,\mu$ are natural transformations.
Let
\[
\xymatrix@C=.5cm@R=.5cm{{\DA_\ct(X)}\xtwocell[1,1]{}\omit{_\beta_1}\ar[r]\ar[d]^-{\mathsf{F}_1} & {\sM(X)}\ar[r]\ar[d]^-{\mathsf{E}_1}\xtwocell[1,1]{}\omit{_\gamma_1} & {\sP(X)}\ar[d]^-{\mathsf{G}_1}\\
{\DA_\ct(Y)}\ar[r] & {\sM(Y)}\ar[r] & {\sP(Y)}}
\quad
\xymatrix@C=.5cm@R=.5cm{{\DA_\ct(X)}\xtwocell[1,1]{}\omit{_\beta_2}\ar[r]\ar[d]^-{\mathsf{F}_2} & {\sM(X)}\ar[r]\ar[d]^-{\mathsf{E}_2}\xtwocell[1,1]{}\omit{_\gamma_2} & {\sP(X)}\ar[d]^-{\mathsf{G}_2}\\
{\DA_\ct(Y)}\ar[r] & {\sM(Y)}\ar[r] & {\sP(Y)}}\]
be commutative diagrams given in the property {\textbf{P1}}, then there exists a unique natural transformation $\theta:\mathsf{E}_1\ra\mathsf{E}_2$ such that the diagram
\[\xymatrix@R=.3cm@C=.3cm{{\DA_\ct(X)}\ar[rr]^-{\pH^0_\sM}\ar@{=}[dd]\ar[rd]^-{\mathsf{F}_1}\xtwocell[1,3]{}\omit{_\beta_1}  & {} & {\sM(X)}\ar[rd]^-{\mathsf{E}_1}\ar@{:}[dd]\ar[rr]^-{\rat^\sM_X}\xtwocell[1,3]{}\omit{_\gamma_1} & {} & {\sP(X)}\ar@{:}[dd]\ar[rd]^-{\mathsf{G}_1} &{}\\
{}\xtwocell[0,1]{}\omit{<2>_\lambda}& {\DA_\ct(Y)}\ar[rr]^-{\pH^0_\sM}\ar@{=}[dd] & {}\xtwocell[0,1]{}\omit{<2>_\theta} & {\sM(Y)}\ar@{=}[dd]\ar[rr]^-{\rat^\sM_Y} & {}\xtwocell[0,1]{}\omit{<2>_\mu} &{\sP(Y)}\ar@{=}[dd]\\
{\DA_\ct(X)}\ar[rd]^-{\mathsf{F}_2}\ar@{.>}[rr]^(.65){\pH^0_\sM}\xtwocell[1,3]{}\omit{_\beta_2} & {} & {\sM(X)}\ar@{.>}[rd]^-{\mathsf{E}_2}\ar@{.>}[rr]^(.65){\rat^\sM_X}\xtwocell[1,3]{}\omit{_\gamma_2} & {} & {\sP(X)}\ar@{.>}[rd]^-{\mathsf{G}_2} &{}\\
{} & {\DA_\ct(Y)}\ar[rr]^-{\pH^0_\sM} & {} & {\sM(Y)}\ar[rr]^-{\rat^\sM_Y} & {} &{\sP(Y)}}\]
is commutative. 

\subsection{Lifting of $2$-functors}\label{subsec:2func}

As in~\cite[\S 1.1]{AyoubI}, in this work, we only consider strict $2$-categories. However, as in loc.cit., $2$-functors are not necessarily strict (see also Deligne's~\cite{DeligneCross}).

Let $(\Sch/k)$ be the category of quasi-projective $k$-varieties and $\eusm C$ be a subcategory of $(\Sch/k)$.
The properties {\textbf{P1}} and {\textbf{P2}} can be used to lift (covariant or contravariant) $2$-functors. Indeed, let $\mathsf{F}:\eusm C\ra\mathfrak{TR}$ be a $2$-functor (let's say covariant to fix the notation), where $\mathfrak{TR}$ is the
$2$-category of triangulated categories, 
such that $\mathsf{F}(X)=\DA_\ct(X)$ for every $k$-variety $X$ in $\eusm C$. Similarly, let $\mathfrak{Ab}$ be the $2$-category of abelian categories, and 
let $\mathsf{G}:\eusm C\ra\mathfrak{Ab}$ be a $2$-functor such that $\mathsf{G}(X)=\sP(X)$ for every $k$-variety $X$ in $\eusm C$
and that $G(f)$ is exact for every morphism $f$ in $\eusm C$.

We have forgetful functors from $\mathfrak{TR}$ and $\mathfrak{Ab}$ to the $2$-category of additive categories.
Assume that $(\Theta,\alpha):\mathsf{F}\ra\mathsf{G}$ is a $1$-morphism of $2$-functors, where we see $\mathsf{F}$ and $\mathsf{G}$ as $2$-functors into the
$2$-category of additive categories via these forgetful functors,
such that $\Theta_X=\pH^0_\sP$ for every $X\in\eusm C$ and that $\alpha_f$ is invertible for every morphism $f$ in $\eusm C$.

Let $f:X\ra Y$ be a morphism in $\eusm C$. By applying {\textbf{P1}} to the square
 \[\xymatrix@C=.5cm@R=.5cm{{\DA_\ct(X)}\ar[r]^-{\pH^0_\sP}\ar[d]^-{\mathsf{F}(f)}\xtwocell[1,1]{}\omit{_\alpha_f} &{\sP(X)}\ar[d]^-{\mathsf{G}(f)}\\
{\DA_\ct(Y)}\ar[r]_-{\pH^0_\sP} & {\sP(Y)}}\]
we get a commutative diagram
\[
\xymatrix@C=.5cm@R=.5cm{{\DA_\ct(X)}\xtwocell[1,1]{}\omit{_\beta_f}\ar[r]\ar[d]^-{\mathsf{F}} & {\sM(X)}\ar[r]\ar[d]^-{\mathsf{E}(f)}\xtwocell[1,1]{}\omit{_\gamma_f} & {\sP(X)}\ar[d]^-{\mathsf{G}(f)}\\
{\DA_\ct(Y)}\ar[r] & {\sM(Y)}\ar[r] & {\sP(Y)}}
\]
where $\mathsf{E}(f)$ is an exact functor and $\beta_f,\gamma_f$ are invertible natural transformations such that the diagram
\[\xymatrix@R=.1cm@C=.1cm{{\DA_\ct(X)}\ar[ddd]^-{\mathsf{F}(f)}\ar@{=}[rrrr]\ar[rd]^(.6){\pH^0_\sM}\xtwocell[4,1]{}\omit{_\beta_f} &{} & {} &{}\xtwocell[2,0]{}\omit{^\rho_X} &{\DA_\ct(X)}\ar[rrdd]^-{\pH^0_\sP}\ar@{.>}[ddd]^-{\mathsf{F}(f)}\xtwocell[5,2]{}\omit{_\alpha_f} &{} &{}\\
{} &{\sM(X)}\ar[rd]^-{\rat^\sM_X}\ar[ddd]^-{\mathsf{E}(f)}\xtwocell[4,1]{}\omit{_\gamma_f} &{} &{} & {}&{} &{}\\
{} &{} &{\sP(X)}\ar@{=}[rrrr]\ar[ddd]^-{\mathsf{G}(f)} &{} & {}&{} &{\sP(X)}\ar[ddd]^-{\mathsf{G}(f)}\\
{\DA_\ct(Y)}\ar@{:}[rrrr]\ar[rd]^{\pH^0_\sM}  &{} &{} &{}\xtwocell[2,0]{}\omit{^\rho_Y} &{\DA_\ct(Y)}\ar@{.>}[rrdd]^-{\pH^0_\sP} &{} &{}\\
{} &{\sM(Y)}\ar[rd]^{\rat^\sM_Y} &{} &{} &{} &{} &{}\\
{} &{} &{\sP(Y)}\ar@{=}[rrrr] &{} & {}&{} &{\sP(Y)}
}\]
is commutative. Let $X\xra{f}Y\xra{g}Z$ be morphisms in $\eusm C$. By applying {\textbf{P2}} to the commutative diagram
\[\xymatrix@R=.1cm@C=.1cm{{\DA_\ct(X)}\ar@{=}[ddd]\ar[rrrr]^-{\pH^0_\sP}\ar[rrdd]^-{\mathsf{F}(g\circ f)} &{} & {} &{}\xtwocell[2,0]{}\omit{_\alpha_{g\circ f}} &{\sP(X)}\ar[rrdd]^-{\mathsf{G}(g\circ f)}\ar@{:}[ddd] &{} &{}\\
{} &{} &{} &{} & {}&{} &{}\\
{}\xtwocell[0,2]{}\omit{<2>_{\kern2pc c_{\mathsf{F}}(f,g)}} &{} &{\DA_\ct(Z)}\ar[rrrr]^-{\pH^0_\sP}\ar@{=}[ddd] &{} & {} \xtwocell[0,2]{}\omit{<2>_{\kern2pc c_{\mathsf{G}}(f,g)}}&{} &{\sP(X)}\ar@{=}[ddd]\\
{\DA_\ct(X)}\ar@{.>}[rrrr]^-{\pH^0_\sP}\ar[rd]^{\mathsf{F}(f)}  &{} &{}\xtwocell[1,0]{}\omit{_\alpha_f}  &{}&{\sP(X)}\ar@{.>}[rd]^-{\mathsf{G}(f)} &{} &{}\\
{} &{\DA_\ct(Y)}\ar@{.>}[rrrr]^-{\pH^0_\sP}\ar[rd]^{\mathsf{F}(g)} &{} &{} &{}\xtwocell[1,0]{}\omit{_\alpha_g} &{\sP(Y)}\ar@{.>}[rd]^-{\mathsf{G}(g)} &{}\\
{} &{} &{\DA_\ct(Z)}\ar[rrrr]^-{\pH^0_\sP} &{} & {}&{} &{\sP(Z)}
}\]
there exists a unique natural transformation $c_\mathsf{E}(f,g):\mathsf{E}(g\circ f)\ra \mathsf{E}(g)\circ\mathsf{E}(f)$ that fits into the commutative diagram
\[\xymatrix@R=.1cm@C=.1cm{{\DA_\ct(X)}\ar@{=}[ddd]\ar[rrr]^-{\pH^0_\sM}\ar[rrdd]^-{\mathsf{F}(g\circ f)} &{} & {}\xtwocell[2,0]{}\omit{<-5>_{\beta_{g\circ f}}} &{\sM(X)}\ar@{:}[ddd]\ar[rrr]^-{\rat^\sM_X}\ar[rrdd]^-{\mathsf{E}(g\circ f)} & {}&{}\xtwocell[2,0]{}\omit{<-5>_{\gamma_{g\circ f}}} &{\sP(X)}\ar[rrdd]^-{\mathsf{G}(g\circ f)}\ar@{:}[ddd] &{} &{}\\
{} &{} &{} &{} & {} & {}&{}  & &{}\\
{}\xtwocell[0,2]{}\omit{<2>_{\kern2pc c_{\mathsf{F}}(f,g)}} &{} &{\DA_\ct(Z)}\ar[rrr]^-{\pH^0_\sM}\ar@{=}[ddd] &{}\xtwocell[0,2]{}\omit{<2>_{\kern2pc c_{\mathsf{E}}(f,g)}} & & {\sM(Z)}\ar@{=}[ddd]\ar[rrr]^-{\rat^\sM_Z} &{}\xtwocell[0,2]{}\omit{<2>_{\kern2pc c_{\mathsf{G}}(f,g)}} & &{\sP(X)}\ar@{=}[ddd]\\
{\DA_\ct(X)}\ar@{.>}[rrr]^-{\pH^0_\sM}\ar[rd]^{\mathsf{F}(f)}  &{} & {}\xtwocell[1,0]{}\omit{_{\beta_f}} & {\sM(X)}\ar@{.>}[rrr]^-{\rat^\sM_X}\ar@{.>}[rd]^-{\mathsf{E}(f)} &   &{}\xtwocell[1,0]{}\omit{_{\gamma_f}}&{\sP(X)}\ar@{.>}[rd]^-{\mathsf{G}(f)} &{} &{}\\
{} &{\DA_\ct(Y)}\ar@{.>}[rrr]^-{\pH^0_\sM}\ar[rd]^{\mathsf{F}(g)} &{} &{}\xtwocell[1,0]{}\omit{_{\beta_g}}&{\sM(Y)}\ar@{.>}[rrr]^-{\rat^\sM_Y}\ar@{.>}[rd]^-{\mathsf{E}(g)} & {} & {}\xtwocell[1,0]{}\omit{_{\gamma_g}} &{\sP(Y)}\ar@{.>}[rd]^-{\mathsf{G}(g)}  &{}\\
{} &{} &{\DA_\ct(Z)}\ar[rrr]^-{\pH^0_\sM} &{} & & {\sM(Z)}\ar[rrr]^{\rat^\sM_Z} & &{} &{\sP(Z).}
}\]
Using the uniqueness and the fact that the functors $\rat^\sM_X$, for $X$ in $\eusm C$, are faithful it is easy to see that the morphisms $c_{\mathsf{E}}$ satisfy the cocycle condition. Hence $\mathsf{E}:\eusm C\ra\mathfrak{Ab}$ is a $2$-functor and  $(\pH^0_\sM,\beta) $, $(\rat^\sM,\gamma)$ are $1$-morphisms of $2$-functors, where again we see the $2$-functors as functors into the $2$-category of
additive categories. As $1$- and $2$-morphisms in $\mathfrak{Ab}$ are the same as $1$- and $2$-morphisms in the $2$-category of additive categories,
the morphism $(\rat^\sM,\gamma)$ is also a $1$-morphism of $2$-functors $\eusm C\ra\mathfrak{Ab}$.

\subsection{Betti realization of \'etale motives}

Let $f:X\ra Y$ be a morphism of quasi-projective $k$-varieties.
Recall that the category $\Dbc(X,\bbQ)$ is equivalent to the derived category of the abelian category of perverse sheaves on $X$ via the realization functor constructed in \cite[3.1.9]{BBD} (it is known to be an equivalence by \cite[Theorem 1.3]{MR923133}). In particular, the four (adjoint) functors 
\begin{equation*}
\xymatrix{{\Dbc(X,\bbQ)}\ar@<-1ex>[r]_-{f^\sP_*} & {\Dbc(Y,\bbQ)}\ar@<-1ex>[l]_-{f^*_\sP}\ar@<-1ex>[r]_-{f_\sP^!} & {\Dbc(X,\bbQ)}\ar@<-1ex>[l]_-{f_!^\sP} }
\end{equation*}
can be seen as functors between the derived categories of perverse sheaves (for their construction in terms of perverse sheaves see \cite{MR923133}).

Let $\Bti^*_X:\DA_\ct(X)\ra\Dbc(X,\bbQ)$ be the realization functor of \cite{AyoubBetti}. If $f:X\ra Y$ is a morphism of  quasi-projective $k$-varieties, by construction, there is an invertible natural transformation
\[\theta_f:f^*_\sP\circ\Bti^*_Y\ra \Bti^*_X\circ f^*\]
(see \cite[Proposition 2.4]{AyoubBetti}). Let $\theta$ be the collection of these natural transformations, then $(\Bti^*,\theta)$ is a morphism of stable homotopical $2$-functors in the sense of \cite[Definition 3.1]{AyoubBetti}. Following the notation in \cite{AyoubBetti}, we denote by
\begin{align*}
\gamma_f &:\Bti^*_Y\circ f_*\ra f^\sP_*\circ\Bti^*_X\; ; \\
\rho_f &: f^\sP_!\circ\Bti^*_X\ra\Bti^*_Y\circ f_! \;;\\
\xi_f &: \Bti^*_X\circ f^!\ra f^!_\sP\circ \Bti^*_Y\;;
\end{align*}
the induced natural transformations. By \cite[Th\'eor\`eme 3.19]{AyoubBetti} these transformations are invertible.

\subsection{Direct images under affine and quasi-finite morphisms}\label{subsec:DirImageImm}

Let ${}^{\mathrm{QAff}}(\Sch/k)$ be the subcategory of $(\Sch/k)$ with the same objects but in which we only retain the morphisms that are quasi-finite and affine. By \cite[Corollaire 4.1.3]{BBD}, for such a morphism $f:X\ra Y$, the functors
\[f^\sP_*,f^\sP_!:\Dbc(X,\bbQ)\ra\Dbc(Y,\bbQ)\]
are $t$-exact for the perverse $t$-structures. In particular, they induce exact functors between categories of perverse sheaves and by applying the property \textbf{P1} to the canonical transformation $\gamma_f:\Bti^*_Y\circ f_*\ra f^\sP_*\circ\Bti^*_X$,  
we get a commutative diagram
\[
\xymatrix@R=.5cm@C=.5cm{{\DA_\ct(X)}\xtwocell[1,1]{}\omit{^\gamma^{\DA}_f}\ar[r]\ar[d]^-{f_*} & {\sM(X)}\ar[r]\ar[d]_-{f^\sM_*}\xtwocell[1,1]{}\omit{^\gamma^\sM_f} & {\sP(X)}\ar[d]^-{f^\sP_*}\\
{\DA_\ct(Y)}\ar[r] & {\sM(Y)}\ar[r] & {\sP(Y)}}
\]
where $f^\sM_*$ is an exact functor and $\gamma^\DA_f,\gamma^\sM_f$ are invertible natural transformations such that the diagram
\[\xymatrix@R=.1cm@C=.1cm{{\DA_\ct(X)}\ar[ddd]^-{f_*}\ar@{=}[rrrr]\ar[rd]^(.6){\pH^0_\sM}\xtwocell[4,1]{}\omit{^\gamma^\DA_f} &{} & {} &{}\xtwocell[2,0]{}\omit{^\rho_X} &{\DA_\ct(X)}\ar[rrdd]^-{\pH^0_\sP}\ar@{.>}[ddd]^-{f_*}\xtwocell[5,2]{}\omit{^\gamma_f} &{} &{}\\
{} &{\sM(X)}\ar[rd]^-{\rat^\sM_X}\ar[ddd]^-{f^\sM_*}\xtwocell[4,1]{}\omit{^\gamma^\sM_f} &{} &{} & {}&{} &{}\\
{} &{} &{\sP(X)}\ar@{=}[rrrr]\ar[ddd]^-{f^\sP_*} &{} & {}&{} &{\sP(X)}\ar[ddd]^-{f^\sP_*}\\
{\DA_\ct(Y)}\ar@{:}[rrrr]\ar[rd]^{\pH^0_\sM}  &{} &{} &{}\xtwocell[2,0]{}\omit{^\rho_Y} &{\DA_\ct(Y)}\ar@{.>}[rrdd]^-{\pH^0_\sP} &{} &{}\\
{} &{\sM(Y)}\ar[rd]^{\rat^\sM_Y} &{} &{} &{} &{} &{}\\
{} &{} &{\sP(Y)}\ar@{=}[rrrr] &{} & {}&{} &{\sP(Y)}
}\]
is commutative. Moreover, since the natural transformations $\gamma_f$ are compatible with the composition of morphisms (that is, with the connection $2$-isomorphisms),
\ref{subsec:2func} provides a $2$-functor 
\[{}^{\mathrm{QAff}}\mathsf{H}_*^\sM:{}^{\mathrm{QAff}}(\Sch/k)\ra\mathfrak{TR}\]
with ${}^{\mathrm{QAff}}\mathsf{H}_*^\sM(X)=\Db(\sM(X))$ and such that $(\pH^0_\sM,\gamma^\DA)$ and $(\rat^\sM,\gamma^\sM)$ are $1$-morphisms of $2$-functors. 
For every affine and quasi-finite morphism $f:X\ra Y$ we have a natural transformation
\[\gamma^\sM_f:\rat^\sM_Y f^\sM_*\ra f_*^\sP\rat^\sM_X\]
compatible with the composition of morphisms.

\subsection{Inverse image by a smooth morphism}\label{subsec:InvImageLiss}

Let $f:X\ra Y$ be a smooth morphism of $k$-varieties. Then, $\Omega_f$ is a locally free $\Osheaf_X$-module of finite rank. Let $d_f$ its rank (which is constant on each connected component of $X$). Then, $d_f$ is the relative dimension of $f$ (see Grothendieck's~EGA~IV, more precisely \cite[(17.10.2)]{MR0238860}) and if $g:Y\ra Z$ is a smooth morphism, then $d_{g\circ f}=d_g+d_f$ with the obvious abuse of notation (see~\cite[(17.10.3)]{MR0238860}). By~\cite[4.2.4]{BBD}, the functor
\[f^*_\sP[d_f]:\Dbc(Y,\bbQ)\ra\Dbc(X,\bbQ)\]
is $t$-exact for the perverse $t$-structures. In particular, it induces an exact functor between the categories of perverse sheaves and by applying the property \textbf{P1} to the canonical transformation $\theta_f:f^*_\sP\circ\Bti^*_Y\ra \Bti^*_X\circ f^*$,  
we get a commutative diagram
\[
\xymatrix@R=.5cm@C=.5cm{{\DA_\ct(Y)}\xtwocell[1,1]{}\omit{_\kern8pt\theta^{\DA}_f}\ar[r]\ar[d]^-{f^*[d_f]} & {\sM(Y)}\ar[r]\ar[d]\xtwocell[1,1]{}\omit{_\kern8pt\theta^\sM_f} & {\sP(Y)}\ar[d]^-{f^*_\sP[d_f]}\\
{\DA_\ct(X)}\ar[r] & {\sM(X)}\ar[r] & {\sP(X)}}
\]
where the functor in the middle $f_\sM^*[d_f]$ is an exact functor and $\theta^\DA_f,\theta^\sM_f$ are invertible natural transformations such that the diagram
\[\xymatrix@R=.1cm@C=.1cm{{\DA_\ct(Y)}\ar[ddd]^-{f^*[d_f]}\ar@{=}[rrrr]\ar[rd]^(.6){\pH^0_\sM}\xtwocell[4,1]{}\omit{_\theta^\DA_f} &{} & {} &{}\xtwocell[2,0]{}\omit{^\rho_Y} &{\DA_\ct(Y)}\ar[rrdd]^-{\pH^0_\sP}\ar@{.>}[ddd]^-{f^*[d_f]}\xtwocell[5,2]{}\omit{_\theta_f} &{} &{}\\
{} &{\sM(Y)}\ar[rd]^-{\rat^\sM_Y}\ar[ddd]^-{f_\sM^*[d_f]}\xtwocell[4,1]{}\omit{_\theta^\sM_f} &{} &{} & {}&{} &{}\\
{} &{} &{\sP(Y)}\ar@{=}[rrrr]\ar[ddd]^(.4){f_\sP^*[d_f]} &{} & {}&{} &{\sP(Y)}\ar[ddd]^-{f_\sP^*[d_f]}\\
{\DA_\ct(X)}\ar@{:}[rrrr]\ar[rd]^{\pH^0_\sM}  &{} &{} &{}\xtwocell[2,0]{}\omit{^\rho_X} &{\DA_\ct(X)}\ar@{.>}[rrdd]^-{\pH^0_\sP} &{} &{}\\
{} &{\sM(X)}\ar[rd]^{\rat^\sM_X} &{} &{} &{} &{} &{}\\
{} &{} &{\sP(X)}\ar@{=}[rrrr] &{} & {}&{} &{\sP(X)}
}\]
is commutative. 

\begin{rema}
Note that $f^*_\sM A$, given $A$ in $\sM(Y)$,  is not yet defined. 
We define the function $f^*_\sM$ by setting $f^*_\sM:=(f^*_\sM[d_f])[-d_f]$.
\end{rema}

Let ${}^{\mathrm{Liss}}(\Sch/k)$ be the subcategory of $(\Sch/k)$ with the same objects but having as morphisms only the smooth morphisms of $k$-varieties. 
Since the natural transformations $\theta_f$ are compatible with the composition of morphisms (that is, with the connection $2$-isomorphisms), \ref{subsec:2func} provides a contravariant $2$-functor 
\[{}^{\mathrm{Liss}}\mathsf{H}^*_\sM:{}^{\mathrm{Liss}}(\Sch/k)\ra\mathfrak{TR}\]
with ${}^{\mathrm{Liss}}\mathsf{H}^*_\sM(X)=\Db(\sM(X))$ and such that $(\pH^0_\sM,\theta^\DA)$ and $(\rat^\sM,\theta^\sM)$ are $1$-morphisms of $2$-functors. For every smooth morphism $f:X\ra Y$ we have a natural transformation
\[\theta^\sM_f:f^*_\sP\rat^\sM_Y\ra\rat^\sM_X f^*_\sM\]
compatible with the composition of morphisms.

\subsection{Exchange structure}

Let us denote by
\[{}^{\mathrm{Imm}}\mathsf{H}_*^\sM:{}^{\mathrm{Imm}}(\Sch/k)\ra\mathfrak{TR}\]
the restriction of the $2$-functor obtained in \ref{subsec:DirImageImm} to the subcategory ${}^{\mathrm{Imm}}(\Sch/k)$ of $(\Sch/k)$ with the same objects but having as morphisms only the closed immersions of $k$-varieties.
Exchange structures are defined in D\'efinition~1.2.1 of~\cite{AyoubI}.

\begin{prop}\label{prop:ExStructure}
The exchange structure $Ex^*_*$ on $({}^{\mathrm{Liss}}\mathsf{H}^*_\sP,{}^{\mathrm{Imm}}\mathsf{H}_*^\sP)$ with respect to cartesian squares can be lifted to an exchange structure on the pair $({}^{\mathrm{Liss}}\mathsf{H}^*_\sM,{}^{\mathrm{Imm}}\mathsf{H}_*^\sM)$. 
\end{prop}

\begin{proof}
The proposition is a simple application of property \textbf{P2}. Consider a cartesian square 
\[\xymatrix{{X'}\ar[r]^-{i'}\ar[d]^-{f'} &{Y'}\ar[d]^-{f}\\
{X}\ar[r]^i &{Y}}\]
such that $i$ is a closed immersion and $f$ is a smooth morphism (more generally $i$ need not be a closed immersion and can be any quasi-finite affine morphism). Note that, since the morphism $i'^*\Omega^1_f\ra\Omega^1_{f'}$ is an isomorphism, $f$ and $f'$ have the same relative dimension $d$.  
Let $Ex_*^*:f^*i_*\ra i'_*f'^*$ and ${}^\sP Ex_*^*:f^*_\sP i^\sP_*\ra i'^\sP_*f'^*_\sP$ the exchange $2$-isomorphisms in $\DA(-)$ and $\Db(\sP(-))$.
We have to construct a $2$-isomorphism
\[\xymatrix@C=2.5cm{{\Db(\sM(X))}\rtwocell^{f^*_\sM i^\sM_*}_{i'^\sM_*f'^*_\sM}{\kern2pc{}^\sM Ex^*_*} & {\Db(\sM(Y'))} }\] 
which is compatible with ${}^\sP Ex_*^*$ via the $2$-isomorphisms $\gamma_g^\sM,\gamma^\sM_{g'}$ and $\theta^\sM_f,\theta^\sM_{f'}$.
This amounts to constructing a $2$-isomorphism
\[\xymatrix@C=2.5cm{{\sM(X)}\rtwocell^{f^*_\sM[d] i^\sM_*}_{i'^\sM_*f'^*_\sM[d]}{\kern2pc{}^\sM Ex^*_*[d]} & {\sM(Y')} }\] 
such that
\[\xymatrix@R=.2cm@C=.2cm{{\sM(X)}\ar@{=}[ddd]\ar[rrrr]^-{\rat^\sM_X}\ar[rd]^-{i^\sM_*} &{} & {}\xtwocell[1,0]{}\omit{_(\gamma^\sM_i)^{-1}} &{} &{\sP(X)}\ar[rd]^-{i^\sP_*}\ar@{:}[ddd] &{} &{}\\
{} &{\sM(Y)}\ar[rd]^-{f^*_\sM[d]}\ar[rrrr]^-{\rat^\sM_X} &{} &{}& {}\xtwocell[1,0]{}\omit{_\theta^\sM_f} &{\sP(Y)}\ar[rd]^-{f^*_\sP[d]} &{}\\
{}\xtwocell[0,2]{}\omit{<2>_{\kern-4pc {}^\sM Ex^*_*[d]}} &{} &{\sM(Y')}\ar[rrrr]^-{\rat^\sM_{X'}}\ar@{=}[ddd] &{} & {} \xtwocell[0,2]{}\omit{<2>_{\kern2pc {}^\sP Ex^*_*[d]}}&{} &{\sP(Y')}\ar@{=}[ddd]\\
{\sM(X)}\ar@{.>}[rrrr]^-{\rat^\sM_X}\ar[rd]^{f'^*_\sM[d]}  &{} &{}\xtwocell[1,0]{}\omit{_\theta^\sM_{f'}}  &{}&{\sP(X)}\ar@{.>}[rd]^-{f'^*_\sP[d]} &{} &{}\\
{} &{\sM(Y)}\ar@{.>}[rrrr]^-{\rat^\sM_{Y}}\ar[rd]^{i'^\sM_*} &{} &{} &{}\xtwocell[1,0]{}\omit{_{(\gamma^\sM_{i'})^{-1}}} &{\sP(Y)}\ar@{.>}[rd]^-{i'^\sP_*} &{}\\
{} &{} &{\sM(Y')}\ar[rrrr]^(.6){\rat^\sM_{Y'}} &{} & {}&{} &{\sP(Y')}
}\]
is commutative. Note that such a $2$-isomorphism is necessarily unique since the functors $\rat^\sM_S$, for $S$ a $k$-variety, are faithful.  For the same reason they will also be compatible with the horizontal and vertical compositions of mixed squares (see~\cite[D\'efinition 1.21]{AyoubI}). The proposition follows from property \textbf{P2} applied to the commutative diagram
\[\xymatrix@R=.2cm@C=.2cm{{\DA_\ct(X)}\ar@{=}[ddd]\ar[rrrr]^-{\pH^0_\sP}\ar[rd]^-{i_*} &{} & {}\xtwocell[1,0]{}\omit{_{(\gamma_i)^{-1}}} &{} &{\sP(X)}\ar[rd]^-{i^\sP_*}\ar@{:}[ddd] &{} &{}\\
{} &{\DA_\ct(Y)}\ar[rd]^-{f^*[d]}\ar[rrrr]^-{\pH^0_\sP} &{} &{}& {}\xtwocell[1,0]{}\omit{_\theta_f} &{\sP(Y)}\ar[rd]^-{f^*_\sP[d]} &{}\\
{}\xtwocell[0,2]{}\omit{<2>_{\kern-4pc Ex^*_*[d]}} &{} &{\DA_\ct(Y')}\ar[rrrr]^-{\pH^0_\sP}\ar@{=}[ddd] &{} & {} \xtwocell[0,2]{}\omit{<2>_{\kern2pc {}^\sP Ex^*_*[d]}}&{} &{\sP(Y')}\ar@{=}[ddd]\\
{\DA_\ct(X)}\ar@{.>}[rrrr]^-{\pH^0_\sP}\ar[rd]^{f'^*[d]}  &{} &{}\xtwocell[1,0]{}\omit{_\theta_{f'}}  &{}&{\sP(X)}\ar@{.>}[rd]^-{f'^*_\sP[d]} &{} &{}\\
{} &{\DA_\ct(X')}\ar@{.>}[rrrr]^-{\pH^0_\sP}\ar[rd]^{i'_*} &{} &{} &{}\xtwocell[1,0]{}\omit{_{(\gamma_{i'})^{-1}}} &{\sP(X')}\ar@{.>}[rd]^-{i'^\sP_*} &{}\\
{} &{} &{\DA_\ct(Y')}\ar[rrrr]^(.6){\pH^0_\sP} &{} & {}&{} &{\sP(Y').}
}\]
The commutativity of this diagram follows from the compatibility of the Betti realization with the exchange structures.
\end{proof}

\begin{rema}\label{rema:compthetaexchange}
The application of property \textbf{P2} ensures that the two exchange structures ${}^\sM Ex^*_*$ and ${}^\sP Ex^*_*$ are compatible with the canonical $2$-isomorphisms $\theta^\sM$. That is, the diagram
\[\xymatrix{{f^*_\sP\rat^\sM_Y i^\sM_*}\ar[r]^-{\gamma^\sM_i}\ar[d]^-{\theta^\sM_f} &{f^*_\sP i^\sP_*\rat^\sM_X}\ar[r]^-{{}^\sP Ex^*_*} & {i'^\sP_*f'^*_\sP\rat^\sM_X}\ar[d]^-{\theta^\sM_{f'}}\\
{\rat^\sM_{Y'}f^*_\sM i^\sM_*}\ar[r]^-{{}^\sM Ex^*_*} & {\rat^\sM_{Y'}i'^\sM_*f'^*_\sM } \ar[r]^-{\gamma^\sM_{i'}} & {i'^\sP_*\rat^\sM_{X'}f'^*_\sM}}\]
is commutative. This follows from the faithfulness of the functors $\rat^\sM$ after applying the shift functor $(-)[d]$.
\end{rema}

\subsection{Adjunction}

Let $f:X\ra Y$ be an affine and \'etale morphism. In that case the exact functors $f_*:\sP(X)\ra\sP(Y) $ and $f_!:\sP(X)\ra\sP(Y) $ are respectively right and left adjoint to the exact functor $f^*_\sP:\sP(Y)\ra\sP(X)$. We can use the property {\textbf{P2}} to lift these adjunctions to the functors $f^\sM_!,f^*_\sM, f^\sM_*$.

\begin{prop}\label{prop:AdjEtaleAffine}
Let $f:X\ra Y$ be an affine and \'etale morphism.
\begin{enumerate}
\item There exist unique natural transformations $\Id\ra f_*^\sM f^*_\sM $ and $f^*_\sM f^\sM_*\ra \Id $ such that the squares
\[\xymatrix{{\rat^\sM_Y}\ar[r]\ar[d] &{\rat^\sM_Yf_*^\sM f^*_\sM}\ar[d]\\
{f_*^\sP f^*_\sP\rat^\sM_Y}\ar[r] &{f^\sP_*\rat^\sM_X f^*_\sM}}
\qquad
\xymatrix{{\rat^\sM_X f^*_\sM f^\sM_*}\ar[r] &{\rat^\sM_X}\\
{f^*_\sP\rat^\sM_Y f^\sM_*}\ar[u]\ar[r] & {f^*_\sP f^\sP_*\rat^\sM_X}\ar[u]}\]
are commutative. With these natural transformations, the functors $(f^*_\sM,f_*^\sM)$ form a pair of adjoint functors.
\item There exist unique natural transformations $\Id\ra f^*_\sM f_!^\sM $ and $f_!^\sM f^*_\sM\ra \Id $ such that the squares
\[\xymatrix{{\rat^\sM_X}\ar[r]\ar[d] &{\rat^\sM_Xf^*_\sM f^\sM_!}\\
{f^*_\sP f_!^\sP\rat^\sM_X}\ar[r] &{f^*_\sP\rat^\sM_Y f_!^\sM}\ar[u]}
\qquad
\xymatrix{{\rat^\sM_Y f_!^\sM f^*_\sM}\ar[r] &{\rat^\sM_Y}\\
{f_!^\sP\rat^\sM_X f^*_\sM}\ar[u] & {f_!^\sP f^*_\sP\rat^\sM_X}\ar[u]\ar[l]}\]
are commutative. With these natural transformations, the functors $(f_!^\sM,f^*_\sM)$ form a pair of adjoint functors.\end{enumerate}
\end{prop}
\begin{proof}
As for \ref{prop:ExStructure}, the proof is a simple application of property \textbf{P2}. The details are left to the reader.
\end{proof}

\subsection{Duality}

The result in this subsection will be used in the proof of \ref{prop:leftadjoint}. Let $\bbD^\sP_X$ be the duality functor for perverse sheaves and $\varepsilon^\sP_X:\Id\ra\bbD^\sP_X\circ\bbD^\sP_X$ be the canonical  $2$-isomorphism. Recall that, given a smooth morphism  $f:X\ra Y$ of relative dimension $d$, there is a canonical $2$-isomorphism 
\[\varepsilon^\sP_f:\bbD^\sP_X \circ f^*_\sP(-)(d)[d]\ra f^*_\sP(-)[d]\circ \bbD^\sP_Y.\]

\begin{prop}\label{prop:duality}
Let $X,Y$ be $k$-varieties and $f:X\ra Y$ be a smooth morphism of relative dimension $d$.
\begin{enumerate}
\item  There exist a contravariant exact functor $\bbD^\sM_X:\sM(X)\ra\sM(X)$, a $2$-isomorphism $\nu^\sM_X:\bbD^\sP_X\circ \rat^\sM_X\ra\rat^\sM_X\circ\bbD^\sM_X$ and a $2$-isomorphism $\varepsilon^\sM_X:\Id\ra\bbD^\sM_X\circ\bbD^\sM_X $ such that the diagram
\[\xymatrix{{} &{\bbD^\sP_X\circ\bbD^\sP_X\circ\rat^\sM_X}\ar[d]^-{\nu_X^\sM}\\
{\rat^\sM_X}\ar[ru]^-{\varepsilon^\sP_X}\ar[rd]_-{\varepsilon^\sM_X} &{\bbD^\sP_X\circ \rat^\sM_X\circ\bbD^\sM_X}\ar[d]^-{\nu_X^\sM}\\
{} &{\rat^\sM_X\circ\bbD^\sM_X\circ\bbD^\sM_X}}\]
is commutative.
\item There exists a $2$-isomorphism 
\[\varepsilon^\sM_f:\bbD^\sM_X \circ f^*_\sM(-)(d)[d]\ra f^*_\sM(-)[d]\circ \bbD^\sM_Y\]
such that the diagram
\[\xymatrix{{\bbD^\sP_X \circ f^*_\sP(-)(d)[d]\circ \rat^\sM_Y}\ar[r]^-{\varepsilon^\sP_f} &{ f^*_\sP(-)[d]\circ \bbD^\sP_Y\circ\rat^\sM_Y}\ar[d]^-{\nu^\sM_Y}\\
{\bbD^\sP_X \circ\rat^\sM_X\circ f^*_\sM(-)(d)[d]}\ar[d]^-{\nu^\sM_X}\ar[u]_-{\theta^\sM_f} &{ f^*_\sP(-)[d]\circ\rat^\sM_Y \circ\bbD^\sM_Y}\ar[d]^-{\theta^\sM_f}\\
{\rat^\sM_X\circ\bbD^\sM_X \circ f^*_\sM(-)(d)[d]}\ar[r]^-{\varepsilon^\sM_f} &{\rat^\sM_X\circ f^*_\sM(-)[d]\circ \bbD^\sM_Y}}\]
is commutative.
\end{enumerate}
\end{prop}

\begin{proof}
Again, the proof is a simple application of property \textbf{P2}, once we
know the existence and properties of the Verdier duality functor
on motives. We give references for these properties and
leave the rest of the details to the reader.

By~\cite[Th\'eor\`eme 2.3.75]{AyoubI} and~\cite[\S4.5]{AyoubII} (see also~\cite[Th\'eor\`eme~3 and Th\'eor\`eme~7 in the introduction]{CD}), 
the categories $\DA_\ct(X)$ are symmetric mono\"idal closed and we have
Verdier duality functors $\bbD_X$ such that, for
$f:X\ra Y$ a morphism of quasi-projective $k$-varieties, there is a
canonical isomorphism $f^*\circ\bbD_Y\simeq\bbD_X\circ f^!$. If $f$ is smooth of relative dimension
$d$, the functor $f^!$ is defined in~\cite[\S1.5.3.1]{AyoubI} as the composition $\mathsf{Th}(\Omega_f)\circ f^*$, where $\mathsf{Th}(\Omega_f)$ is the Thom equivalence associated with the locally free $\Osheaf_X$-module $\Omega_{f}$. As $\Omega_f$ has rank $d$, we get an isomorphism $f^!\simeq f^*(d)[2d]$ by~\cite[corollaire 2.14]{AyoubEtale} (see also property 4 of~\cite[A.5.1]{CD}). Hence, we get an isomorphism
$f^*[d]\circ\bbD_Y\simeq\bbD_X\circ f^*(d)[d]$. 
Moreover, the Betti realization functors
$\Bti_X$ are symmetric mono\"idal unital (see
\cite[Remarque~3.3]{AyoubBetti}), and they commute with
internal $\Hom$s on constructible objects by
\cite[Th\'eor\`eme~3.19]{AyoubBetti}; so it commutes with the
Verdier duality functor on constructible objects, 
as that functor is constructed using the
four operations and the internal $\Hom$ (see for example
\cite[A.5.2]{CD}). The last crucial observation is that Verdier duality
on $\Dbc(X,\bbQ)$ restricts to an exact contravariant functor on the
subcategory of perverse sheaves (see for example the beginning
of Section~4 of~\cite{BBD}).
\end{proof}

\subsection{Perverse motives as a stack}

Let $S$ be a quasi-projective $k$-variety. Let us denote by $\mathsf{AffEt}/S$ the category of affine \'etale schemes over $S$ endowed with the \'etale topology. As in \cite[Tag 02XU]{stacks-project}, the $2$-functor 
\begin{align*}
\mathsf{AffEt}/X&\ra\mathfrak{Ab}\\
U& \mapsto\sM(U)\\
u & \mapsto u^*_\sM
\end{align*}
can be turned into a fibered category $\sM\ra \mathsf{AffEt}/S$ such that the fiber over an object $U$ of $\mathsf{AffEt}/X$ is the category $\sM(U)$.

\begin{prop}\label{prop:stack}
The fibered category $\sM\ra \mathsf{AffEt}/S$ is a stack for the \'etale topology. 
\end{prop}

\begin{proof}
Let $U$ be a $k$-variety, $I$ be a finite set and $\mathscr U=(u_i:U_i\ra U)_{i\in I}$ be a covering of $U$ by affine and \'etale morphisms. If $J\subseteq I$ is a nonempty subset of $I$, we denote by $U_J$ the fiber product of the $U_j$, $j\in J$, over $U$ and by $u_J:U_J\ra U$ the induced morphism. Given an object $A\in\sM(U)$, and $k\in\bbZ$, we set
\[C^k(A,\mathscr U):=
\begin{cases}
0 & $\textrm{if $k<0$;}$\\
\bigoplus_{J\subseteq I, |J|=k+1} (u_J)^\sM_*(u_J)_\sM^*A & \textrm{if $k\geqslant 0$.}
 \end{cases}\]
 We make $C^\bullet(A,\mathscr U)$ into a complex using the alternating sum of the maps obtained from the unit of the adjunction in \ref{prop:AdjEtaleAffine}. The unit of this adjunction also provides a canonical morphism $A\ra C^\bullet(A,\mathscr U)$ in $\Cb(\sM(U))$. This morphism induces a quasi-isomorphism on the underlying complex of perverse sheaves and so is a quasi-isomorphism itself since the forgetful functor to the derived category of perverse sheaves is conservative.
 
 By \cite[Tag 0268]{stacks-project}, to prove the proposition we have to show the following:
 \begin{enumerate}
 \item if $U$ is an object in $\mathsf{AffEt}/S$ and $A,B$ are objects in $\sM(U)$, then the presheaf $(V\xra{v} U)\mapsto \Hom_{\sM(V)}(v_\sM^*A,v^*_\sM B)$  on $\mathsf{AffEt}/U$ is a sheaf for the \'etale topology;
 \item for any covering $\mathscr U=(u_i:U_i\ra U)_{i\in I}$ of the site $\mathsf{AffEt}/S$, any descent datum is effective.
 \end{enumerate}
We already now that the similar assertions are true for perverse sheaves by \cite[Proposition 3.2.2, Th\'eor\`eme 3.2.4]{BBD}.  Let $\mathscr U=(u_i:U_i\ra U)_{i\in I}$ be a covering in the site $\mathsf{AffEt}/S$. Given $i,j\in I$, we denote by $u_{ij}:U_{ij}:=U_i\times_UU_j\ra U$ the fiber product  and by $p_{ij}:U_{ij}\ra U_i$, $p_{ji}:U_{ij}\ra U_j$ the projections.

Let us first prove (1). Let $A,B$ be objects in $\sM(U)$ and $K,L$ be their underlying perverse sheaves. Consider the canonical commutative diagram
\[\xymatrix@C=.5cm{{\Hom(A,B)}\ar[r]\ar[d] &{\prod_{i\in I}\Hom((u_i)^*_\sM A,(u_i)^*_\sM B)}\ar@<1ex>[r]\ar@<-1ex>[r]\ar[d] &{\prod_{i,j\in I}\Hom((u_{ij})^*_\sM A,(u_{ij})^*_\sM B)}\ar[d]\\
{\Hom(K,L)}\ar[r] &{\prod_{i\in I}\Hom((u_i)^*_\sP K,(u_i)^*_\sP L)}\ar@<1ex>[r]\ar@<-1ex>[r] &{\prod_{i,j\in I}\Hom((u_{ij})^*_\sP K,(u_{ij})^*_\sP L).}}\]
The lower row is exact and the vertical arrows are injective. We only have to check that the upper row is exact at the middle term. Let $c$ be an element in $\prod_{i\in I}\Hom((u_i)^*_\sM A,(u_i)^*_\sM B)$ which belongs to the equalizer of the two maps on the right-hand side. Then, it defines (by adjunction) a morphism $c_0$ and a morphism $c_1$ such that the square
\begin{equation}\label{eq:squarestack}
\xymatrix{{C^0(A,\mathscr U)}\ar[r]^-{d^0}\ar[d]^{c_0} &{C^1(A,\mathscr  U)}\ar[d]^-{c_1}\\
{C^0(B,\mathscr U)}\ar[r]^-{d^0} &{C^1(B,\mathscr U)}}
\end{equation}
 is commutative. Since $A\ra C^\bullet(A,\mathscr U)$ and $B\ra C^{\bullet}(B,\mathscr U)$ are quasi-isomorphisms, $A$ is the kernel of the upper map in \eqref{eq:squarestack} and $B$ is the kernel of the lower map. Hence, $c_0$ and $c_1$ induce a morphism $A\ra B$ in $\sM(U)$ which maps to $c$.
 
 Now we prove (2). Consider a descent datum. In other words, consider, for every $i\in I$, an object $A_i$ in $\sM(U_i)$ and, for every $i,j\in I$, an isomorphism $\phi_{ij}:(p_{ij})^*_\sM A_i\ra (p_{ji})^*_\sM A_j$ in $\sM(U_{ij})$ satisfying  the usual cocycle condition. Let $A$ the kernel of the map
 \[\bigoplus_{i\in I}(u_i)^\sM_*A_i\ra\bigoplus_{i,j\in I}(u_i)^\sM_*(p_{ij})_*^\sM(p_{ij})^*_\sM A_i=\bigoplus_{i,j\in I}(u_{ij})^\sM_*(p_{ij})^*_\sM A_i\]
 given on $(u_i)^\sM_*A_i$ by the difference of the maps obtained by composing the morphism induced by adjunction
 \[(u_i)^\sM_*A_i\ra (u_i)^\sM_*(p_{ij})_*^\sM(p_{ij})^*_\sM A_i \]
 with either the identity or the isomorphism $\phi_{ij}$. Using the fact that descent data on perverse sheaves are effective, it is easy to see that $A$ makes the given descent datum effective.
 \end{proof}

\subsection{A simpler generating quiver}

Let $X$ be a $k$-variety. Consider the quiver $\Pairs^\eff_X$ defined as follows. A vertex in $\Pairs^\eff_X$ is a triple $(a:Y\ra X,Z,n)$ where $a:Y\ra X$ is morphism of $k$-varieties, $Z$ is a closed subscheme of $Y$ and $n\in\bbZ$ is an integer. 
\begin{itemize}
\item Let $(Y_1,Z_1,i)$ and $(Y_2,Z_2,i)$ be vertices in $\Pairs_X^\eff$. Then, every morphism of $X$-schemes $f:Y_1\ra Y_2$ such that $f(Z_1)\subseteq Z_2$ defines an edge
\begin{equation}\label{edge1}
f:(Y_1,Z_1,i)\ra (Y_2,Z_2,i).
\end{equation}
\item For every vertex $(a:Y\ra X,Z,i)$ in $\Pairs_X^\eff$  and every closed subscheme $W\subseteq Z$, we have an edge
\begin{equation}\label{edge2}
\partial:(a:Y\ra X,Z,i)\ra (az:Z\ra X,W,i-1)
\end{equation}
where $z:Z\hookrightarrow Y$ is the closed immersion.
\end{itemize}
The quiver $\Pairs^\eff_X$ admits a natural representation in $\Dbc(X,\bbQ)$.  If $c=(a:Y\ra X,Z,i)$ is a vertex in the quiver $\Pairs^\eff_X$ and $u:U\hookrightarrow Y$ is the inclusion of the complement of $Z$ in $Y$, then  we set 
\[B(c):=a^\sP_!u^\sP_*K_U[-i]\]
where $K_U$ is the dualizing complex of $U$.

\begin{rema}
There is a difference between the representation $\pH^0\circ \overline{B}$ used here and the representation used in \cite[7.2--7.4]{IvorraPNM} (see \cite[Remark 7.8]{IvorraPNM}). In loc.cit. the relative dualizing complex $u^!_\sP a^!_\sP\bbQ_X$ is used instead of the absolute dualizing complex $K_U$. If $X$ is smooth, then the two different choices lead to equivalent categories.
\end{rema}

On vertices the representation $B$ is defined as follows.  Let $c_1:=(a_1:Y_1\ra X,Z_1,i)$, $c_2:=(a_2:Y_2\ra X,Z_2,i)$ be vertices in $\Pairs^\eff_X$ and $f:c_1\ra c_2$ be an edge of type \eqref{edge1}. The morphism $f$ maps $Z_1$ to $Z_2$ and therefore $U:=f^{-1}(U_2)$ is contained in $U_1$. Let $u:U\hookrightarrow U_1$ be the open immersion. Then, we have a morphism
\[f^\sP_!u^\sP_{1*}K_{U_1}\xra{\textrm{adj.}}f^\sP_!(u_1u)^\sP_*K_{U}\longrightarrow u^\sP_{2*}f^\sP_!K_U=u^\sP_{2*}f^\sP_!f^!_\sP K_{U_2}\xra{\textrm{adj.}} u^\sP_{2*}K_{U_2}\]
where the arrow in the middle is given by the exchange morphism. By taking the image of this morphism under $a_{2!}[-i]$, we get a morphism
\[B(f):B(c_1):=a^\sP_{1!}u^\sP_{1*}K_{U_1}[-i]\ra B(c_2):=a^\sP_{2!}u^\sP_{2*}K_{U_2}[-i].\]
Let $c=(Y\xra{a}X,Z,i)$ be a vertex in $\Pairs^\eff_X$, and $W\subseteq Z$ be a closed subset. 
Consider the commutative diagram
\[\xymatrix{{U:=Y\setminus Z}\ar[r]^-j\ar@/^2em/[rr]^-{u} &  {Y\setminus W}\ar[r]^-{v_Y} & {Y}\ar[r]^a & {X}\\
{} & {V:=Z\setminus W}\ar@{}[ru]|\square\ar[r]^-v\ar[u]^-{z_V} & {Z}\ar[u]^z\ar[ru]_b}\]
where $v,v_Y,j$ are the open immersions, $z$ the closed immersion and $a,b$ the structural morphisms. The localization triangle
\[(z_V)^\sP_!(z_V)_\sP^!\ra \Id \ra j^\sP_*j^*_\sP\xrightarrow{+1},\]
applied to $K_{Y\setminus W}$,  provides a morphism 
\[j^\sP_*K_U\ra (z_V)^\sP_!K_V[1].\]
As $z$ and $z_V$ are closed immersions, applying $(v_Y)_*$, yields a morphism
\[u^\sP_*K_U\ra z^\sP_!v^\sP_*K_V[1].\]
Applying $a_![-i]$, we obtain a morphism
\[B(\partial):B(c):=a^\sP_!u^\sP_*K_U[-i]\ra B(az:Z\ra X,Z,i-1):=b^\sP_!v^\sP_*K_V[1-i].\]

The category of perverse Nori motives considered in \cite{IvorraPNM} is defined as follows.
\begin{defi}
Let $X$ be a $k$-variety. The category of effective perverse Nori motives is the abelian category
\[\sN^\eff(X):=\Ab^\qv(\Pairs^\eff_X,\pH^0\circ B).\]
\end{defi}

Recall that the category $\sM(X)$ can also be obtained by considering $\DA_\ct(X)$ simply as a quiver, that is it is canonically equivalent to the abelian category $\Ab^\qv(\DA_\ct(X),\pH^0\circ\Bti^*_X)$ (see \ref{LemmCompB}). The Grothendieck six operations formalism constructed in \cite{AyoubI,AyoubII} and its compatibility  with its topological counterpart on the triangulated categories $\Dbc(X,\bbQ)$ shown in \cite{AyoubBetti}, imply that the quiver representation $B$
can be lifted via the realization functor $\Bti^*_X$ to a quiver representation 
\[\overline{B}:\Pairs^\eff_X\ra\DA_\ct(X).\] 
 In particular, since the diagram 
\[\xymatrix{{\DA_\ct(X)}\ar[r]^-{\Bti^*_X} &{\Dbc(X,\bbQ)}\\
{\Pairs^\eff_X,}\ar[u]^-{\overline{B}}\ar[ru]_-{B} & {}}\]
is commutative (up to natural isomorphisms),
there exists a canonical faithful exact functor 
 \begin{equation}\label{FoncEHMM}
 \sN^\eff(X)\ra\sM(X).
 \end{equation}
 
 Let us explain now how Tate twists can be defined in the categories $\sN^\eff(X)$ and $\mathscr M(X)$. In the category $\DA_\ct(X)$, the Tate twist $(-)(1)$ is defined to be the endofunctor $\mathsf{Th}(\Osheaf_X)(-)[-2]$ where $\mathsf{Th}(\Osheaf_X)$ is the Thom equivalence associated with the trivial locally free sheaf $\Osheaf_X$ (see \cite[\S 1.5.3]{AyoubI}). This construction, being compatible with the usual Tate twist via the Betti realization, induces an exact functor $(-)(1)$ on the category $\mathscr M(X)$. Note that this functor is an equivalence by construction. 
 
 In the category $\sN^\eff(X)$ Tate twists can be defined using the following observation: if $S$ is a $k$-variety, $q:\bbG_{m,S}\ra S$ is the structural morphism and $v:V\hookrightarrow \bbG_{m,S}$ is the complement of the unit section, then  $q_!v_*v^*q^!K=K(1)[1]$ for every $K\in\Dbc(S,\bbQ)$. In particular, if $Q:\Pairs^\eff_X\ra\Pairs^\eff_X$ is the morphism of quivers which maps 
$(Y, Z,n)$ to $(\bbG_{m,Y}, \bbG_{m,Z}\cup Y, n + 1)$ (here $Y$ is embedded in $\bbG_{m,Y}$ via the unit section), then one has a natural isomorphism between $B(Q(Y,Z,n))$ and $B(Y,Z,n)(1)$. As a consequence, the Tate twist on the category of effective perverse Nori motives can be defined as the exact functor induced by the morphism of quivers $Q$ (and the usual Tate twist). 
 
 This last construction does not yield an equivalence and one defines the category $\sN(X)$ to be the category obtained from $\sN^\eff(X)$ by inverting the Tate twists (see \cite[7.6]{IvorraPNM} for details). By construction, the category of Nori motives $\HM(k)$ of Nori's work~\cite{Nori} coincides with $\sN(k)$.

 \begin{lemm}
 The functor \eqref{FoncEHMM} extends to a faithful exact 
 \begin{equation}\label{FoncHMM}\sN(X)\ra\sM(X).
 \end{equation}
 \end{lemm}
 \begin{proof}
  To prove the lemma it is enough to observe that
there is a natural isomorphism in $\DA_\ct(X)$  between $\overline{B}(Q(Y,Z,n))$ and $\overline{B}(Y,Z,n)(1)$.
 \end{proof}

\begin{prop}\label{prop:CompNori}
The category $\sM(k)$ is canonically equivalent to the abelian category of Nori motives $\HM(k)$. More precisely the functor \eqref{FoncHMM} is an equivalence when $X=\Spec(k)$
\end{prop}

\begin{proof} (See also Proposition~4.12 of the paper~\cite{BVHP} by Barbieri-Viale, Huber and Prest.)
Consider the triangulated functor $\mathcal R_{\mathrm N,s}:\DA_\ct(X)\ra \Db(\HM(k))$ constructed in Proposition~7.12 of the paper~\cite{MR3649230} of Choudhury and Gallauer. Up to a natural isomorphism, the diagram
\[\xymatrix{{\DA_\ct(k)}\ar@/^2em/[rr]^{\Bti_k^*}\ar[r]_{\mathcal{R}_{\mathrm{N},s}} & {\Db(\HM(k))}\ar[r]_-{\textrm{forgetful}} & {\Db(\bbQ)}}\]
is commutative. In particular, it provides a factorization of the cohomological functor $H^0\circ \Bti_k^*$ 
\[\DA_\ct(k)\xra{H^0\circ \mathcal R_{\mathrm{N},s}}\HM(k)\xra{\textrm{forgetful}}\vec(\bbQ).\]
This implies the existence of a canonical faithful exact functor $\sM(k)\ra\HM(k)$ such that
\[\xymatrix{{\DA_\ct(k)}\ar[r]_-{\pH^0_\sM}\ar@/^2em/[rr]^-{H^0\circ \mathcal R_{\mathrm{N},s}} &{\sM(k)}\ar[r] & {\HM(k)}}\]
is commutative up to a natural isomorphism. Using the universal properties, it is easy to see that it is a quasi-inverse to \eqref{FoncHMM}.
\end{proof}

The following conjecture seems reasonable and reachable via our current technology. 

\begin{conj}\label{conj:compBetti}
Let $X$ be a smooth $k$-variety. Let $\sN(X)$ be the category of perverse motives constructed in \cite{IvorraPNM} and
\[\mathsf{RL}^\sN_X:\DA_\ct(X)\ra\Db(\sN(X))\]
be the triangulated functor constructed in \cite{IvorraHodge}. Then, the Betti realization $\Bti_X^*$ is isomorphic to the composition
\[\DA_\ct(X)\ra\Db(\sN(X))\xra{\mathrm{forgetful}}\Db(\Perv(X))\xra{\real}\Dbc(X,\bbQ).\]
\end{conj}

If \ref{conj:compBetti} holds then the same proof as the one of \ref{prop:CompNori} implies the following

\begin{conj}
Let $X$ be a smooth $k$-variety. Then, the functor \eqref{FoncHMM} is an equivalence.
\end{conj} 

\section{Unipotent nearby and vanishing cycles}\label{sec:nearbycycles}

In~\cite{MR923134}, Be{\u\i}linson has given an alternate construction of unipotent vanishing cycles functors for perverse sheaves and has used it to explain a gluing procedure for perverse sheaves (see \cite[Proposition 3.1]{MR923134}). In this section, our main goal is to obtain similar results for perverse Nori motives. Later on, the vanishing cycles functors for perverse Nori motives will play a crucial role in the construction of the inverse image functor (see \ref{sec:pullback}).

Given the way the abelian categories of perverse Nori motives are constructed from the triangulated categories of \'etale motives, our first step is to carry out Be{\u\i}linson's constructions 
for perverse sheaves within the categories of \'etale motives or analytic motives (the latter categories being equivalent to the classical unbounded derived categories of sheaves of $\bbQ$-vector spaces on the associated analytic spaces). This is done in \ref{subsec:Unmotshfun} and \ref{subsec:CompBetti}.  Our starting point is the logarithmic specialization system constructed by Ayoub in~\cite{AyoubII}. However, by working in triangulated categories instead of abelian categories as Be{\u\i}linson did, one has to face the classical functoriality issues, one of the major drawback of triangulated categories. To avoid these problems and ensure that all our constructions are functorial we will rely heavily on the fact that the triangulated categories of motives underlie a triangulated derivator.

Only then, using the compatibility with the Betti realization, will we be able to obtain in \ref{subsec:AppPerverseMotive} the desired functors for perverse Nori motives.

\subsection{Reminder on derivators}\label{subsec:derivator}

Let us recall some features of triangulated (a.k.a. stable) derivators $\bbD$ needed in the construction of the motivic unipotent vanishing cycles functor and the related exact triangles. For the general theory, originally introduced by Grothendieck~\cite{GrothDer}, we refer to the works~\cite{AyoubI,AyoubII} by Ayoub, \cite{MR2382732} by Cisinki--Neeman, \cite{GrothBook, GrothDeriv} by Groth and~\cite{MR2951712}
by Maltsiniotis.

We will assume that our derivator $\bbD$ is defined over all small categories. In our applications, the derivators considered will be of the form $\bbD:=\DA(S,-)$ for some $k$-variety $S$. Given a functor $\rho:A\ra B$, we denote by
\[\rho^*:\bbD(B)\ra\bbD(A),\quad \rho_*:\bbD(A)\ra\bbD(B),\quad \rho_\sharp:\bbD(A)\ra\bbD(B)\]
the structural functor and its right and left adjoint. Note that in the literature on derivators, the notation $\rho_!$ is used instead of $\rho_\sharp$. We follow here the notation used in \cite{AyoubI,AyoubII}.

{\itshape Notation:}  We let $\mathbf e$ be the punctual category reduced to one object and one morphism. Given a small category $A$, we denote by $p_A:A\ra {\mathbf e}$ the projection functor and, if $a$ is an object in $A$, we denote by $a:{\mathbf e}\ra A$ the functor that maps the unique object of $\mathbf{e}$ to $a$. Given $n\in\bbN$, we let $\underline{\mathbf n}$ be the category
\[n\leftarrow \cdots\leftarrow 1\leftarrow 0.\]
If one thinks of functors in $\mathsf{Hom}(A^\op,\bbD(\mathbf{e}))$ as diagrams, then an object in $\bbD(A)$ can be thought as a ``coherent diagram''. Indeed, every object $M$ in $\bbD(A)$ has an underlying diagram called its $A$-skeleton and defined to be the functor $A^{\op}\ra \mathbb{D}({\mathbf e})$ which maps an object $a$ in $A$ to the object $a^*M$ of $\bbD({\mathbf e})$. This construction gives the $A$-skeleton functor
\[\bbD(A)\ra\mathsf{Hom}(A^\op,\bbD(\mathbf{e}))\]
which is not an equivalence in general (coherent diagrams are richer than simple diagrams). We say that $M\in\bbD(A)$ is a coherent lifting of a given diagram of shape $A$ if its $A$-skeleton is isomorphic to the given diagram.

We will not give here the definition of a stable derivator (see e.g. \cite[Definition 2.1.34]{AyoubI} or \cite[Section 1 \& Definition 4.1]{GrothDeriv}), but instead recall a few properties which will be constantly used.
\par\smallskip
\textbf{(1)} Let $\rho:A\ra B$ be a functor and $b$ be an object in $B$. Denote by $j_{A/b}:A/b\ra A$ and $j_{b\backslash A}: b\backslash A\ra A$ be the canonical functors where $A/b $ and $b\backslash A$ are respectively the slice and coslice categories. The exchange $2$-morphisms (given by adjunction)
\[b^*\rho_*\ra (p_{A/b})_*j^*_{A/b}\quad;\quad (p_{b\backslash A})_\sharp j^*_{b\backslash A}\ra b^*\rho_\sharp\]
are invertible (see \cite[D\'efinition 2.1.34]{AyoubI} or the base change axiom {\textbf{Der 3}} of \cite[Definition 1.11]{MR2382732}).
\par\smallskip
\textbf{(2)} If a small category $A$ admits an initial object $o$ (resp. a final object $o$), then the $2$-morphism $o^*\ra (p_A)_\sharp$ (resp. the $2$-morphism $(p_A)_*\ra o^*$) is invertible too (see \cite[Corollaire 2.1.40]{AyoubI}).
\par\smallskip
 \textbf{(3)} Let $A$ and $B$ be small categories. Given an object $a\in A$ we denote by $a:B\ra A\times B$ the functor which maps $b\in B$ to the pair $(a,b)$. The $A$-skeleton of an object $M$ in $\bbD(A\times B)$ is defined to be the functor $A^{\op}\ra \mathbb{D}(B)$ which maps an object $a$ in $A$ to the object $a^*M$ of $\bbD(B)$. This construction gives the $A$-skeleton functor
\[\bbD(A\times B)\ra\mathsf{Hom}(A^\op,\bbD(B)).\]
This functor is conservative. Moreover if $A=\underline{\mathbf{1}}$, it is full and essentially surjective. (See the axioms {\textbf{Der 2}} and {\textbf{Der 5}} of \cite[Definition 1.11]{MR2382732}.)
\par\smallskip
 We denote by $\carre=\underline{\mathbf 1}\times \underline{\mathbf 1}$ the category 
\begin{equation}\label{eq:catsquare}
\xymatrix{{(1,1)} & {(0,1)}\ar[l]\\
{(1,0)}\ar[u] & {(0,0).}\ar[u]\ar[l]}
\end{equation}
We denote by $\smallucarre$ the full subcategory of $\carre$ that does not contain the object $(0,0)$ and by $i_{\smallucarre}:\smallucarre\ra\carre $ the inclusion functor. We denote by $(-,1):\underline{\mathbf{1}}\ra\smallucarre$ the fully faithful functor which maps $0$ to $(0,1)$ and $1$ to $(1,1)$. 
Similarly we denote by $\smalllrcarre$ the full subcategory of $\carre$ that does not contain the object $(1,1)$ and $i_{\smalllrcarre}:\smalllrcarre\ra\carre$ the inclusion functor. We denote by $(0,-):\underline{\mathbf 1}\ra\smalllrcarre$ the fully faithful functor that maps $0$ and $1$ respectively to $(0,0)$ and $(0,1)$

An object $M$ in $\bbD(\carre)$ is said to be cocartesian (resp. cartesian) if and only if the canonical morphism $(i_{\smallucarre})_\sharp (i_{\smallucarre})^*M\ra M$ (resp. $M\ra (i_{\smallucarre})_* (i_{\smallucarre})^* M$) is an isomorphism. Since $\bbD$ is stable, a square $M$ in $\bbD(\carre)$ is cartesian if and only if it is cocartesian.

Let $\doublecarre$ be the category
\begin{equation}\label{eq:catdoublesquare}
\xymatrix{{(2,1)} & {(1,1)}\ar[l] & {(0,1)}\ar[l]\\
{(2,0)}\ar[u] & {(1,0)}\ar[u]\ar[l] & {(0,0).}\ar[u]\ar[l]}
\end{equation}
There are three natural ways to embed $\carre$ in $\doublecarre$ and an object $M\in\bbD(\doublecarre)$ is said to be cocartesian  if the squares in $\bbD(\carre)$ obtained by pullback along those embeddings are cocar\-te\-sian.
A coherent triangle is a cocartesian object $M\in\bbD(\doublecarre)$ such that $(0,1)^*M$ and $(2,0)^*M$ are zero. For such an object, we have a canonical isomorphism $(0,0)^*M\simeq (2,1)^*M[1]$ and the induced sequence
\begin{equation}\label{eq:cohextr}
(2,1)^*M\ra (1,1)^*M\ra (1,0)^*M\ra (2,1)^*M[1]
\end{equation}
is an exact triangle in $\bbD(\mathbf{e})$.

One of the main advantages of working in a stable derivator is the possibility to associate with a coherent morphism $M\in\bbD(\underline{\mathbf 1})$ functorially a coherent triangle. Let us briefly recall the construction of this triangle. Let $U$ be the full subcategory of $\doublecarre$  that does not contain $(0,0)$ and $(1,0)$. Denote by $v:\underline{\mathbf 1}\ra U$ the functor that maps $0$ and $1$ respectively to $(1,1)$ and $(2,1)$ and by $u:U\ra\doublecarre$ the inclusion functor. The image under the functor 
\[u_\sharp v_*:\bbD(A\times\underline{\mathbf 1})\ra\bbD(A\times\doublecarre)\]
of a coherent morphism $M$ in $\bbD(A\times\underline{\mathbf 1})$ is a coherent triangle. 
Using the properties {\textbf{(1--2)}} recalled above, we see that \eqref{eq:cohextr} provides an exact triangle
\begin{equation}\label{eq:cofextr}
1^*M\ra 0^*M\ra \mathsf{Cof}(M)\ra 1^*M[1]
\end{equation}
where the cofiber functor $\mathsf{Cof}$ is defined by
\begin{equation}\label{eq:Cof}
\mathsf{Cof}:=(1,0)^*u_\sharp v_*:\bbD(\underline{\mathbf 1})\ra\bbD({\mathbf e}).
\end{equation}
Using the properties {\textbf{(1--3)}} recalled above, it is easy to see that this functor is also given by
\[\mathsf{Cof}=(0,0)^*(i_{\smallucarre})_\sharp (-,1)_*.\]
In the exact triangle \eqref{eq:cofextr}, the canonical morphism $0^*M\ra\mathsf{Cof}(M)$ is the $\underline{\mathbf{1}}$-skeleton of the coherent morphism $(1,-)^*u_\sharp v_*M$ where $(1,-):\underline{\mathbf 1}\ra \doublecarre$ is the fully faithful functor that maps $0$ and $1$ respectively to $(1,0)$ and $(1,1)$. Note that we have an isomorphism of functors
\[(1,-)^*u_\sharp v_*M\simeq(0,-)^*(i_{\smalllrcarre})^*(i_{\smallucarre})_\sharp (-,1)_*:\bbD(\underline{\mathbf 1})\ra\bbD(\underline{\mathbf 1}).\]
Similarly the boundary morphism $\mathsf{Cof}(M)\ra 1^*M[1]$ is the $\underline{\mathbf{1}}$-skeleton of the coherent morphism $(-,0)^*u_\sharp v_*M$ where $(-,0):\underline{\mathbf 1}\ra \doublecarre$ is the fully faithful functor that maps $0$ and $1$ respectively to $(0,0)$ and $(1,0)$.

The construction of the cofiber functor $\mathsf{Cof}$ and the cofiber triangle \eqref{eq:cofextr} can be dualized to get a fiber functor $\mathsf{Fib}$ and a fiber triangle. Let us recall the following lemma.
\begin{lemm}\label{lemm:FibCS}
Let $M\in\bbD(\carre)$. Then, we have a morphism of exact triangles
\[\xymatrix{{\mathsf{Fib}((-,1)^*M)}\ar[r]\ar[d]_-u &{(1,1)^*M}\ar[r]\ar[d] & {(0,1)^*M}\ar[r]^-{+1}\ar[d] & {}\\
{\mathsf{Fib}((-,0)^*M)}\ar[r] &{(1,0)^*M}\ar[r] &{(0,0)^*M}\ar[r]^-{+1} & {}}\]
which is functorial in $M$. Furthermore, if $M$ is cartesian if and only if the canonical morphism 
\[u:\mathsf{Fib}((-,1)^*M)\ra \mathsf{Fib}((-,0)^*M)\]
is an isomorphism.
\end{lemm}

\begin{proof}
The first statement follows from the fact that we have functorially defined fibers, as we just recalled. The second statement is Proposition~15.1.10 of Moritz Groth's
unpublished book \emph{Introduction to the theory of derivators}. 
Let us recall its proof. We have a commutative cube
\[\xy
\POS(00,00)
\xymatrix"A"{{(1,0)^*M}\ar[r]\ar[d] & {(0,0)^*M}\ar[d] \\
{\mathsf{Fib}((-,0)^*M)}\ar[r] &{0}}
\POS(-13,7) 
\xymatrix@C=1.2cm{{(1,1)^*M}\ar["A"]\ar[r]\ar[d] & {(0,1)^*M}\ar["A"]\ar@{.>}[d] \\
{\mathsf{Fib}((-,1)^*M)}\ar["A"]^(.65)u\ar@{.>}[r] &{0}\ar@{.>}["A"]}
\endxy \]
The front and back squares are cartesian by definition of an exact triangle. By \cite[Proposition~4.6]{GrothDeriv}, the top square is cartesian if and only if the bottom square is
cartesian. But the top square is $M$ and the bottom square is cartesian if and only if $u$ is an isomorphism, hence the result.

\end{proof}

There is also a functorial version of the octahedron axiom in $\bbD$ (see e.g. \cite[Proof of Theorem 9.44]{GrothBook} or \cite[Proof of Theorem 4.15]{GrothDeriv}), that is, there is a functor $\bbD(\underline{\mathbf 2})\ra \bbD(O)$
which associates to a coherent sequence of morphisms a coherent octahedron diagram. Here the category $O\subseteq \underline{\mathbf 4}\times\underline{\mathbf 2}$ is the full subcategory that does not contain the objects $(4,0)$ and $(0,2)$. In other words, $O$ is the category

\[\xymatrix{{(4,2)} &{(3,2)}\ar[l] & {(2,2)}\ar[l] &{(1,2)}\ar[l] &{}\\
{(4,1)}\ar[u] &{(3,1)}\ar[l]\ar[u] &{(2,1)}\ar[l]\ar[u] &{(1,1)}\ar[l]\ar[u] &{(0,1)}\ar[l]\\
{} &{(3,0)}\ar[u] &{(2,0)}\ar[l]\ar[u] &{(1,0)}\ar[l]\ar[u] &{(0,0).}\ar[l]\ar[u] }\]

Let $W$ be the full subcategory of $O$ that does not contain the objects $(1,1)$, $(2,1)$, $(3,1)$, $(0,0)$, $(1,0)$ and $(2,0)$. Denote by $\omega:\underline{\mathbf 2}\ra W$ the fully faithful functor which maps $0$, $1$ and $2$ respectively on $(2,2)$, $(3,2)$ and $(4,2)$ and by $w:W\ra O$ the inclusion functor. The octahedron diagram functor is defined to be the functor
\[w_\sharp\omega_*:\bbD(\underline{\mathbf 2})\ra\bbD(O).\]

Denote by $\mathsf{sm}:\underline{\mathbf 1}\ra\underline{\mathbf 2}$ the fully faithful functor that maps $0$ and $1$ respectively to $0$ and $1$ and by $\mathsf{fm}:\underline{\mathbf 1}\ra\underline{\mathbf 2}$ the fully faithful functor that maps $0$ and $1$ respectively to $1$ and $2$. Denote also by $\mathsf{cm}:\underline{\mathbf 1}\ra\underline{\mathbf 2}$ the functor which maps $0$ and $1$ respectively to $0$ and $2$.  Consider the fully faithful functor $\mathsf{fsq}:\carre\ra O$ which maps the square \eqref{eq:catsquare} to the square
\[\xymatrix{{(4,2)} & {(3,2)}\ar[l]\\
{(4,1)}\ar[u] & {(3,1).}\ar[u]\ar[l]}\]
Similarly we denote by $\mathsf{ssq}:\carre\ra O$ (resp. $\mathsf{csq}:\carre\ra O$) the fully faithful functor which maps the square \eqref{eq:catsquare} to the square
\[\xymatrix{{(3,2)} & {(2,2)}\ar[l]\\
{(3,0)}\ar[u] & {(2,0)}\ar[u]\ar[l]}\qquad \textrm{ (resp.}
\xymatrix{{(4,2)} & {(2,2)}\ar[l]\\
{(4,1)}\ar[u] & {(2,1) \textrm{).}}\ar[u]\ar[l]}\]
We have the following lemma.
\begin{lemm}\label{lemm:basicoctahedron}
We have canonical isomorphisms
\[\mathsf{fsq}^*w_\sharp\omega_*\simeq (i_{\smallucarre})_\sharp(-,1)_*\mathsf{fm}^*,\quad \mathsf{ssq}^*w_\sharp\omega_*\simeq (i_{\smallucarre})_\sharp(-,1)_*\mathsf{sm}^*\]
and 
\[\mathsf{csq}^*w_\sharp\omega_*\simeq (i_{\smallucarre})_\sharp(-,1)_*\mathsf{cm}^*.\]
\end{lemm}

\begin{proof}
Let $i:\smallucarre\ra W$ be the fully faithful functor that maps $(0,1)$, $(1,0)$ and $(1,1)$ respectively to $(3,2)$, $(4,1)$ and $(4,2)$. 
Since $\omega\circ \mathsf{fm}=i\circ (-,1)$, we get a natural transformation $ i^*\omega_*\ra (0,1)_*\mathsf{fm}^*$. Using the properties \textbf{(Der1--3)}, it is easy to see that this natural transformation is invertible. Similarly, since $w\circ i=\mathsf{fsq}\circ i_{\smallucarre}$, there is a natural transformation $(i_{\smallucarre})_\sharp i^* \ra\mathsf{fsq}^*w_\sharp$. Again, using  the properties \textbf{(Der1--3)}, we see that it is invertible. This provides invertible natural transformations
\[\xymatrix{{(i_{\smallucarre})_\sharp i^* \omega_*}\ar[d]\ar[r] & {(i_{\smallucarre})_\sharp (-,1)_*\mathsf{fm}^*}\\
{\mathsf{fsq}^*w_\sharp \omega_*.} & {}}\]
The other invertible natural transformations are constructed similarly.
\end{proof}
In particular, it follows from \ref{lemm:basicoctahedron} that $(3,1)^*w_\sharp\omega_*$ is isomorphic to $\mathsf{Cof}\circ\mathsf{fm}^*$, $(2,0)^*w_\sharp\omega_*$ is  isomorphic to $\mathsf{Cof}\circ\mathsf{sm}^*$ and $(2,1)^*w_\sharp\omega_*$ is isomorphic to $\mathsf{Cof}\circ\mathsf{cm}^*$.
Since the inverse image of $w_\sharp\omega_*$ along the fully faithful functor $\carre\ra O$ that maps the square \eqref{eq:catsquare} to the square
\[\xymatrix{{(3,1)} & {(2,1)}\ar[l]\\
{(3,0)}\ar[u] & {(2,0)}\ar[u]\ar[l]}\]
is a cocartesian square, by \ref{lemm:basicoctahedron} and \cite[D\'efinition 2.1.34]{AyoubI}, we get a natural exact triangle
\begin{equation}\label{eq:foncOct}
\mathsf{Cof}(\mathsf{fm}^*(-))\ra \mathsf{Cof}(\mathsf{cm}^*(-))\ra \mathsf{Cof}(\mathsf{sm}^*(-))\xra{+1}.
\end{equation}

Let us recall \cite[Lemma 1.4.8]{AyoubI}. Note that the functors $j^*:\DA(X,\eusm I)\ra\DA(U,\eusm I)$ and $j_*:\DA(U,\eusm I)\ra\DA(X,\eusm I)$ used below are induced by the functoriality of the categories of presheaves on diagrams of schemes (see \cite[\S 4.5]{AyoubII} for details).

\begin{lemm}
Let $\eusm I$ be a small category and $j:U\hookrightarrow X$ be an open immersion.
Assume that we have a exact triangle
\[M\ra j_*j^*M\ra C(M)\xra{+1}\]
for every given object $M\in\DA(X,\eusm I)$. Then, for every morphism $\alpha:M\ra N$ in $\DA(X,\eusm I)$ there exists one and only one morphism $C(M)\ra C(N)$ such that the square
\[\xymatrix{{C(M)}\ar[r]\ar[d]^-{C(\alpha)} & {M[1]}\ar[d]^{\alpha[1]}\\
{C(N)}\ar[r] &{N[1]}}\]
is commutative. Moreover the whole diagram
\[\xymatrix{{M}\ar[r]\ar[d]^-{\alpha} & {j_*j^*M}\ar[d]^-{j_*j^*\alpha}\ar[r] & {C(M)}\ar[r]\ar[d]^-{C(\alpha)} &{M[1]}\ar[d]^-{\alpha[1]}\\
{N}\ar[r] & {j_*j^*N}\ar[r] & {C(N)}\ar[r] &{N[1]}}\]
is commutative.
\end{lemm}

Note that in loc.cit. the lemma is stated only in the case $\eusm I=\mathbf{e}$. However its proof works in the more general situation considered here.

We will need the following technical lemma.

\begin{lemm}\label{lemm:techlemma}
Let $\eusm I$ be a small category and $f:Y\ra X$ be a morphism of separated $k$-schemes of finite type.
There exists a functor $\Delta^*_f:\DA(X,\eusm I)\ra\DA(X,\underline{\mathbf 1}\times \eusm I)$ such that, for every $M\in\DA(X,\eusm I)$, the $\underline{\mathbf 1}$-skeleton of $\Delta^*_f(M)$ is $M\ra f_*f^*M$.
\end{lemm}
\begin{proof}
Consider the diagram of $k$-varieties $(\mathscr F,\underline{\mathbf 1}\times \eusm I):\underline{\mathbf{1}}\times \eusm I\ra (\Sch/k)$ that maps $(0,i)$ to $Y$ and $(1,i)$ to $X$ and the canonical morphisms of diagrams of $k$-varieties
\[\xymatrix{{(\mathscr F,\underline{\mathbf 1}\times \eusm I)}\ar[d]^-{\alpha}\ar[r]^-{\beta} & {(X,\underline{\mathbf 1}\times \eusm I)}\\
{(X,\eusm I).} & {}}\]
The functor $\Delta_f^*:=\beta_*\alpha^*$ satisfies the desired property.

\end{proof}

\begin{rema}\label{Rema:loc}
Assume $\eusm I=\mathbf{e}$. Given $M$ in $\DA(X)$, we have an exact triangle 
\[M\ra j_*j^*M\ra\mathsf{Cof}(\Delta^*_j(M))\xra{+1}\]
functorial in $M$. It follows from \cite[Lemme 1.4.8]{AyoubI} that the functor $i_!i^!(-)[1]$ is isomorphic to $\mathsf{Cof}\circ \Delta^*_j(-)$
\end{rema}

Similarly we will need the following lemma. Its proof is completely similar to the one of \ref{lemm:techlemma} and will be omitted.

\begin{lemm}
Let $\eusm I$ be a small category and $f:Y\rightarrow X$ be a smooth morphism of separated $k$-schemes of finite type.
There exists a functor $\Delta^!_f:\DA(X,\eusm I)\ra\DA(X,\eusm I\times\underline{\mathbf 1})$ such that for every $A\in\DA(X,\eusm I)$ the $\underline{\mathbf 1}$-skeleton of $\Delta^!_f(A)$ is $f_\sharp f^*A\ra A$.
\end{lemm}

\begin{rema}\label{Rema:loc2}
Assume $\eusm I=\mathbf{e}$. Given $M$ in $\DA(X)$, as in \ref{Rema:loc}, we have an exact triangle 
\[j_!j^!M\ra M\ra\mathsf{Cof}(\Delta^!_j(M))\xra{+1}\]
functorial in $M$. It follows from (the dual statement of) \cite[Lemme 1.4.8]{AyoubI}, that the functor $i_*i^*(-)$ is isomorphic to $\mathsf{Cof}\circ \Delta^!_j(-)$.
\end{rema}

\subsection{Motivic unipotent vanishing cycles functor}\label{subsec:Unmotshfun}
Let $f:X\ra\bbA^1_k$ be a morphism of $k$-varieties. We consider the following diagram of $k$-varieties
\[\xymatrix{{X_\eta}\ar[d]^-{f_\eta}\ar[r]\ar@{}[rd]|{\square} & {X}\ar[d]^-{f} & {X_\sigma}\ar[d]^-{f_\sigma}\ar@{}[ld]|{\square}\ar[l]\\
{\bbG_{m,k}}\ar[r]^-{j} & {\bbA^1_k} & {\Spec(k)}\ar[l]_-{i}}\]
where $i$ denotes the zero section of $\bbA^1_k$ and $j$ the open immersion of the complement. We denote also by $i$ the closed immersion of the special fiber $X_\sigma$ in $X$ and by $j$ the open immersion of the generic fiber $X_\eta$ in $X$. Let $\mathsf{Log}_f$ be the logarithmic specialization system constructed in \cite[3.6]{AyoubII} (see also \cite[p.103--109]{AyoubEtale}). It is defined by
\[\mathsf{Log}_f:=\chi_f((-)\otimes f^*_\eta\Log^\vee)=:i^*j_*((-)\otimes f_\eta^*\Log^\vee)\]
where $\Log^\vee$ is the commutative associative unitary algebra in $\DA(\bbG_{m,k})$ constructed in \cite[D\'efinition 3.6.29]{AyoubII} (see also \cite[D\'efinition 11.6]{AyoubEtale}). The monodromy triangle
\begin{equation}\label{eq:monTr}
\bbQ(0)\ra\Log^\vee\xra{N}\Log^\vee(-1)\xra{+1}
\end{equation}
(see \cite[Corollaire 3.6.21]{AyoubII} or \cite[(116)]{AyoubEtale}) in the triangulated category $\DA(\bbG_{m,k})$
induces an exact triangle
\[\chi_f(-)\ra \mathsf{Log}_f(-)\ra\mathsf{Log}_f(-)(-1)\xra{+1}.\]
To construct the motivic unipotent vanishing cycles functor, we shall use the fact that the 
$\underline{\mathbf 1}$-skeleton functor 
\[\DA(\bbA^1_k,\underline{\mathbf 1})\ra \mathsf{Hom}(\underline{\mathbf 1}^\op,\DA(\bbA^1_k))\]
is full and essentially surjective. This allows to choose an object $\mathscr L$ in $\DA(\bbA^1_k,\underline{\mathbf 1})$ that lifts the morphism $\bbQ(0)\ra j_*\Log^\vee$ obtained as the composition of the adjunction morphism $\bbQ(0)\ra j_*\bbQ(0)$ and the image under $j_*$ of the unit $\bbQ(0)\ra\Log^\vee$ of the commutative associative unitary algebra $\Log^\vee$. Moreover, using the monodromy triangle \eqref{eq:monTr}, we can fix an isomorphism between $\Log^\vee(-1)$ and the cofiber of $j^*\mathscr L$ such that the diagram
\[\xymatrix{{\bbQ(0)}\ar@{=}[d]\ar[r] & {\Log^\vee}\ar@{=}[d]\ar[r] &{\Log^\vee(-1)}\ar[d]\ar[r]^-{+1} & {}\\
{\bbQ(0)}\ar[r] &{\Log^\vee}\ar[r] &{\mathsf{Cof}(j^*\mathscr L)}\ar[r]^-{+1} &{}}\]
is commutative.

Consider the object $\mathscr Q:=\Delta^*_j(\mathscr L)$ in $\DA(\bbA^1_k,\carre)$ obtained by applying the functor $\Delta^*_j$ of \ref{lemm:techlemma}. Its $\carre$-skeleton is the commutative square
\[\xymatrix{{\bbQ(0)}\ar[r]\ar[d] & {j_*\Log^\vee}\ar@{=}[d]\\
{j_*\bbQ(0)}\ar[r] &{j_*\Log^\vee.}}\]

Let $\smallllcarre$ be the full subcategory of $\carre$ that does not contain $(0,1)$. Denote by $i_{\smallllcarre}:\smallllcarre\ra\carre$ the inclusion and by $p_{\carre,\smallllcarre}:\carre\ra\smallllcarre$ the unique functor which is the identity on $\smallllcarre$ and maps $(0,1)$ to $(0,0)$. Consider the functor
\[\Theta_f(-):=(p_{\carre,\smallllcarre})^*(i_{\smallllcarre})^*\Delta^*_j((p_{\underline{\mathbf 1}})^*(-)\otimes f^*\mathscr L):\DA(X)\ra\DA(X,\carre).\]
By construction, $\Theta_f(-)$ is a coherent lifting of the commutative square
\[\xymatrix{{\Id(-)}\ar[r]\ar[d] & {j_*(j^*(-)\otimes f_\eta^*\Log^\vee)}\ar@{=}[d]\\
{j_*j^*(-)}\ar[r] & {j_*(j^*(-)\otimes f_\eta^*\Log^\vee).}}\]
By pulling back along the closed immersion $i:X_\sigma\hookrightarrow X$ we get the functor
\[i^*\Theta_f(-):\DA(X)\ra\DA(X_\sigma,\carre)\]
which is a coherent lifting of the commutative square
\[\xymatrix{{i^*(-)}\ar[r]\ar[d] & {\mathsf{Log}_f(j^*(-))}\ar@{=}[d]\\
{\chi_f(j^*(-))}\ar[r] & {\mathsf{Log}_f(j^*(-)).}}\]

Let $(-,1):\underline{\mathbf 1}\ra\carre$ be the fully faithful functor that maps $0$ and $1$ respectively to $(0,1)$ and $(1,1)$. In particular, the $\underline{\mathbf 1}$-skeleton of $(-,1)^*i^*\Theta_f(-)$ is the morphism
\[i^*(-)\ra \mathsf{Log}_f(j^*(-)).\]

\begin{defi}
The motivic unipotent vanishing cycles functor $\Phi_f:\DA(X)\ra\DA(X_\sigma)$ is defined as the composition of $(-,1)^*i^*\Theta_f(-)$ and the cofiber functor: 
\[\Phi_f:=\mathsf{Cof}\circ (-,1)^*i^*\Theta_f(-).\]
\end{defi}

By construction, we get
a natural transformation $can:\mathsf{Log}_f(-)\circ j^*\ra\Phi_f(-)$ and an exact triangle
\begin{equation}\label{eq:CanDTVanCycl}
i^*\ra\mathsf{Log}_f(-)\circ j^*\xra{can}\Phi_f\xra{+1}.
\end{equation}
We also get a natural transformation
 \[var:\Phi_f(-)\ra\mathsf{Log}_f(j^*(-))(-1)\] 
 such that $var\circ can=N$. Indeed, let $(-,0):\underline{\mathbf 1}\ra\carre$ be the fully faithful functor that maps $0$ and $1$ respectively to $(0,0)$ and $(1,0)$.
 
The chosen isomorphism between $\Log^\vee(-1)$ and the cofiber of $j^*\mathscr L$ induces an isomorphism between $\mathsf{Log}_f(j^*(-))(-1)$ and the cofiber of $(-,0)^*i^*\Theta_f(-)$  such that the diagram
\[\xymatrix{{\chi_f(j^*(-))}\ar@{=}[d]\ar[r] & {\mathsf{Log}_f(j^*(-))}\ar@{=}[d]\ar[r]^-N &{\mathsf{Log}_f(j^*(-))(-1)}\ar[r]^-{+1} & {}\\
{\chi_f(j^*(-))}\ar[r] &{\mathsf{Log}_f(j^*(-))}\ar[r] &{\mathsf{Cof}((-,0)^*i^*\Theta_f(-))}\ar[u]\ar[r]^-{+1} &{}}\]
is commutative. On the other hand, the canonical morphism $(-,1)^*\Theta_f(-)\ra (-,0)^*\Theta_f(-)$ in $\DA(X,\underline{\mathbf 1})$ induces a commutative diagram
\[\xymatrix{{\chi_f(j^*(-))}\ar[r] & {\mathsf{Log}_f(j^*(-))}\ar@{=}[d]\ar[r] &{\mathsf{Cof}((-,0)^*i^*\Theta_f(-))}\ar[r]^-{+1} & {}\\
{i^*(-)}\ar[r]\ar[u] &{\mathsf{Log}_f(j^*(-))}\ar[r]^-{can} &{\Phi_f(-)}\ar[u]\ar[r]^-{+1} &{.}}\]
By applying the coherent triangle functor $u_\sharp v_*$ to the object $i^*\Theta_f(-)$ of the category $\DA(X_\sigma,\carre)=\DA(X_\sigma,\underline{\mathbf 1}\times\underline{\mathbf 1})$, we get a functor
\[\DA(X)\ra\DA(X_\sigma,\underline{\mathbf 1}\times\doublecarre)\]
which is a coherent lifting of the commutative diagram
\[\xy
\POS(00,00)
\xymatrix"A"{{\chi_f(j^*(-))}\ar[r]\ar[d] & {\mathsf{Log}_f(j^*(-))}\ar[r]\ar[d]_(.28){N} &{0}\ar[d]\\
{0}\ar[r] &{\mathsf{Log}_f(j^*(-))(-1)}\ar[r] &{\chi_f(j^*(-))[1].}}
\POS(-13,7) 
\xymatrix@C=1.2cm{{i^*(-)}\ar["A"]\ar[r]\ar[d] & {\mathsf{Log}_f(j^*(-))}\ar@{=}["A"]\ar[r]\ar@{.>}[d]_(.3){can} &{0}\ar["A"]\ar@{.>}[d]\\
{0}\ar["A"]\ar@{.>}[r] &{\Phi_f(-)}\ar@{.>}["A"]^-{var}\ar@{.>}[r] &{i^*(-)[1]}\ar@{.>}["A"]}
\endxy \]
The category $\underline{\mathbf 1}\times\doublecarre$ is given by 

\[\xy
\POS(-13,8) 
\xymatrix"A"{{(1,2,1)} & {(1,1,1)}\ar[l] &{(1,0,1)}\ar[l]\\
{(1,2,0)}\ar[u] &{(1,0,1)}\ar@{.>}[u]\ar@{.>}[l] &{(1,0,0)}\ar@{.>}[u]\ar@{.>}[l]}
\POS(00,00)
\xymatrix{{(0,2,1)}\ar["A"] & {(0,1,1)}\ar[l]\ar["A"] &{(0,1,0)}\ar[l]\ar["A"]\\
{(0,2,0)}\ar[u]\ar["A"] &{(0,1,0)}\ar[u]\ar[l]\ar@{.>}["A"] &{(0,0,0)}\ar[u]\ar[l]\ar@{.>}["A"]}
\endxy \]
and we consider the functor $\mathsf{sq}:\carre\ra \underline{\mathbf 1}\times\doublecarre$ which maps \eqref{eq:catsquare} to the square
\[\xymatrix{{(1,0,1)} & {(1,0,0)\ar[l]}\\
{(0,1,0)}\ar[u] & {(0,0,0)}\ar[l]\ar[u]}\]
inside $\underline{\mathbf 1}\times\doublecarre$. In the next subsection, we will be mainly focusing on the functor
\[\mathsf{sq}^*u_\sharp v_*i^*\Theta_f:\DA(X)\ra\DA(X,\carre)\]
which is a coherent lifting of the commutative square
\[\xymatrix{{\Phi_f(-)}\ar[r]\ar[d]^-{var} & {i^*(-)[1]}\ar[d]\\
{\mathsf{Log}_f(j^*(-))(-1)}\ar[r] & {\chi_f(j^*(-))[1].}}\]

\begin{rema}\label{rema:cartsquarePhi}
The square $\mathsf{sq}^*u_\sharp v_*i^*\Theta_f$ is cartesian. This can be deduced from the basic properties of cartesian squares (for example from \cite[Proposition 4.6]{GrothDeriv}).
\end{rema}

\subsection{Maximal extension functor}\label{subsec:MaxExtFunc}

Let us now construct Be{\u\i}linson's maximal extension functor $\Xi_f$ (see~\cite{MR923134}) and the related exact triangles  in the  triangulated categories of \'etale motives. This will be essential to prove \ref{theo:support} and for gluing perverse motives. By applying the coherent triangle functor $u_\sharp v_*$ to the object $\Theta_f(-)$ in $\DA(X,\carre)=\DA(X,\underline{\mathbf 1}\times\underline{\mathbf 1})$, we get a functor
\[u_\sharp v_*\Theta_f:\DA(X)\ra\DA(X,\underline{\mathbf 1}\times\doublecarre)\]
which is a coherent lifting of the commutative diagram
\[\xy
\POS(00,00)
\xymatrix"A"{{j_*j^*(-))}\ar[r]\ar[d] & {j_*(j^*(-)\otimes f^*_\eta\Log^\vee)}\ar[r]\ar[d] &{0}\ar[d]\\
{0}\ar[r] &{j_*(j^*(-)\otimes f^*_\eta\Log^\vee(-1))}\ar[r] &{j_*j^*(-)[1].} }
\POS(-13,7) 
\xymatrix@C=1.2cm{{\Id(-)}\ar["A"]\ar[r]\ar[d] & {j_*(j^*(-)\otimes f^*_\eta\Log^\vee)}\ar@{=}["A"]\ar[r]\ar@{.>}[d] &{0}\ar["A"]\ar@{.>}[d]\\
{0}\ar["A"]\ar@{.>}[r] &{\bullet}\ar@{.>}["A"]\ar@{.>}[r] &{\Id(-)[1]}\ar@{.>}["A"]}
\endxy \]
(Here $\bullet$ is some motive which we do not need to specify).
The category $\underline{\mathbf 1}\times\doublecarre$ is given by 
\[\xy
\POS(-13,8) 
\xymatrix"A"{{(1,2,1)} & {(1,1,1)}\ar[l] &{(1,0,1)}\ar[l]\\
{(1,2,0)}\ar[u] &{(1,0,1)}\ar@{.>}[u]\ar@{.>}[l] &{(1,0,0)}\ar@{.>}[u]\ar@{.>}[l]}
\POS(00,00)
\xymatrix{{(0,2,1)}\ar["A"] & {(0,1,1)}\ar[l]\ar["A"] &{(0,1,0)}\ar[l]\ar["A"]\\
{(0,2,0)}\ar[u]\ar["A"] &{(0,1,0)}\ar[u]\ar[l]\ar@{.>}["A"] &{(0,0,0).}\ar[u]\ar[l]\ar@{.>}["A"]}
\endxy \]
Let $\smalllrcarre$ be the full subcategory of $\carre$ that does not contain $(1,1)$. 
Then $\underline{\mathbf 1}\times \smalllrcarre$ is the category
\begin{equation}\label{eq:catdoublecarreone}
\xymatrix@R=.5cm@C=.5cm{{} & {(1,0,1)} &{}\\
{(1,1,0)} &{(1,0,0)}\ar[l]\ar[u] &{(0,0,1)}\ar[lu]\\
{} &{(0,1,0)}\ar[lu] &{(0,0,0).}\ar[l]\ar[lu]\ar[u]}
\end{equation}
We denote by $\alpha:\underline{\mathbf 1}\times\smalllrcarre\ra\doublecarre\times\underline{\mathbf 1}$ the functor which maps \eqref{eq:catdoublecarreone} to
\[\xy
\POS(-2,8) 
\xymatrix"A"{{} & {} &{(1,0,1)}\\
{} &{} &{(1,0,0)}\ar@{.>}[u]}
\POS(00,00)
\xymatrix{{} & {(0,1,1)} &{(0,1,0)}\ar[l]\ar["A"]\\
{} &{(0,1,0)}\ar[u] &{(0,0,0).}\ar[u]\ar[l]\ar@{.>}["A"]}
\endxy \]
Then, $\alpha^*u_\sharp v_*\Theta_f(-):\DA(X)\ra\DA(X,\underline{\mathbf 1}\times\smalllrcarre)$ is a coherent lifting of the commutative diagram
\[\xymatrix@R=.25cm@C=.25cm{{} & {0}\ar[rd]\ar[d] &{}\\
{j_*(j^*(-)\otimes f^*_\eta\Log^\vee)}\ar[r]\ar[rd]_(.45)N &{0}\ar[rd] &{\Id(-)[1]}\ar[d]\\
{} &{j_*(j^*(-)\otimes f_\eta^*\Log^\vee(-1))}\ar[r] &{j_*j^*(-)[1]}}\]
and $(0\times\Id_{\smalllrcarre})^*\alpha^*u_\sharp v_*\Theta_f(-)=(i_{\smalllrcarre})^*\mathsf{sq}^*u_\sharp v_*\Theta_f(-)$. Let $\beta:\underline{\mathbf 1}\times\smalllrcarre\ra\underline{\mathbf 1}\times\underline{\mathbf 1}\times\smalllrcarre $ be the fully faithful functor defined by
\begin{align*}
(0,0,0)\mapsto (0,0,0,0) & (1,0,0)\mapsto (0,1,0,0)  \\ 
(0,0,1)\mapsto (0,0,0,1) & (1,0,1)\mapsto (0,1,0,1)\\
(0,1,0)\mapsto (0,0,1,0) & (1,1,0)\mapsto  (1,1,1,0).
\end{align*}
The $\underline{\mathbf 1}\times\smalllrcarre$-skeleton of the functor
 \[\Sigma_f(-):=\beta^*\circ\Delta^!_j\circ\alpha^*u_\sharp v_*\Theta_f(-):\DA(X)\ra\DA(X,\underline{\mathbf 1}\times \smalllrcarre)\]
is now the commutative diagram
\[\xymatrix@R=.25cm@C=.25cm{{} & {0}\ar[rd]\ar[d] &{}\\
{j_!\big(j^*(-)\otimes f^*_\eta\Log^\vee\big)}\ar[r]\ar[rd] &{0}\ar[rd] &{\Id(-)[1]}\ar[d]\\
{} &{j_*\big(j^*(-)\otimes f_\eta^*\Log^\vee(-1)\big)}\ar[r] &{j_*j^*(-)[1]}}\]
where the non-zero diagonal morphism 
is obtained via the canonical morphism $j_!\ra j_*$ and the monodromy operator. Note that we have \[(0\times \Id_{\smalllrcarre})^*\Sigma_f=(0\times \Id_{\smalllrcarre})^*\alpha^*u_\sharp v_*\Theta_f(-)=(i_{\smalllrcarre})^*\mathsf{sq}^*u_\sharp v_*\Theta_f(-).\]
In particular, we have canonical isomorphisms $(0,0,0)^*\Sigma_f(-)=j_*j^*(-)[1]$ and \\$(0,0,1)^*\Sigma_f(-)=\Id(-)[1]$.

\begin{defi}
Let $\Xi_f:\DA(X)\ra\DA(X)$ be the functor defined by
\[\Xi_f(-):=(1,0)^*\mathsf{Cof}(\Sigma_f(-)).\]
We also define $\Omega_f:\DA(X)\ra \DA(X) $ to be the functor
\[\Omega_f(-):=(1,1)^*(i_{\smalllrcarre})_*\mathsf{Cof}(\Sigma_f(-)).\]
\end{defi}

By construction, we have an exact triangle
\begin{equation}\label{eq:DTOmegaXi}
\Omega_f(-)\ra \Xi_f(-)\oplus (0,1)^*\mathsf{Cof}(\Sigma_f(-))\ra (0,0)^*\mathsf{Cof}(\Sigma_f(-))\xra{+1}.
\end{equation}
Since the canonical morphisms
\[\Id(-)[1]=(0,0,1)^*\Sigma_f(-)\ra (0,1)^*\mathsf{Cof}(\Sigma_f(-))\]
and 
\[j_*j^*(-)[1]=(0,0,0)^*\Sigma_f(-)\ra (0,0)^*\mathsf{Cof}(\Sigma_f(-))\]
are isomorphisms, the exact triangle \eqref{eq:DTOmegaXi} can be rewritten as
\[\Omega_f(-)\ra \Xi_f(-)\oplus \Id(-)[1]\ra j_*j^*(-)[1]\xra{+1}.\]
On the other hand, we have an exact triangle
\[(1,1,0)^*\Sigma_f(-)\ra (0,1,0)^*\Sigma_f(-)\ra\Xi_f(-)\xra{+1}\]
that is an exact triangle 
\begin{equation}\label{eq:DTXi}
j_!\big(j^*(-)\otimes f_\eta^*\Log^\vee\big)\ra j_*\big(j^*(-)\otimes f_\eta^*\Log^\vee(-1)\big)\ra \Xi_f(-)\xra{+1}.
\end{equation}

\begin{prop}
There are exact triangles
\begin{equation}\label{eq:XiA}
i_*\mathsf{Log}_f(j^*(-))\ra\Xi_f\ra j_*j^*(-)[1]\xra{+1}
\end{equation}
and
\begin{equation}\label{eq:XiB}
j_!j^*(-)[1]\ra\Xi_f\ra i_*\mathsf{Log}_f(j^*(-))(-1)\xra{+1}.
\end{equation}
\end{prop}

\begin{proof}
Let us first construct \eqref{eq:XiA} using the functorial version of the octahedron axiom (see \ref{subsec:derivator}). Recall that by definition
\[\Xi_f(-):=(1,0)^*\mathsf{Cof}(\Sigma_f(-))=\mathsf{Cof}((-,1,0)^*\Sigma_f(-)).\]
Let us set
\[\Sigma'_f(-):=\Delta^!_j\circ\alpha^*u_\sharp v_*\Theta_f(-):\DA(X)\ra\DA(X,\underline{\mathbf 1}\times\underline{\mathbf 1}\times \smalllrcarre)\]
so that $\Sigma_f(-)=\beta^*\Sigma'_f(-)$. Now let $\gamma:\underline{\mathbf 2}\ra \underline{\mathbf 1}\times\underline{\mathbf 1}\times \smalllrcarre$ be the fully faithful functor that maps $0$, $1$ and $2$ respectively to $(0,0,1,0)$, $(0,1,1,0)$ and  $(1,1,1,0)$. Recall that $\mathsf{cm}:\underline{\mathbf 1}\ra\underline{\mathbf 2}$ is the fully faithful functor that maps $0$ and $1$ respectively to $0$ and $2$. Then, $\beta\circ (-,1,0)=\gamma\circ \mathsf{cm}$. In particular, we get that
\[(-,1,0)^*\Sigma_f(-)=\mathsf{cm}^*\gamma^*\Sigma'_f(-).\]
Using the exact triangle \eqref{eq:foncOct} given by the functorial octahedron axiom, we get an exact triangle
\[\mathsf{Cof}(\mathsf{fm}^*\gamma^*\Sigma'_f(-))\ra \Xi_f(-)\ra\mathsf{Cof}(\mathsf{sm}^*\gamma^*\Sigma'_f(-))\xra{+1}.\]
However, by construction, we have an exact triangle
\[j_!(j^*(-)\otimes f^*_\eta\Log^\vee)\ra j_*(j^*(-)\otimes f^*_\eta\Log^\vee)\ra\mathsf{Cof}(\mathsf{fm}^*\gamma^*\Sigma'_f(-))\xra{+1}. \]
Using \ref{Rema:loc2}, we see that 
$\mathsf{Cof}(\mathsf{fm}^*\gamma^*\Sigma'_f(-))$ is isomorphic to 
\[i_*\mathsf{Log}_f(j^*(-)):=i_*i^*j_*(j^*(-)\otimes f^*_\eta\Log^\vee).\]
On the other hand, $\mathsf{sm}^*\gamma^*\Sigma'_f(-)=(0,1,-)^*u_\sharp v_*\Theta_f(-)$,
so that we get an isomorphism
\[\mathsf{Cof}(\mathsf{fm}^*\gamma^*\Sigma'_f(-))=(0,0,0)^*u_\sharp v_*\Theta_f(-)=j_*j^*(-)[1].\]
This constructs the exact triangle \eqref{eq:XiA}.
Consider now the localization triangle
\[j_!j^*\Xi_f(-)\ra\Xi_f(-)\ra i_*i^*\Xi_f(-)\xra{+1}.\]
To obtain \eqref{eq:XiB} it is enough to check that $j^*\Xi_f(-)$ is isomorphic to $j^*(-)[1]$ and that $i^*\Xi_f(-)$ is isomorphic to $\mathsf{Log}_f(j^*(-))$. The first isomorphism is obtained by applying $j^*$ to \eqref{eq:XiA} and the second isomorphism is obtained by applying $i^*$ to \eqref{eq:DTXi}.

\end{proof}

\begin{prop}
There are exact triangles
\begin{equation}\label{eq:TrOmegaId}
i_*\mathsf{Log}_f(j^*(-))\ra\Omega_f\ra \Id(-)[1]\xra{+1}
\end{equation}
and
\begin{equation}\label{eq:LocTrOmega}
j_!j^*(-)[1]\ra\Omega_f\ra i_*\Phi_f(-)\xra{+1}.
\end{equation}
\end{prop}

\begin{proof}
Using \eqref{eq:XiA}, the exact triangle \eqref{eq:TrOmegaId} is obtained by applying \ref{lemm:FibCS} to the cartesian square $(i_{\smalllrcarre})_*\mathsf{Cof}(\Sigma_f(-))$. 

Since $j^*i_*=0$, \eqref{eq:TrOmegaId} provides an isomorphism between $j^*\Omega_f(-)$ and $j^*(-)[1]$. Now, consider the localization triangle
\[j_!j^*\Omega_f\ra\Omega_f\ra i_*i^*\Omega_f(-)\xra{+1}.\]
To construct \eqref{eq:LocTrOmega}, it is enough to obtain an isomorphism between $i^*\Omega_f(-)$ and $\Phi_f(-)$.
By definition 
\[i^*\Omega_f(-)=(1,1)^*(i_{\smalllrcarre})_*\mathsf{Cof}(i^*\Sigma_f(-)).\]
However since $i^*j_!=0$, the canonical morphism
\[(0\times \Id_{\smalllrcarre})^*i^*\Sigma_f(-)\ra \mathsf{Cof}(i^*\Sigma_f(-))\]
is an isomorphism.
Given that
 \[(0\times \Id_{\smalllrcarre})^*\Sigma_f=(0\times \Id_{\smalllrcarre})^*\alpha^*u_\sharp v_*\Theta_f(-)=(i_{\smalllrcarre})^*\mathsf{sq}^*u_\sharp v_*\Theta_f(-),\]
 we get isomorphisms
\[ \xymatrix{{(1,1)^*(i_{\smalllrcarre})_*(0\times \Id_{\smalllrcarre})^*i^*\Sigma_f(-)}\ar[r]^-{\simeq}\ar@{=}[d] & { (1,1)^*(i_{\smalllrcarre})_*\mathsf{Cof}(i^*\Sigma_f(-))=i^*\Omega_f(-)}\\
 { (1,1)^*(i_{\smalllrcarre})_*(i_{\smalllrcarre})^*\mathsf{sq}^*u_\sharp v_*\Theta_f(-).} &{}}\]
 By \ref{rema:cartsquarePhi}, the canonical morphism 
 \[\Phi_f(-)=(1,1)^*\mathsf{sq}^*u_\sharp v_*\Theta_f(-)\ra (1,1)^*(i_{\smalllrcarre})_*(i_{\smalllrcarre})^*\mathsf{sq}^*u_\sharp v_*\Theta_f(-)\]
 is an isomorphism. This concludes the proof.
\end{proof}

\subsection{Betti realization}\label{subsec:CompBetti}

Let $X$ be a complex algebraic variety. Let $\AnDA(X)$ be the triangulated category of analytic motives. This category is obtained as the special case of the category $\SH^\anc_\dM(X)$ considered in \cite{AyoubBetti} when the stable model category $\dM$ is taken to be the category of unbounded complexes of $\bbQ$-vector spaces with its projective model structure. Recall that the canonical triangulated functor
\begin{equation}\label{eq:BettiA}
i^*_X:\mathbf{D}(X)\ra\AnDA(X)
\end{equation}
is an equivalence of categories (see \cite[Th\'eor\`eme 1.8]{AyoubBetti}). Here $\mathbf{D}(X)$ denotes the (unbounded) derived category of sheaves of $\bbQ$-vector spaces on the associated analytic space $X^\anc$. The functor
\[\mathsf{An}_X:(\Sm/X)\ra (\mathrm{AnSm}/X^\anc)\] 
which maps a smooth $X$-scheme $Y$ to the associated $X^\anc$-analytic space $Y^\anc$ induces a triangulated functor
\begin{equation}\label{eq:BettiB}
\mathsf{An}^*_X:\DA(X)\ra\AnDA(X).
\end{equation}
The Betti realization $\Bti^*_X$ of \cite{AyoubBetti} is obtained as the composition of \eqref{eq:BettiB} and a quasi-inverse to \eqref{eq:BettiA}.

Let $\Log^\vee_\sP$ be the image under the Betti realization of the motive $\Log^\vee$ and consider the specialization system it defines
\[\mathsf{Log}_f^\sP(-):=i^*_\sP j^\sP_*((-)\otimes (f_\eta)_\sP^*\Log^\vee_\sP):\mathbf{D}(X_\eta)\ra\mathbf{D}(X_\sigma).\]
Recall that in \ref{subsec:Unmotshfun} we fixed an object $\mathscr L$ in $\DA(\bbA^1_k,\underline{\mathbf 1})$ that lifts the morphism $\bbQ(0)\ra j_*\Log^\vee$ obtained as the composition of the adjunction morphism $\bbQ(0)\ra j_*\bbQ(0)$ and the image under $j_*$ of the unit $\bbQ(0)\ra\Log^\vee$ of the commutative associative unitary algebra $\Log^\vee$.

Let $\mathscr L_\sP$ the image in $\mathbf{D}(X,\underline{\mathbf 1})$ of $\mathscr L$. Using this object, we can perform the same constructions as in \ref{subsec:Unmotshfun} and \ref{subsec:MaxExtFunc} using the derivator $\mathbf{D}(X,-)$ to obtain functors
\[\Xi^\sP_f(-),\Omega^\sP_f(-):\mathbf{D}(X)\ra\mathbf{D}(X)\]
and 
\[\Phi^\sP_f(-):\mathbf{D}(X)\ra\mathbf{D}(X_\sigma)\]
and four exact triangles: the two triangles
\[
i^\sP_*\mathsf{Log}^\sP_fj_\sP^*\ra\Xi^\sP_f\ra j^\sP_*j_\sP^*[1]\xra{+1},\; j^\sP_!j^*_\sP[1]\ra\Xi^\sP_f\ra i^\sP_*\mathsf{Log}^\sP_f(-1)\xra{+1}
\]
and the two triangles
\[
i^\sP_*\mathsf{Log}^\sP_fj_\sP^*\ra\Omega^\sP_f\ra \Id[1]\xra{+1},\;j^\sP_!j^*_\sP[1]\ra\Omega^\sP_f\ra i^\sP_*\Phi^\sP_f\xra{+1}.
\]
Moreover, we have canonical natural transformations
\[\Bti^*\circ\mathsf{Log}_f\ra \mathsf{Log}^\sP_f\circ \Bti^*,\;\Bti^*\circ\Phi_f\ra \Phi^\sP_f\circ \Bti^*\]
and 
\[\Bti^*\circ\Xi_f\ra \Xi^\sP_f\circ \Bti^*,\;\Bti^*\circ\Omega_f\ra \Omega^\sP_f\circ \Bti^*\;\]
which are isomorphisms when applied to constructible motives (see \cite[Th\'eor\`eme 3.9]{AyoubBetti}) and are also compatible with the various exact triangles.

As proved in \cite[Th\'eor\`eme 4.9]{AyoubBetti}, the Betti realization is compatible with the (total) nearby cycles functors for constructible motives. In this subsection, we will need the compatibility of the Betti realization with the unipotent nearby cycles functors.  

\begin{lemm}\label{lemm:compNCBetti}
The functor $\mathsf{Log}_f^\sP(-)$ is isomorphic to the unipotent nearby cycles functor $\psi^\un_f(-)$.
\end{lemm}

Let $e:\bbC \ra \bbC^\times; z\mapsto \exp(z)$
be the universal cover of the punctured complex plane $\bbC^\times$. 
The group of deck transformations is identified with $\bbZ$ by mapping the integer $k\in\bbZ$ to the deck transformation $z\mapsto z+2i\pi k$.

Let $\mathscr E_n$ be the unipotent rational local system on $\bbC^\times$ of rank $n+1$ with (nilpotent) monodromy given by one Jordan block of maximal size. It underlies a variation of $\bbQ$-mixed Hodge structures described e.g. in \S1.1 of Morihiko Saito's~\cite{MR1047741}. 

Let us recall the description of this local system and relate it to Ayoub's logarithmic motive $\Log^\vee_n$. The following description is given in M. Saito, \cite[2.3. Remark]{MR1045997}. Let $\mathscr E_n$ be the subsheaf of $e_*\bbQ_{\bbC}$ annihilated by $(T-\Id)^{n+1}$ where $T$ is the automorphism of $e_*\bbQ_\bbC$ induced by the deck transformation corresponding to $1\in\bbZ$. The restriction of $T$ to $\mathscr E_n$ is unipotent and we denote by $N=\log T$ the associated nilpotent endomorphism.

The sheaf $\mathscr E_n$ is a local system on $\bbC^\times$ of rank $n+1$. Let $(\mathscr E_n)_1$ be its fiber over $1$. We have an inclusion
\[(\mathscr E_n)_1\subseteq(e_*\bbQ_\bbC)_1=\prod_{k\in\bbZ}(\bbQ_\bbC)_{2i\pi k}=\prod_{k\in\bbZ}\bbQ.\]
Note that the automorphism $T$ acts by mapping a sequence $(a_k)_{k\in\bbZ}$ to $(a_{k+1})_{k\in\bbZ}$. 
Let $\tau_n$ be the element in $(\mathscr E_n)_1$ given by $\tau_n=(k^n/n!)_{k\in\bbZ}$. The family $(1,\tau_1,\ldots,\tau_n)$ is a basis of $(\mathscr E_n)_1$ such that 
$T(\tau_r)=\sum_{k=0}^r\tau_k/(r-k)! $
for every $r\in \llbracket1,n\rrbracket$. The matrix with respect to the basis $(1,\tau_1,\ldots,\tau_n)$ of the unipotent endomorphism $T$ of $(\mathscr E_n)_1$ is thus given by $\sum_{k=0}^n (J_n)^k/k!$ 
where $J_n$ is the nilpotent Jordan block of size $n+1$ and therefore $N$ is given by the Jordan block $J_n$ in the basis $(1,\tau_1,\ldots,\tau_n)$.

The multiplication $e_*\bbQ_\bbC\otimes e_*\bbQ_\bbC\ra e_*\bbQ_\bbC$ induces a morphism of local systems $\mathscr E_k\otimes \mathscr E_\ell\ra \mathscr E_{k+\ell}$. In particular, for $n\in\bbN^*$, we have a canonical morphism $\mathscr E_1^{\otimes n}\ra \mathscr E_n$ which defines a morphism
\begin{equation}\label{eq:MorLocSys}
\mathrm{Sym}^n\mathscr E_1\ra \mathscr E_n.
\end{equation}
If $\tau:=\tau_1$, then $\tau_n=\tau^n/n!$ and the above description of $\mathscr E_n$ implies that \eqref{eq:MorLocSys} is an isomorphism.

Let us consider the Kummer natural transform 
$e_K:\Id(-)(-1)[-1]\ra\Id(-)$
in Betti cohomology (see \cite[D\'efinition 3.6.22]{AyoubII}).
By 5.1 Lemma in M. Saito's \cite{SaitoMixedSheaves}, the local system $\mathscr E_1$ fits into an exact triangle
\[\bbQ(-1)[-1]\xra{e_K}\bbQ\ra \mathscr E_1\xra{+1}.\]
By \cite[Th\'eor\`eme 3.19]{AyoubBetti} the Betti realization is compatible with the Kummer transform (for constructible motives). In particular, we have a natural isomorphism $\Bti^*\mathscr K\ra \mathscr E_1$ where $\mathscr K\in\DA(\bbG_{m})$ is the motivic Kummer extension, that is, the cone of the Kummer natural transform for \'etale motives (see \cite[Lemme 3.6.28]{AyoubII}). Since the Betti realization $\Bti^*$ is a symmetric monoidal functor, it induces an isomorphism 
\[\Bti^*\Log^\vee_n=\Bti^*\mathrm{Sym}^n\mathscr K\xra{\simeq}\mathrm{Sym}^n\Bti^*\mathscr K\xra{\simeq} \mathrm{Sym}^n\mathscr E_1\xra{\eqref{eq:MorLocSys}}\mathscr E_n\]
for every integer $n\in\bbN$. Therefore, we get an isomorphism 
\begin{equation}\label{eq:isoLogE}
\Log^\vee_\sP:=\Bti^*\Log^\vee\xra{\simeq}\mathscr E
\end{equation}
where $\mathscr E$ is the ind-local system given by $\mathscr E=\colim_{n\in\bbN^\times} \mathscr E_n$.

Let $K\in\Dbc(X,\bbQ)$, the unipotent nearby cycles functor $\psi_f^{\mathrm{un}}$ is given by 
\[\psi^\mathrm{un}_f(K)=i^*_\sP j^\sP_*(K\otimes (f_\eta)_\sP^*\mathscr E)\]
 (see \cite[(2.3.3)]{MR1045997}, Beilison's~\cite{MR923134} or Reich's~\cite{MR2671769}). With this description, \ref{lemm:compNCBetti} is an immediate consequence of \eqref{eq:isoLogE}.

\begin{coro}
The functors \[{}^p\mathsf{Log}_f^\sP(-):=\mathsf{Log}_f^\sP(-)[-1],\quad{}^p\Phi_f^\sP(-):=\Phi_f^\sP(-)[-1] ,\quad {}^p\Xi^\sP_f(-):= \Xi^\sP_f(-)[-1]\]
and
\[\quad {}^p\Omega^\sP_f(-):=\Omega^\sP_f(-)[-1] \] are $t$-exact for the perverse $t$-structure. 
\end{coro}

\begin{proof}
Since the functor $\psi_f^{\mathrm{un}}(-)[-1]$ is $t$-exact for the perverse $t$-structure, the corollary is an immediate consequence of \ref{lemm:compNCBetti} and the exact triangles relating the various functors.
\end{proof}

\subsection{Application to perverse motives}
\label{subsec:AppPerverseMotive}

Now, we can apply the universal property of the categories of perverse motives to obtain four exact functors
\[{}^p\mathsf{Log}_f^\sM(-):\sM(X_\eta)\ra\sM(X_\sigma),\;{}^p\Phi_f^\sM(-):\sM(X)\ra\sM(X_\sigma)\]
and 
\[{}^p\Xi_f^\sM(-):\sM(X)\ra\sM(X),\;{}^p\Omega_f^\sM(-):\sM(X)\ra\sM(X).\]
Moreover we have four canonical exact sequences obtained from the exact triangles relating the four functors used in the construction. Two exact sequences \[
0\ra i^\sM_*{}^p\mathsf{Log}^{\sM}_f(j_\sM^*(-))\ra{}^p\Xi^{\sM}_f\ra j^\sM_*j^*_\sM(-)\ra 0
\]
and 
\[0\ra j^\sM_!j^*_\sM(-)\ra{}^p\Xi^{\sM}_f\ra i^\sM_*{}^p\mathsf{Log}^{\sM}_f(-)(-1)\ra 0.\]
As well as two exact sequences
\begin{equation}\label{eq:shortexNC}
0\ra i^\sM_*{}^p\mathsf{Log}^{\sM}_f(j_\sM^*(-))\ra{}^p\Omega^{\sM}_f(-)\ra \Id(-)\ra 0
\end{equation}
and 
\[0\ra j^\sM_!j_\sM^*(-)\ra{}^p\Omega^{\sM}_f(-)\ra i^\sM_*{}^p\Phi^{\sM}_f(-)\ra 0.
\]
These four functors and the associated exact sequences are compatible with the various functors and exact triangles constructed in \ref{subsec:Unmotshfun}, \ref{subsec:MaxExtFunc}  and \ref{subsec:CompBetti}.

Now we can prove the following theorem.

\begin{theo}\label{theo:support}
Let $i:Z\hookrightarrow X$ be a closed immersion of $k$-varieties. Then, the functor \[i^\sM_*:\sM(Z)\ra\sM(X)\] is fully faithful and its essential image is the kernel, denoted by $\sM_Z(X)$, of the exact functor \[j^*_\sM:\sM(X)\ra\sM(U)\]
where $j:U\hookrightarrow X$ is the open immersion of the complement of $Z$ in $X$. 
\end{theo}

We first consider the case of the immersion of a special fiber.

\begin{lemm}\label{lemm:support}

Let $X$ be a $k$-variety and $f:X\ra\bbA^1_k$ be a morphism. Let $i:X_\sigma\hookrightarrow X$ be the closed immersion of the special fiber in $X$ and $Z$ be a closed subscheme of $X_\sigma$. Then, the exact functor
\[i^\sM_*:\sM_Z(X_\sigma)\ra\sM_Z(X)\]
is an equivalence of categories.
\end{lemm}
\begin{proof}
We may assume $Z=X_\sigma$. Indeed, let $u:X\setminus Z\hookrightarrow X$ and $v:X_\sigma\setminus Z\hookrightarrow X_\sigma$ be the open immersion. By   \ref{prop:ExStructure} applied to cartesian square
\[\xymatrix{{X_\sigma\setminus Z}\ar[r]^-{i'}\ar@{}[rd]|{\square}\ar[d]^-{v} &{X\setminus Z}\ar[d]^-u\\
{X_\sigma}\ar[r]^-i &{X,}}\]
we get an isomorphism $u^*_\sM i_*^\sM\simeq i'^\sM_*v^*_\sM$. Since the functor $i'^\sM_*$ is conservative (it is faithful exact), we see that an object $A$ in $\sM(X_\sigma)$ belongs to $\Ker v^*_\sM$ if and only if $i^\sM_*A$ belongs to $\Ker u^*_\sM$. Hence, it is enough to show that 
\[i^\sM_*:\sM(X_\sigma)\ra\sM_{X_\sigma}(X)\] 
is an equivalence.

 Let us show that  the functor ${}^p\Phi^\sM_f$ is a quasi-inverse.  Let $X_\eta$ be the generic fiber and $j:X_\eta\hookrightarrow X$ be the open immersion. The exact triangle \eqref{eq:CanDTVanCycl}, provides an isomorphism of endomorphisms of $\DA(X_\sigma)$ between $i^*i_*$ and $\Phi_f[-1] i_*$. By composing with the isomorphism of functors $i^*i_*\ra \Id$, we get an isomorphism of functors between the identity of $\DA(X_\sigma)$ and  ${}^p\Phi_f[-1] i_*$.

Similarly, we get an isomorphism  between the identity  of $\D(X_\sigma,\bbQ)$ and the functor $\Phi_f^\sP[-1] i^\sP_*$. Since these isomorphisms are compatible with the Betti realization, the property \textbf{P2}, ensures that ${}^p\Phi^\sM_f i^\sM_*$ is isomorphic to the identity functor of the category $\sM(X_\sigma)$.

An isomorphism between the identity of $\sM_{X_\sigma}(X):=\Ker j^*_\sM$  and $i^*_\sM{}^p\Phi_f$ is provided by the exact sequences 
\[0\ra i^\sM_*{}^p\mathsf{Log}^{\sM}_f(j^*_\sM(-))\ra{}^p\Omega^{\sM}_f(-)\ra \Id(-)\ra 0\]
and 
\[0\ra j^\sM_!j^*_\sM(-)\ra{}^p\Omega^{\sM}_f(-)\ra i^\sM_*{}^p\Phi^{\sM}_f(-)\ra 0
\]
(the first terms vanish for objects in the kernel of $j^*_\sM$). This concludes the proof.
\end{proof}
\begin{proof}[Proof of \ref{theo:support}]
Using \ref{prop:stack}, we may assume that $X$ is an affine scheme. Let $U$ be the open complement of $Z$ in $X$ and let $f_1,\ldots,f_r$ be elements in $\Osheaf(X)$ such that $U=D(f_1)\cup\cdots D(f_r)$. Let $Z_{r+1}=X$ and set $Z_k=Z_{k+1}\setminus D(f_k)$ for $k\in\llbracket 1,r\rrbracket$. Let $i_k:Z_k\hookrightarrow Z_{k+1}$ be the closed immersion. We have $Z_1=Z$ and $i=i_r\circ i_{r-1}\circ\dots\circ i_1$, so that the functor $i^\sM_*:\sM(Z)\ra\sM_Z(X)$ is obtained as the composition
\[\sM(Z)\xra{(i_1)_*^\sM}\sM_Z(Z_2)\xra{(i_2)^\sM_*}\sM_Z(Z_3)\ra\cdots\ra\sM_Z(Z_k)\xra{(i_k)^\sM_*}\sM_Z(X).\]
By \ref{lemm:support}, all these functors are equivalences. This concludes the proof.
 \end{proof}

\section{Inverse images}\label{sec:pullback}
The purpose of this section is to extend the (contravariant) $2$-functor ${}^{\mathrm{Liss}}\mathsf{H}^*_\sM $ constructed in \ref{subsec:InvImageLiss} into a (contravariant) $2$-functor
\begin{align*}
\mathsf{H}^*_\sM:(\Sch/k)&\ra\mathfrak{TR}\\
X& \mapsto \Db(\sM(X))\\
f & \mapsto f^*_\sM.
\end{align*}
To do this, we first use the vanishing cycles functor to show that the (covariant) $2$-functor ${}^{\mathrm{Imm}}\mathsf{H}^\sM_*$ admits a global left adjoint ${}^{\mathrm{Imm}}\mathsf{H}^*_\sM$ (we recall that a global left adjoint is unique up to unique isomorphism and refer to \cite[D\'efinition 1.1.18]{AyoubI} for the definition). Then, we show that the $2$-functors ${}^{\mathrm{Liss}}\mathsf{H}^*_\sM $ and ${}^{\mathrm{Imm}}\mathsf{H}^*_\sM $ can be glued into a $2$-functor $\mathsf{H}^*_\sM$.

\subsection{Inverse image by a closed immersion}

By \cite[Proposition 1.1.17]{AyoubI}, to show that ${}^{\mathrm{Imm}}\mathsf{H}^\sM_*$ admits a global left adjoint ${}^{\mathrm{Imm}}\mathsf{H}^*_\sM$ it suffices to show that for every closed immersion $i:Z\hookrightarrow X$ the functor $i_*^\sM$ admits a left adjoint; this in turn is proved in
 \ref{prop:ImmExistenceAdjoint}.

\begin{theo}\label{theo:SupportDb}
Let $i:Z\hookrightarrow X$ be a closed immersion. Then, the functor 
 \[i^\sM_*:\Db(\sM(Z))\ra\Db(\sM(X))\]
 is fully faithful and its essential image is  is the kernel, denoted by  $\Db_Z(\sM(X))$, of the exact functor \[j^*_\sM:\Db(\sM(X))\ra \Db(\sM(U))\]
where $j:U\hookrightarrow X$ is the open immersion of the complement of $Z$ in $X$. 

\end{theo}

\begin{proof}
 We know that the
essential image of $i^\sM_*:\Db(\sM(Z))\ra\Db(\sM(X))$ is contained
in $\Db_Z(\sM(X))$ by \ref{theo:support}. We now want to prove that
the functor 
$i^\sM_*:\Db(\sM(Z))\ra\Db_Z(\sM(X))$ is an equivalence of categories.
Note that the obvious $t$-structure on $\Db(\sM(X))$ induces a $t$-structure
on $\Db_Z(\sM(X))$, whose heart is the thick abelian subcategory $\sM_Z(X)$
of $\sM(X)$.
By \ref{theo:support},
the functor $i^\sM_*:\sM(Z)\ra\sM(X)$ induces an equivalence of categories
$\sM(Z)\ra\sM_Z(X)$. So, by \cite[Lemma 1.4]{MR923133}, the functor
$i^\sM_*:\Db(\sM(Z))\ra\Db_Z(\sM(X))$ is an equivalence of categories if and only
if, for any $A,B$ in $\sM_Z(X)$ and
$i\geqslant 1$,  and any class
$u\in\Ext^i_{\sM(X)}(A,B)$, there exists a monomorphism $B\hookrightarrow B'$
in $\sM_Z(X)$ such that the image of $u$ in $\Ext^i_{\sM(X)}(A,B')$ is $0$.

Suppose that $j:V\hookrightarrow X$ is an affine open immersion, that $A$ is
an object of $\sM(X)$ and that $B$ is an object of $\sM(V)$.
Let $i\geqslant 1$. Then, we have \[\Ext^i_{\sM(X)}(A,j^\sM_*B)=\Ext^i_{\sM(V)}(j^*_\sM A,B)\]
by \ref{prop:AdjEtaleAffine}, and, if $u\in\Ext^i_{\sM(V)}(j^*_\sM A,B)$ and
$B\hookrightarrow B'$ is a monomorphism of $\sM(V)$ such that
the image of $u$ in $\Ext^i_{\sM(V)}(j^*_\sM A,B')$ is $0$, then, the image
in $\Ext^i_{\sM(X)}(A,j^\sM_*B')$ of the element of $\Ext^i_{\sM(X)}(A,j^\sM_*B)$
corresponding to $u$ is also $0$. Applying this to an open cover
$j_1:U_1\hookrightarrow X$, \ldots, $j_n:U_n\hookrightarrow X$ of $X$ by
affine subsets and using the fact the canonical map
$B\ra\bigoplus_{r=1}^n (j_r)^\sM_*(j_r)^*_\sM B$ given by \ref{prop:AdjEtaleAffine}
is a monomorphism for every 
object $B$ of $\sM(X)$, we reduce to the case where $X$ is affine.

If $X$ is affine, then, as in the proof of \ref{theo:support},
we write $i=i_r\circ\dots\circ i_1$, where $Z_1=Z$, $Z_{r+1}=X$, and, for
every $k\in\{1,\dots,r\}$,
$i_k:Z_k\ra Z_{k+1}$ is the immersion of the complement of an open
set of the form $D(f)$, with $f\in\Osheaf(Z_{k+1})$. It suffices to show
that each $(i_{k})^\sM_*:\Db(\sM(Z_k))\ra\Db_{Z_k}(\sM(Z_{k+1}))$ is an
equivalence of categories. So we may assume that there exists
$f\in\Osheaf(X)$ such that $i$ is the immersion of the complement
of $D(f)$. In that case, we showed in the proof of \ref{lemm:support}
that the trivial derived functor of the exact functor
${}^p\Phi^\sM_f:\sM(X)\ra\sM(Z)$ induces a quasi-inverse of
$i^\sM_*:\Db(\sM(Z))\ra\Db_Z(\sM(X))$.
\end{proof}

\begin{prop}\label{prop:ImmExistenceAdjoint}
Let $i:Z\hookrightarrow X$ be a closed immersion. Then, the functor \[i^\sM_*:\Db(\sM(Z))\ra\Db(\sM(X))\] admits a left adjoint.
\end{prop}

\begin{proof}
By \ref{theo:SupportDb}, it suffices to show that the inclusion functor
\begin{equation}\label{eq:inclfunc}
\Db_Z(\sM(X))\ra\Db(\sM(X))
\end{equation}
admits a left adjoint $C^\bullet$.  Let $j:U\hookrightarrow X$ be the open immersion of the complement of $Z$ in $X$. Let us first assume that $U$ is affine. In that case, given $A$ in $\Cb(\sM(X))$, we define $C^\bullet(A)$ as the mapping cone of the canonical morphism $j^\sM_!j_\sM^*A\ra A$ given by \ref{prop:AdjEtaleAffine}. This construction induces a triangulated functor $C^\bullet:\Db(\sM(X))\ra\Db(\sM(X))$ and there is a canonical exact triangle
$j^\sM_!j_\sM^*A\ra  A\ra C^\bullet(A)\ra j^\sM_!j_\sM^*A[1]$,
which shows that $C^\bullet$ takes its values in the full subcategory $\Db_Z(\sM(X))$. Let $B\in\Db_Z(\sM(X))$. Using the long exact sequence associated with this triangle and \ref{prop:AdjEtaleAffine} which ensures that 
\[\Hom_{\Db(\sM(X))}(j^\sM_!j_\sM^*A,B[n])=\Hom_{\Db(\sM(U))}(j_\sM^*A,j_\sM^*B[n])=0,\]
we get a functorial isomorphism
\[\Hom_{\Db(\sM(X))}(C^\bullet(A),B)\xra{\simeq} \Hom_{\Db(\sM(X))}(A,B)\]
as desired. 

In the general case, the adjoint $C^\bullet$ can be constructed by considering a finite set $I$ and an affine open covering  $\mathscr U=(j_i:U_i\ra U)_{i\in I}$. For every $J\subseteq I$, let $j_J$ be the inclusion
$\bigcap_{i\in J}U_i\hookrightarrow X$.
We define an exact functor $C^\bullet:\sM(X)\ra \Cb(\sM(X))$ in the following
way. Let $A$ be an object of $\sM(X)$. We set:
\[C^i(A)=\left\{\begin{array}{ll}0 & \mbox{ if }i\geqslant 1 \\
A & \mbox{ if }i=0 \\
\bigoplus_{I\subset\{1,\dots,r\},|J|=-i}(j_J)^\sM_!(j_J)^*_\sM A & \mbox{ if }i\leqslant -1.
\end{array}\right.\]
The differential of $C^\bullet(A)$ is an alternating sum of maps given by
\ref{prop:AdjEtaleAffine}. Then, the left adjoint of $\Db_Z(\sM(X))\ra\Db(\sM(X))$ is
the functor sending $A^\bullet$ to the total complex of $C^\bullet(A^\bullet)$.
\end{proof}

Let $Z$ be a closed immersion such that the open immersion $j:U\hookrightarrow X$ of the complement of $Z$ in $X$  is affine. It follows from the proof of \ref{prop:ImmExistenceAdjoint} that we have a canonical exact triangle
\[j^\sM_!j^*_\sM\ra\Id\ra i^\sM_*i^*_\sM\xra{+1}.\]
Moreover the diagram
\begin{equation}\label{eq:MorDTLocRat}
\xymatrix{{j^\sP_!j^*_\sP\rat^\sM_X}\ar[r]\ar[d]^-{\theta^\sM_j} &{\rat^\sM_X}\ar[r]\ar@{=}[d] &{i^\sP_*i^*_\sP\rat^\sM_X}\ar[d]^-{\theta_i^\sM}\ar[r]^-{+1} &{}\\
{j^\sP_!\rat^\sM_U j^*_\sM}\ar[r]\ar[d]^{\rho^\sM_j} &{\rat^\sM_X}\ar[r]\ar@{=}[d] &{i^\sP_*\rat^\sM_Z i^*_\sM}\ar[r]^-{+1}\ar[d]^{(\gamma_i^\sM)^{-1}} &{}\\
{\rat^\sM_Xj^\sM_!j^*_\sM}\ar[r] &{\rat^\sM_X}\ar[r] &{\rat^\sM_X i^\sM_*i^*_\sM}\ar[r]^-{+1} &{}}
\end{equation}
is commutative (the morphisms in the second row are those obtained by adjunction from $\theta_i^\sM$ and the inverse of $\theta^\sM_j$). 

\begin{lemm}
Let $i:Z\ra X$ be a closed immersion. Then, the natural transformation 
\[\theta^\sM_i:i^*_\sP\circ\rat^\sM_X\ra\rat^\sM_Z\circ i^*_\sM\]
is invertible.
\label{lemma_theta_i}
\end{lemm}

\begin{proof}
The statement is local on $X$, so we may assume that $X$ is affine.
Then, as in the proof of \ref{theo:SupportDb}, we can write
$i=i_r\circ\ldots\circ i_1$, where each $i_s$ is a closed immersion
with affine complement.
Using the compatibility of the $2$-morphisms $\theta^\sM_i$ with the composition of morphisms in ${}^{\mathrm{Imm}}(\Sch/k)$ we may assume that the open immersion $j:U\hookrightarrow X$ of the complement of $Z$ in $X$  is affine. Then, our assertion follows from \eqref{eq:MorDTLocRat} and the conservativity of the functor $i_*^\sP$.
\end{proof}

\subsection{Gluing of the pullback $2$-functors}
Let us fix a global left adjoint ${}^{\mathrm{Imm}}\mathsf{H}^*_\sM$ of ${}^{\mathrm{Imm}}\mathsf{H}_*^\sM$. 
To be able to glue the $2$-functors ${}^{\mathrm{Imm}}\mathsf{H}^*_\sM$ and ${}^{\mathrm{Liss}}\mathsf{H}^*_\sM$  using \cite[Th\'eor\`eme 1.3.1]{AyoubI}, it suffices to construct, for every commutative square
\begin{equation}\label{eq:GlueSquare}
\xymatrix{{X'}\ar[r]^-{i'}\ar[d]^-{f'} &{Y'}\ar[d]^-{f}\\
{X}\ar[r]^i &{Y}}
\end{equation}
such that $i,i'$ are closed immersions and $f,f'$ smooth morphisms, a $2$-isomorphism
\begin{equation}\label{eq:Glue2iso}
i'^*_\sM\circ f^*_\sM\xra{\simeq}f'^*_\sM\circ i^*_\sM 
\end{equation}
and prove that these $2$-isomorphisms define an exchange structure, that is, they are compatible with the horizontal and vertical composition of commutative squares (see \cite[D\'efinition 1.2.1]{AyoubI}).

\subsubsection*{dg-enhancements}
For the general theory of dg categories we refer to Drinfeld's~\cite{MR2028075}, Keller's~\cite{MR1258406,MR2275593} or To\"en's \cite{MR2762557}.
Let $\eusm A$ be an abelian category. We denote by $\eusmCb(\eusm A)$ the dg category of bounded complexes of objects of $\eusm A$ and by $\eusmDb(\eusm A)$ the dg quotient of $\eusmCb(\eusm A)$ by the subcategory of acyclic bounded complexes (for a simple construction of the dg quotient see \cite[\S3.1]{MR2028075}). The bounded derived category $\Db(\eusm A)$ of $\eusm A$ is the homotopy category of the dg category $\eusmDb(\eusm A)$. We let $\mathrm{rep}(\eusmDb(\eusm A),\eusmDb(\eusm B))$
 be the category of dg quasi-functors from $\eusmDb(\eusm A) $ to $\eusmDb(\eusm B) $ (this category is denoted by $\mathcal{T}(\eusmDb(\eusm A),\eusmDb(\eusm B))$ in Vologodsky's paper~\cite{MR2729639}). Let us recall the following particular case of \cite[Theorem 1]{MR2729639}.

\begin{prop}\label{prop:Vologodsky}
Let $\eusm A,\eusm B$ be abelian categories and $F,G\in r(\eusm A,\eusm B)$ be dg quasi-functors. Assume that the induced triangulated functors $F,G:\Db(\eusm A)\ra\Db(\eusm B)$ are $t$-exact for the classical $t$-structures.
Then $F,G$ are respectively canonically isomorphic to the functors induced by the exact functors $H^0F:\eusm A\ra\eusm B,H^0G:\eusm A\ra\eusm B$ and the canonical map
\[\Hom_{\mathrm{rep}(\eusmDb(\eusm A),\eusmDb(\eusm B))}(F,G)\ra\Hom_{\mathrm{Fct}(\eusm A,\eusm B)}(H^0F,H^0G)\]
is an isomorphism.
\end{prop}

A triangulated functor $\Db(\eusm A)\ra\Db(\eusm B)$ is said to be dg enhanced if it is induced by some dg quasi-functor in $\mathrm{rep}(\eusmDb(\eusm A),\eusmDb(\eusm B))$. Note that a composition of dg enhanced functors is also dg enhanced. 

\begin{rema}\label{rema:dgenhanced}
Let $i:Z\hookrightarrow X$ be a closed immersion and $f:X\ra Y$ be a smooth morphism of quasi-projective $k$-varieties. By construction the triangulated functors $i^\sM_*$ and  $f^*_\sM$  are dg enhanced. This is also the case of the triangulated functor 
 \[i^*_\sM:\Db(\sM(X))\ra\Db(\sM(Z)).\] 
 Indeed, let $j:U\hookrightarrow X$ be the open immersion of the complement of $Z$ in $X$ and fix a finite open covering of $U$ by affine open subsets. Let $\eusm D^{\mathrm{b}}_{\mathrm{dg},Z}(\sM(X)) $ be the dg full subcategory of $\eusmDb(\sM(X)) $ formed by the complexes that belongs to $\Db_Z(\sM(X))$. We have then dg-functors
\begin{equation}\label{eq:dgenhanced}
\xymatrix{{\eusmDb(\sM(Z))}\ar[r]^-{i^\sM_*} & {\eusm D^{\mathrm{b}}_{\mathrm{dg},Z}(\sM(X))} & {\eusmDb(\sM(X))}\ar[l]_-{C^\bullet}}
\end{equation}
where $C^\bullet$ is the dg functor constructed (using the given open covering of $U$ by affine open subsets) in the proof of \ref{prop:ImmExistenceAdjoint}. 
Since the dg-functor on the left is a quasi-equivalence, the diagram \eqref{eq:dgenhanced} defines a quasi-functor from $\eusmDb(\sM(X)) $ to $\eusmDb(\sM(Z)) $ that induces the triangulated functor $i^*_\sM$.
\end{rema}

\subsubsection*{Gluing of the $2$-functors}
Let us now start with the construction of the $2$-isomorphisms \eqref{eq:Glue2iso}.\par\medskip
{\underline{Step 1:}} When the square \eqref{eq:GlueSquare} is cartesian the $2$-isomorphism \eqref{eq:Glue2iso} is obtained 
by considering the exchange structure ${}^\sM Ex^*_*$ on the pair $({}^{\mathrm{Imm}}\mathsf{H}_*^\sM,{}^{\mathrm{Liss}}\mathsf{H}^*_\sM)$
obtained in \ref{prop:ExStructure} (in this exchange structure, all squares are cartesian). By applying \cite[Proposition 1.2.5]{AyoubI}, we get an exchange structure ${}^\sM Ex^{**}$ on the pair $({}^{\mathrm{Imm}}\mathsf{H}^*_\sM,{}^{\mathrm{Liss}}\mathsf{H}^*_\sM)$ for the class of cartesian squares \eqref{eq:GlueSquare}. The uniqueness in loc.cit. implies that this exchange structure lifts the trivial exchange structure on $({}^{\mathrm{Imm}}\mathsf{H}^*_\sP,{}^{\mathrm{Liss}}\mathsf{H}^*_\sP)$ given by the connection $2$-isomorphisms of the $2$-functor $\mathsf{H}^*_\sP$. In particular, the conservativity of the functors $\rat^\sM_X:\Db(\sM(X))\ra\Db(\sP(X))$ implies that ${}^\sM Ex^{**}$ is an iso-exchange.

\par\medskip
\underline{Step 2:} Let us consider a commutative triangle
\begin{equation}\label{eq:diapure}
\xymatrix{{X}\ar[r]^-{i}\ar[rd]_-{g} & {Y}\ar[d]^-{f}\\
{} & {S}}
\end{equation} 
in which $i$ is a closed immersion and $f,g$ are smooth morphisms. As preparation for the construction of the $2$-isomorphism \eqref{eq:Glue2iso},  we first construct a $2$-isomorphism
\begin{equation}\label{eq:Step3}
i^*_\sM\circ f^*_\sM\ra g^*_\sM.
\end{equation}
To do this, observe that, if $d$ is the relative dimension of $g$, then the triangulated functors $i^*_\sM\circ f^*_\sM[d]$ and $g^*_\sM[d] $ are $t$-exact for the classical $t$-structures. This is a vanishing statement that can be checked after application of the functor $\rat^\sM_X$ and, for perverse sheaves, it follows from \cite[4.2.4]{BBD} since $g^*_\sP$ and $i^*_\sP\circ f^*_\sP$ are isomorphic. Moreover both functors are dg enhanced by \ref{rema:dgenhanced}.

By \ref{prop:Vologodsky}, to construct \eqref{eq:Step3}, it is enough to construct a $2$-isomorphism 
\begin{equation}\label{eq:nu}
i^*_\sM\circ f^*_\sM[d]\ra g^*_\sM[d]
\end{equation}
where both functors are exact functors from $\sM(S)$ to $\sM(X)$. Therefore, it suffices to prove the following proposition.

\begin{prop} 
Consider the commutative diagram
\eqref{eq:diapure}.
 Let $A$ be an object in $\sM(S)$ and let $K$ be its underlying perverse sheaf. Then, the canonical morphism of perverse sheaves  $i^*_\sP\circ f^*_\sP[d](K)\ra g^*_\sP[d](K)$
lies in the image of the injective morphism
\begin{equation}\label{eq:morphismcannu}
\Hom_{\sM(X)}(i^*_\sM\circ f^*_\sM[d](A),g^*_\sM[d](A))\ra \Hom_{\sP(X)}(i^*_\sP\circ f^*_\sP[d](K),g^*_\sP[d](K)).
\end{equation}

\end{prop}

\begin{rema}
Note that the map \eqref{eq:morphismcannu} is obtained via the functor $\rat^\sM_X$ using the invertible natural transformations $f^*_\sP\circ\rat^\sM_S\ra \rat^\sM_Y\circ f^*_\sM$, $g^*_\sP \circ\rat^\sM_S\ra\rat^\sM_X\circ g^*_\sM $ and $i^*_\sP\circ \rat^\sM_Y\ra\rat^\sM_X\circ i^*_\sM$ which have been previously constructed.
\end{rema}

\begin{proof}
\emph{Step (a).} Consider a commutative diagram
\[\xymatrix{{X'}\ar[r]^-{i'}\ar@/_4em/[rdd]_-{g'}\ar@{}[rd]|{\square}\ar[d]^-{v} & {Y'}\ar[d]^-{u}\ar@/^2em/[dd]^-{f'}\\
{X}\ar[r]^i\ar[rd]_g & {Y}\ar[d]^-{f}\\
{} & {S}}\]
where $i$ is a closed immersion, $f,g$ are smooth morphisms and $u$ is an \'etale morphism. By step 1, we have a natural transformation $i'^*_\sM\circ u^*_\sM\ra v^*_\sM\circ i^*_\sM$ that lifts the corresponding natural transformation in the derived category of perverse sheaves. Assume the proposition true for the diagram \eqref{eq:diapure}. Then, the morphism $i^*_\sP\circ f^*_\sP[d](K)\ra g^*_\sP[d](K)$ lifts to a morphism $i^*_\sM\circ f^*_\sM[d](A)\ra g^*_\sM[d](A)$. By applying $v^*_\sM$ to this lift we obtain a morphism 
$i'^*_\sM\circ f'^*_\sM[d](A)\ra g'^*_\sM[d](A) $ 
that lifts the morphism  $i'^*_\sP\circ f'^*_\sP[d](K)\ra g'^*_\sM[d](K)$. 
This shows, in particular, that if the proposition is true for the diagram \eqref{eq:diapure} then it is also true for the diagram
\[\xymatrix{{X'}\ar[r]^-{i'}\ar[rd]_-{g'} & {Y'}\ar[d]^-{f'}\\
{} & {S.}}\]
\par\smallskip 
\emph{Step (b).}  Let $\mathscr Y=(Y_\alpha)_{\alpha\in I}$ be a finite Zariski open covering of $Y$ and consider for every $\alpha\in I$ the commutative diagram
\[\xymatrix{{X_\alpha}\ar[r]^{i_\alpha}\ar@/_4em/[rdd]_-{g_\alpha}\ar@{}[rd]|{\square}\ar[d]^-{v_\alpha} & {Y_\alpha}\ar[d]^-{u_\alpha}\ar@/^2em/[dd]^-{f_\alpha}\\
{X}\ar[r]^i\ar[rd]_g & {Y}\ar[d]^-{f}\\
{} & {S}}\]
where $u_\alpha$ is the open immersion of $Y_\alpha$ in $Y$. Note that the canonical morphism of perverse sheaves $i^*_\sP\circ f^*_\sP[d](K)\ra g^*_\sP[d](K)$ is obtained by gluing the morphism $i^*_{\alpha,\sP}\circ f^*_{\alpha,\sP}[d](K)\ra g^*_{\alpha,\sP}[d](K)$ along the Zariski open covering $\mathscr X=(X_\alpha)_{\alpha\in I}$
of $X$. Hence it follows from step (a) and \ref{prop:stack} that the proposition is true for the diagram \eqref{eq:diapure} if and only if it is true for the diagrams
\[\xymatrix{{X_\alpha}\ar[r]^-{i_\alpha}\ar[rd]_{g_\alpha} & {Y_\alpha}\ar[d]^-{f_\alpha}\\
{} & {S.}}\]

\par\smallskip
\emph{Step (c).} By step (b) the problem is local on $Y$ for the Zariski topology. Since both $Y$ and $X$ are smooth over $S$, we may assume that there exists a cartesian square
\[\xymatrix{{X}\ar[r]^-{i}\ar@/_4em/[rdd]_-{g}\ar@{}[rd]|{\square}\ar[d]^-{v} & {Y}\ar[d]^-{u}\ar@/^2em/[dd]^-{f}\\
{\bbA^d_S}\ar[r]\ar[rd] & {\bbA^{d+c}_S}\ar[d]\\
{} & {S}}\]
where $u$ is an \'etale morphism. Using step (a) and induction, we are reduced to proving the proposition in the case 
\[\xymatrix{{\bbA^d_S}\ar[rd]_-{\pi}\ar[r]^{s} & {\bbA^{d+1}_S}\ar[d]^-{p}\\
{} & {S}} \]
where $p$ and $\pi$ are the projections and $s$ is the zero section. By considering the factorization
\[\xymatrix{{\bbA^d_S}\ar@{=}[rd]\ar[r]^{s}\ar[rdd]_-{\pi} & {\bbA^{d+1}_S}\ar@/^2em/[dd]^-{p}\ar[d]\\
{} & {\bbA^d_S}\ar[d]^-{\pi}\\
{} & {S}} \]
and observing that the functors $\pi^*_\sM[d]$, $\pi^*_\sP[d]$ are exact, we may further assume $d=0$.
\par\smallskip
\emph{Step (d).}  
It remains to prove the proposition in the case of the diagram
\[\xymatrix{{S}\ar[r]^-{s}\ar@{=}[rd] & {\bbA^1_S}\ar[d]^p\\
{} & {S}}\]
where $s$ is the zero section and $p$ is the projection.
Let  $f:\bbA^1_S=\bbA^1\times S\ra \bbA^1$ be the first projection, $a:\bbG_m\times S\ra \bbA^1\times S$ be the inclusion. We set $q=p\circ a$. Given a motive $B\in\sM(\bbA^1_S)$, consider the connecting morphism  $B\ra s^\sM_*{}^p\mathsf{Log}^\sM_f(a^*_\sM(B))[1]$
in $\Db(\sM(\bbA^1_S))$ obtained from the exact sequence \eqref{eq:shortexNC}. By adjunction, we get a morphism $s^*_\sM(B)\ra {}^p\mathsf{Log}^\sM_f(a^*_\sM B)[1]$
 in the category $\Db(\sM(S))$. Taking $B$ to be the perverse motive $B=p^*_\sM[1] A$, we get after a shift a morphism
 \[s^*_\sM\circ p^*_\sM(A)\ra {}^p\mathsf{Log}^\sM_f(q^*_\sM[1](A))\]
in $\Db(\sM(S))$. As both objects are concentrated in degree zero, the above morphism is actually a morphism in the abelian category $\sM(S)$. Moreover, it is an isomorphism since it is on the underlying perverse sheaves. Moreover, we know that the square 
\[\xymatrix{{\Hom_{\sM(S)}({}^p\mathsf{Log}^\sM_f(q^*_\sM[1](A)),A)}\ar[r]\ar[d]^-{\simeq} & {\Hom_{\sP(S)}({}^p\mathsf{Log}^\sP_f(q^*_\sP[1](K)),K)}\ar[d]^-{\simeq}\\
{\Hom_{\sM(S)}(s^*_\sM\circ p^*_\sM(A),A)}\ar[r] & {\Hom_{\sP(S)}(s^*_\sP\circ p^*_\sP(K),K)}}\]
is commutative. Hence, to conclude, it suffices to show that the canonical morphism of perverse sheaves
\begin{equation}\label{eq:nattransfNC}
{}^p\mathsf{Log}^\sP_f(q^*_\sP[1](K))\ra K
\end{equation}
lifts to a morphism ${}^p\mathsf{Log}^\sM_f(q^*_\sM[1](A))\ra A $ in the abelian category $\sM(S)$.  By construction of the exact functors ${}^p\mathsf{Log}_f^\sM$ and $q^*_\sM[1]$, this is an application of the property \textbf{P2}, since \eqref{eq:nattransfNC} is the Betti realization of a natural transformation
\[\mathsf{Log}_f(q^*(-))\ra\Id\]
 in the triangulated category of \'etale motives on $S$. 
\end{proof}

\begin{lemm}\label{lemm:CompositionTriangle}
Consider a commutative diagram
\[\xymatrix{{X}\ar[r]^-{i}\ar[rrd]_-{h} &{Y}\ar[r]^-{s}\ar[rd]^-{g} &{Z}\ar[d]^-{f}\\
{} &{} &{S}}\]
in which $i,s$ are closed immersions and $f,g,h$ are smooth morphisms. Then, the diagram
\[\xymatrix{{i^*_\sM\circ s^*_\sM\circ f^*_\sM}\ar[r]\ar[d]_-{\simeq} & {i^*_\sM\circ g^*_\sM}\ar[r] &{h^*_\sM}\\
{(s\circ i)^*_\sM \circ f^*_\sM}\ar@/_2em/[rru] &{} &{}}\]
is commutative.
\end{lemm}

\begin{proof} The lemma follows from the analogous statement for perverse sheaves. Indeed,  let $d$ be the relative dimension of $h$. It suffices to show that the diagram 
\[\xymatrix{{i^*_\sM\circ s^*_\sM\circ f^*_\sM[d]}\ar[r]\ar[d]_-{\simeq} & {i^*_\sM\circ g^*_\sM[d]}\ar[r] &{h^*_\sM[d]}\\
{(s\circ i)^*_\sM \circ f^*_\sM[d]}\ar@/_2em/[rru] &{} &{}}\] is commutative. Since all functors in this diagram are dg enhanced and $t$-exact for the classical $t$-structures, by \ref{prop:Vologodsky} it suffices to check the commutativity of the diagram induced on the hearts. This can be checked on the underlying perverse sheaves.
\end{proof}

{\underline{Step 3:}} To construct the $2$-isomorphisms \eqref{eq:Glue2iso} in the general case, we can decompose the commutative square \eqref{eq:GlueSquare} as follows
\[\xymatrix{{X'}\ar[r]^-{i'''}\ar[rd]_-{f'}\ar@/^2em/[rr]^-{i'} & {X\times_YY'}\ar[r]^{i''}\ar[d]^-{f''}\ar@{}[rd]|{\square} &{Y'}\ar[d]^-{f}\\
{} &{X}\ar[r]^-{i} & {Y}}\]
where $i'',i'''$ are closed immersions and $f''$ is a smooth morphism. Then, using the iso-exchange constructed in step 1, the $2$-isomorphism of step 2 and the connection $2$-isomor\-phisms of the $2$-functor ${}^{\mathrm{Imm}}\mathsf{H}^*_\sM$ we get \eqref{eq:Glue2iso} as the composition
\[i'^*_\sM\circ f^*_\sM\xra{\simeq} i'''^*_\sM\circ i''^*_\sM\circ f^*_\sM\xra{\simeq}i'''^*_\sM\circ f''^*_\sM\circ i^*_\sM\xra{\simeq} f'^*_\sM\circ i^*_\sM.\]

\begin{lemm}\label{lemm:SquareTriangleA}
 Let
\[\xymatrix{{X'}\ar[r]^-{i'}\ar[d]_-{f'}\ar@{}[rd]|{\square} &{Y'}\ar[r]^-{s}\ar[d]^-{f} &{Z}\ar[ld]^-{g}\\
{X}\ar[r]^-i &{Y} &{}}\]
be a commutative diagram of morphisms of $k$-varieties in which $g,f$ are smooth and $i,s$ are closed immersions. Consider the commutative diagram
\[\xymatrix{{X'}\ar@/^2em/[rr]^-{s\circ i'}\ar[r]^-{s'}\ar[rd]_-{f'} &{X\times_YZ}\ar[r]^-{i''}\ar[d]^-{g'}\ar@{}[rd]|{\square} &{Z}\ar[d]^-{g}\\
{} &{X}\ar[r]^-{i} &{Y.}}\]
Then, the following diagram is commutative
\[\xymatrix{{(s\circ i')^*_\sM\circ g^*_\sM}\ar@{=}[d]\ar[r]^-{\simeq} &{i'^*_\sM\circ s^*_\sM\circ g^*_\sM}\ar[r] &{i'^*_\sM\circ f^*_\sM}\ar[r] &{f'^*_\sM\circ i^*_\sM}\ar@{=}[d]\\
{(i''\circ s')^*_\sM\circ g^*_\sM}\ar[r]^-{\simeq} &{s'^*_\sM \circ i''^*_\sM\circ g^*_\sM}\ar[r] &{s'^*_\sM\circ g'^*_\sM\circ i^*_\sM}\ar[r] &{f'^*_\sM\circ i^*_\sM.}}\]
\end{lemm}
\begin{proof}
By adjunction, it is enough to show that the diagram
\[\xymatrix{{s^*_\sM\circ g^*_\sM\circ i_*^\sM A}\ar[r]\ar@/_2em/[rd] & {f^*_\sM\circ  i^\sM_*A}\ar[r] & {i'^\sM_*\circ f'^*_\sM A}\\
{}& {s^*_\sM \circ i''^\sM_*\circ g'^*_\sM A}\ar[r] &{i'^\sM_*\circ s'^*_\sM \circ g'^*_\sM A}\ar[u]}\]
is commutative for every object $A$ in $\Db(\sM(X))$. Since all the entries of the above diagram are dg enhanced and $t$-exact functors up to a shift by the relative dimension $d$ of $f$, by \ref{prop:Vologodsky} it suffices to check the commutativity of the diagram induced on the hearts. This can be checked on the underlying perverse sheaves.
\end{proof}

\begin{lemm}\label{lemm:SquareTriangleB}
Consider a commutative diagram
\[\xymatrix{{X'}\ar[d]^-{f'}\ar[r]^-{i'}\ar@{}[rd]|{\square}\ar@/_5em/[rdd]_-{h'} & {Y'}\ar[d]^-f\ar@/^2em/[dd]^-h\\
{X}\ar[r]^i\ar[rd]_-{g'} &{Y}\ar[d]^-g\\
{} & {S}}\]
in which $i,i'$ are closed immersions and all other morphisms are smooth.
Then, the diagram
\[\xymatrix{{i'^*_\sM\circ f^*_\sM\circ g^*_\sM}\ar[r]\ar[d]^-{\simeq} &{f'^*_\sM\circ i^*_\sM\circ g^*_\sM}\ar[r] &{f'^*_\sM\circ g'^*_\sM}\ar[d]^-{\simeq}\\
{i'^*_\sM\circ h^*_\sM}\ar[rr] &{} &{h'^*_\sM}}\]
is commutative.
\end{lemm}
\begin{proof}
Let $d$ be the relative dimension of $h'$. It is enough to check that the diagram
\[\xymatrix{{i'^*_\sM\circ f^*_\sM\circ g^*_\sM[d]}\ar[r]\ar[d]^-{\simeq} &{f'^*_\sM\circ i^*_\sM\circ g^*_\sM[d]}\ar[r] &{f'^*_\sM\circ g'^*_\sM[d]}\ar[d]^-{\simeq}\\
{i'^*_\sM\circ h^*_\sM[d]}\ar[rr] &{} &{h'^*_\sM[d]}}\]
is commutative. Since all functors in this diagram are dg enhanced and $t$-exact for the classical $t$-structures, by \ref{prop:Vologodsky} it suffices to check the commutativity of the diagram induced on the hearts. This can be checked on the underlying perverse sheaves.\end{proof}

\begin{prop}
The $2$-isomorphisms \eqref{eq:Glue2iso} define an exchange structure i.e., they are compatible with the horizontal and vertical compositions of commutative squares.
\end{prop}

\begin{proof}

$\bullet$ \underline{Horizontal composition of squares}. Consider a commutative diagram
\begin{equation}\label{eq:HorDia}
\xymatrix{{Z'}\ar[r]^-{s'}\ar[d]^-{f''} &{X'}\ar[r]^-{i'}\ar[d]^-{f'} &{Y'}\ar[d]^-f\\
{Z}\ar[r]^-s &{X}\ar[r]^-{i} &{Y}}
\end{equation}
in which $i,s,i',s'$ are closed immersions and $f,f',f''$ are smooth morphisms. We have to prove that the diagram
\[\xymatrix{{(i'\circ s')^*_\sM\circ f^*_\sM}\ar[d]^-{\simeq}\ar[rr] &{} &{f''^*_\sM\circ (i\circ s)^*_\sM}\ar[d]^-{\simeq}\\
{s'^*_\sM\circ i'^*_\sM\circ f^*_\sM}\ar[r] &{s'^*_\sM\circ f'^*_\sM\circ i^*_\sM}\ar[r] &{f''^*_\sM\circ s^*_\sM\circ i^*_\sM}}\]
is commutative. Let us decompose \eqref{eq:HorDia} the following ways:
\begin{equation}\label{eq:DiaHorComp}
\xymatrix{{Z'}\ar[r]\ar[rd] &{Z\times_XX'}\ar[r]\ar[d]\ar@{}[rd]|{\square} &{X'}\ar[r]\ar[d] &{X\times_YY'}\ar[r]\ar[ld]\ar@{}[d]|{\square} &{Y'}\ar[ld]\\
{} &{Z}\ar[r] &{X}\ar[r] &{Y} &{}}
\end{equation}
and 
\[\xymatrix{{Z'}\ar[r]\ar[rd] &{Z\times_YY'}\ar[d]\ar[r]\ar@{}[rd]|{\square} &{Y'}\ar[d]\\
{} &{Z}\ar[r] &{Y.}}\]
Since $Z\times_X(X\times_YY')=Z\times_YY'$ we can rewrite the portion
\[\xymatrix{{Z\times_XX'}\ar[r]\ar[d]\ar@{}[rd]|{\square} &{X'}\ar[r]\ar[d] &{X\times_YY'}\ar[ld]\\
{Z}\ar[r] &{X} &{}}\]
of the diagram \eqref{eq:DiaHorComp} as
\[\xymatrix{{Z\times_XX'}\ar[r]\ar[rd] &{Z\times_YY'}\ar[r]\ar[d]\ar@{}[rd]|{\square} &{X\times_YY'}\ar[d]\\
{} &{Z}\ar[r] &{X.}}\]
Therefore the desired compatibility is a consequence of \ref{prop:ExStructure}, \ref{lemm:SquareTriangleA} and \ref{lemm:CompositionTriangle}.

\par\bigskip

$\bullet$ \underline{Vertical composition of squares}. Consider a commutative diagram
\begin{equation}\label{eq:DiaVerComp}
\xymatrix{{X''}\ar[r]^-{i''}\ar[d]^-{g'} &{Y''}\ar[d]^-g\\
{X'}\ar[r]^-{i'}\ar[d]^-{f'} &{Y'}\ar[d]^-f\\
{X}\ar[r]^i &{Y}}
\end{equation}
in which $i,i',i''$ are closed immersions and $f,g,f',g'$ are smooth morphisms. We have to prove that the diagram
\[\xymatrix{{i''^*_\sM\circ (f\circ g)^*_\sM}\ar[rr]\ar[d] & & {(f'\circ g')^*_\sM\circ i^*_\sM}\ar[d]^-{\simeq}\\
{i''^*_\sM\circ g^*_\sM\circ f^*_\sM}\ar[r] &{g'^*_\sM\circ i'^*_\sM\circ f^*_\sM}\ar[r] &{g'^*_\sM\circ f'^*_\sM\circ i^*_\sM}}\]
is commutative. We can refine \eqref{eq:DiaVerComp} into the following commutative diagrams
\[\xymatrix{{X''}\ar[r]\ar@/_2em/[rd] &{X'\times_{Y'}Y''}\ar[r]\ar[d]\ar@{}[rd]|{\square} &{X\times_YY''}\ar[r]\ar[d]\ar@{}[rd]|{\square} &{Y''}\ar[d]\\
{} &{X'}\ar[r]\ar@/_2em/[rd] &{X\times_YY'}\ar[r]\ar[d]\ar@{}[rd]|{\square} &{Y'}\ar[d]\\
{} &{} &{X}\ar[r] &{Y}}\]
or 
\[\xymatrix{{X''}\ar[r]\ar@/_2em/[rd] &{X\times_YY''}\ar[r]\ar[d]\ar@{}[rd]|{\square} &{Y''}\ar[d]\\
{} &{X}\ar[r] &{Y.}}\]
The desired compatibility is now a consequence of \ref{prop:ExStructure}, \ref{lemm:SquareTriangleB} and \ref{lemm:CompositionTriangle}.
\end{proof}

\section{Main theorem}\label{sec:MainTheo}
In \ref{subsec:AppPerverseMotive}, we have shown that the unipotent nearby and vanishing cycles functors can be defined at the level of perverse Nori motives.

Our goal is to prove that the four operations \eqref{Eq:FourFunDbc} can be lifted to the derived categories of perverse Nori motives.
To obtain these various functors
\begin{equation}\label{eq:FuncDMFormalism}
\xymatrix{{\Db(\sM(X))}\ar@<-1ex>[r]_-{f^{\sM}_*} & {\Db(\sM(Y))}\ar@<-1ex>[l]_-{f_{\sM}^*}\ar@<1ex>[r]^-{f_{\sM}^!} & {\Db(\sM(X))}\ar@<1ex>[l]^-{f_!^{\sM}} }
\end{equation}
(and their compatibility relations) with the least amount of effort, we have chosen to follow Ayoub's approach developed in \cite{AyoubI} around the notion 
of stable homotopical $2$-functor, which encompasses in a small package all the ingredients needed to build the rest of the formalism.

\subsection{Statement of the theorem}
\label{subsec:statement_main_theorem}

As before, $(\Sch/k)$ denotes the category of quasi-projective $k$-varieties.  Recall that a contravariant $2$-functor
\[\mathsf{H}^*:(\Sch/k)\ra \mathfrak{TR}\]
is a called a stable homotopical $2$-functor (see \cite[D\'efinition 1.4.1]{AyoubI}) when the following six properties are satisfied. 
\begin{enumerate}
\item $\mathsf{H}(\varnothing)=0$ (that is, $\mathsf{H}(\varnothing)$ is the trivial triangulated category).
\item For every morphism $f:X\ra Y$ in $(\Sch/k)$, the functor $f^*:\mathsf{H}(Y)\ra\mathsf{H}(X)$ admits a right adjoint. Furthermore for every immersion $i$ the counit $i^*i_*\ra\Id$ is invertible. 
\item  For every smooth morphism $f:X\ra Y$ in $(\Sch/k)$, the functor $f^*:\mathsf{H}(Y)\ra\mathsf{H}(X)$ admits a left adjoint $f_\sharp$. Furthermore, for every cartesian square
\[\xymatrix{{X'}\ar[r]^-{g'}\ar[d]^-{f'} & {X}\ar[d]^-{f}\\
{Y'}\ar[r]^-{g} & {Y}}\]
with $f$ smooth, the exchange $2$-morphism $f'_\sharp g'^*\ra g^*f_\sharp$ is invertible. 
\item If $j:U\ra X$ is an open immersion in $(\Sch/k)$ and $i:Z\ra X$ is the closed immersion of the complement, then the pair $(j^*,i^*)$ is conservative.
\item If $p:\bbA^1_X\ra X$ is the canonical projection, then the unit morphism $\Id\ra p_*p^*$ is invertible.
\item If $s$ is the zero section of the canonical projection $p:\bbA^1_X\ra X$, then $p_\sharp s_*:\mathsf{H}(X)\ra\mathsf{H}(X)$ is an equivalence of categories.
\end{enumerate}

The main theorem of \cite{AyoubI} says that these data can be expanded into a complete formalism of the four operations (see \cite[Scholie 1.4.2]{AyoubI}).

\begin{theo}\label{theo:maintheo}
The contravariant $2$-functor $\mathsf{H}^*_\sM$ constructed in \ref{sec:pullback}
is a stable homotopical $2$-functor is the sense of \cite[D\'efinition 1.4.1]{AyoubI}, and $(\rat^\sM,\theta^\sM)$ is a morphism of stable homotopical $2$-functors.
\end{theo}

In particular, we can apply \cite[Scholie 1.4.2]{AyoubI} to get the functors \eqref{eq:FuncDMFormalism}. The next subsection is devoted to the proof of \ref{theo:maintheo}, and the reader will find some applications of the main theorem in \ref{subsec:consequence}.

\subsection{Proof of the main theorem (\ref{theo:maintheo})}

We start by showing the existence of the direct image functor. The most important step is the proof of the existence of the direct image by the projection of the affine line $\bbA^1_Y$ onto its base $Y$.

\begin{prop}\label{prop:rightadjoint}
For every morphism $f:X\ra Y$ in $(\Sch/k)$, the functor 
\[f^*_\sM:\Db(\sM(Y))\ra\Db(\sM(X))\]
admits a right adjoint $f_*^\sM$. Moreover
\begin{enumerate}
\item if $i:Z\ra X$ is a closed immersion, the counit of the adjunction $i^*_\sM i_*^\sM\ra \Id$ is invertible;
\item the natural transformation
\[\gamma_f^\sM:\rat^\sM_Y f^\sM_*\ra f^\sP_*\rat^\sM_X,\]
obtained from $\theta^\sM_f$ by adjunction, is invertible;
\item if $p:\bbA^1_X\ra X$ is the canonical projection, then the unit morphism $\Id\ra p^\sM_*p^*_\sM$ is invertible.
\end{enumerate}
\end{prop}

\begin{proof}
In the proof, all products are fiber products over the base field $k$ and $\bbA^1$ is the affine line over $k$.
\bigskip

\underline{Step 1}:
Suppose first that $f$ is a closed immersion. Then
$f^*_\sM$ admits $f_*^\sM$ as a right adjoint by construction of $f^*_\sM$,
we know point (2) by \ref{lemma_theta_i},
and point (1) is true by (2) and by conservativity of
$\rat_X^\sM$.

\bigskip

\underline{Step 2}:  
Now we consider the case where $f$ is the
projection morphism $p:X:=\bbA^1_Y\ra Y$. As before, if we can
prove that $p^*_\sM$ admits a right adjoint satisfying (2),
then point
(3) will follow automatically.

We consider the following commutative diagram:
\[\xymatrix{\bbA^1\times Y\ar[d]_p & \bbA^1\times \bbA^1\times Y 
\ar[l]_-{q_2}\ar[d]^(.3){q_1}& U\times Y\ar[l]^-j \\
Y & \bbA^1\times Y\ar[l]^-p & \bbA^1\times Y\ar[l]^-{\Id}\ar[lu]_-{i}\ar[llu]^
(.7){\Id}|!{[l];[lu]}\hole}\]
where $q_1=\Id_{\bbA^1}\times p$, $q_2$ is the product of the projection $\bbA^1\ra\Spec k$  and of $\Id_{\bbA^1\times Y}$,
$i$ is the product of the diagonal morphism of $\bbA^1$ and of
$\Id_Y$, and $j$ is the complementary open inclusion.
We also denote by $s:Y\ra\bbA^1\times Y$ the zero section of $p$.
By the smooth base change theorem (or a direct calculation), the base
change map $p^*_\sP p_*^\sP\ra q_{1*}^\sP q_{2\sP}^*$ is an isomorphism,
so we get a functorial isomorphism $p_*^\sP\simeq s^*_\sP
p^*_\sP p_*^\sP\ra s_\sP^* q_{1*}^\sP q_{2\sP}^*$.

Let $K$ be a perverse sheaf on $Y$. Then $L:=q_{2\sP}^*K[1]$ is perverse, and we
have $i_\sP^*L=K[1]$, so we get an exact sequence of perverse sheaves on
$\bbA^1\times Y$:
\[0\ra i^\sP_*K\ra j^\sP_!j_\sP^* L\ra L\ra 0.\]
Applying the functor $q_{1*}^\sP$ and using the fact that $q_1\circ i=\Id_{\bbA^1\times Y}$, we get an exact triangle:
\[q_{1*}^\sP q_{2\sP}^* K\ra K\ra
q_{1*}^\sP j_!^\sP j^*_\sP L\xrightarrow{+1}.\]
We claim that $q_{1*}^\sP j_!^\sP j^*_\sP L$ is perverse. Indeed, this
complex is concentrated in perverse degrees $-1$ and $0$ by
\cite[4.1.1 \& 4.2.4]{BBD}. So we just need to prove
that  $M:=\pHc^{-1}
q_{1*}^\sP j_!^\sP j^*_\sP L$ is equal to $0$. 
By \cite[4.2.6.2]{BBD}, the adjunction
morphism $q_{1\sP}^*M[1]\ra j_!^\sP j^*_\sP L$ is injective; we denote its quotient
by $N$. Then, as $q_1\circ i=\Id_{\bbA^1\times Y}$ and $i^*_\sP j_!^\sP=0$,
we have $i^*_\sP N=M[2]$. But $i$ is the complement of an open affine
embedding, so $i^*_\sP$ is of perverse cohomological amplitude
$[-1,0]$ by \cite[4.1.10]{BBD}, hence $M=0$.

Finally, we get an exact sequence of perverse sheaves on $\bbA^1\times Y$:
\[0\ra \pHc^0 q_{1*}^\sP q_{2\sP}^* K\ra K\ra
q_{1*}^\sP j_!^\sP j^*_\sP q_{2\sP}^*K[1]\ra \pHc^1 q_{1*}^\sP q_{2\sP}^* K\ra 0.\]

Consider the functors $\mathsf{F}_\sP,\mathsf{G}_\sP:
\sP(\bbA^1\times Y)\ra\Db(\sP(\bbA^1\times Y))$ defined by
 \[\mathsf{F}_\sP(K):=K\]
and 
\[\mathsf{G}_\sP(K):=q_{1*}^\sP j_!^\sP j^*_\sP q_{2\sP}^*K[1].\]
We have just proved that these functors are $t$-exact (of course, this is
obvious for the first one) and that there is a functorial exact
triangle
\[q_{1*}^\sP q_{2\sP}^*\ra\mathsf{F}_\sP\ra\mathsf{G}_\sP\xrightarrow{+1}.\]
The functors $\mathsf{F}_\sP$ and $\mathsf{G}_\sP$ are defined in terms of the four operations. The existence of these operations in the categories 
$\DA_\ct(-)$ and the compatibility of the Betti realization with the four operations (see \cite[Th\'eor\`eme 3.19]{AyoubBetti}), imply by the universal property of the categories of perverse motives that there exist:
 \begin{itemize}
\item two exact functors
 \[\mathsf{F}_\sM,\mathsf{G}_\sM:\sM(\bbA^1\times Y)\ra\sM(\bbA^1\times Y),\]
\item a  natural transformation $\mathsf{F}_\sM\ra\mathsf{G}_\sM$, and
\item two invertible natural transformations $\rat^\sM_{\bbA^1\times
Y}\circ \mathsf{F}_\sM\ra \mathsf{F}_\sP\circ \rat^\sM_{\bbA^1\times Y}$
and $\rat^\sM_{\bbA^1\times Y}\circ \mathsf{G}_\sM\ra \mathsf{G}_\sP\circ 
\rat^\sM_{\bbA^1\times Y} $
\end{itemize}
 such that the diagram 
 \[\xymatrix{{\rat^\sM_{\bbA^1\times Y}\circ \mathsf{F}_\sM}\ar[r]\ar[d] & 
{\rat^\sM_{\bbA^1\times Y}\circ \mathsf{G}_\sM}\ar[d]\\
 {\mathsf{F}_\sP\circ\rat^\sM_{\bbA^1\times Y}}\ar[r] & 
{\mathsf{G}_\sP\circ\rat^\sM_{\bbA^1\times Y}}}\]
 is commutative.

 Given a complex $M^\bullet$ of perverse motives on $X=\bbA^1\times Y$, let
$\mathsf{H}_\sM(M^\bullet)$ be the mapping fiber of the morphism
$\mathsf{F}_\sM(M^\bullet)\ra\mathsf{G}_\sM(M^\bullet)$ of complexes 
of perverse motives on $X$. We get a triangulated functor
 \[\mathsf{H}_\sM:\Db(\sM(\bbA^1\times Y))\ra\Db(\sM(\bbA^1\times Y)),\] 
and
the Betti realization of $\mathsf{H}_\sM$ is isomorphic to
$q_{1*}^\sP q_2^{*\sP}$.

We now define a functor 
\[p^\sM_\bullet:=s^*_\sM\mathsf{H}_\sM(-):\Db(\sM(\bbA^1\times Y))
\ra\Db(\sM(Y)).\]
By construction of $p^\sM_\bullet$, we have an invertible natural
transformation 
\[\rat^\sM_{\bbA^1\times Y} p_\bullet^\sM\ra p_*^\sP
\rat^\sM_{\bbA^1\times Y}.\]

Note also the following useful fact. We denote by $f:\bbA^1\times Y\ra
\bbA^1$ the first projection and by $a:\bbG_m\times Y\ra
\bbA^1\times Y$ the inclusion. Then applying $s^*_\sM$ to the
connecting map in the exact sequence \eqref{eq:shortexNC}
in \ref{subsec:AppPerverseMotive}, we get a natural transformation
\[s^*_\sM\ra{}^p\mathsf{Log}^{\sM}_f a^*_\sM[1],\]
whose composition with the functor $\mathsf{H}_\sM$ is invertible. Indeed,
we can check this last statement after applying the functors
$\rat^\sM_Y$, and then this follows from the exact triangle
$q_{1*}^\sP q_2^{*\sP}\ra\mathsf{F}_\sP\ra\mathsf{G}_\sP\xrightarrow{+1}$
and the fact that the composition of the natural transformation
$s^*_\sP\ra{}^p\mathsf{Log}^{\sP}_f a^*_\sP[1]$ and of the functor
$q_{1*}^\sP q_2^{*\sP}\simeq p^*_\sP p_*^\sP\simeq\bbQ_{\bbA^1}
\boxtimes p_*^\sP$ is invertible. As the functor
${}^p\mathsf{Log}^{\sM}_f a^*_\sM$ is exact, we get an
isomorphism from $p_\bullet^\sM$
to the mapping cone of the morphism of exact
functors ${}^p\mathsf{Log}^{\sM}_f a^*_\sM\circ\mathsf{F}_\sM\ra
{}^p\mathsf{Log}^{\sM}_f a^*_\sM\circ\mathsf{G}_\sM$.

Let us prove that the functor $p_\bullet^\sM$
is right adjoint to the functor $p^*_\sM$.
Let $\eta_\sP:\Id\ra p^\sP_*p_\sP^*$ and $\delta_\sP: p_\sP^* p_*^\sP\ra
\Id$
be the unit and the counit
of the adjunction between $p^*_\sP$ and $p_\sP^*$. 
It suffices to lift $\eta_\sP$ and $\delta_\sP$ to natural
transformations
$\eta_\sM:\Id\ra p^\sM_\bullet p_\sM^*$ and $\delta_\sM: p_\sM^* p_\bullet
^\sM\ra\Id$ such that the two natural transformations
\[p^*_\sM\xra{p^*_\sM\eta_\sM}p^*_\sM p_\bullet^\sM p^*_\sM
\xra{\delta_\sM p^*_\sM}p^*_\sM\]
and
\[p_\bullet^\sM\xra{\eta_\sM p_\bullet^\sM} p_\bullet^\sM p^*_\sM p_\bullet^\sM
\xra{p_\bullet^\sM\delta_\sM}
p_\bullet^\sM\]
are isomorphisms and the first one is the identity (see Section 3.1 of M. Saito's~\cite{SaitoMixedSheaves}).
Note that the fact that these natural transformations are isomorphisms
will follow automatically from the conservativity of the functors
$\rat_X^\sM$.

We first construct $\eta_\sM$. Let us first show that $\mathsf{G}_\sM\circ p^*_\sM=0$.
As the functors $\rat_X^\sM$ are conservative, it suffices to prove that
$\mathsf{G}_\sP\circ p^*_\sP=0$. Let $k:U\ra\bbA^1\times\bbA^1$ be the open immersion
(remember that $U$ is the complement of the diagonal in $\bbA^1\times
\bbA^1$), so that $j=k\times\Id_Y$, and let $\pi:\bbA^1\times\bbA^1\ra
\bbA^1$ be the first projection, so that $q_1=\pi\times\Id_Y$. Then
\[\mathsf{G}_\sP\circ p^*_\sP=q_{1*}^\sP j_!^\sP j^*_\sP q_{2\sP}^*p^*_\sM[1]\simeq
q_{1*}^\sP((k_!^\sP\bbQ_{U})\boxtimes -)[1]
\simeq(\pi^\sP_*k^\sP_!\bbQ_{U})\boxtimes (-)[1],\]
so it suffices to show that
\[\pi_*^\sP k_!^\sP\bbQ_{U}=0.\]
Let $\Delta:\bbA^1\ra\bbA^1\times\bbA^1$ be the diagonal embedding. Then we
have an exact triangle
\[k_!\bbQ_U\ra\bbQ_{\bbA^1\times\bbA^1}\ra
\Delta_*\bbQ_{\bbA^1}\xrightarrow{+1},\]
so, applying $\pi_*^\sP$, we get an exact triangle
\[\pi_*^\sP k_!^\sP\bbQ_{U}\ra\bbQ_{\bbA^1}\xrightarrow
{\Id}
\bbQ_{\bbA^1}\xrightarrow{+1},\]
and this implies the desired result.

Now that we know that $\mathsf{G}_\sP\circ p^*_\sM=0$, we get $\mathsf{H}_\sM\circ p^*_\sM=
p^*_\sM$, hence $p_\bullet^\sM p^*_\sM=s^*_\sM\circ \mathsf{H}_\sM\circ p^*_\sM=
s^*_\sM p^*_\sM$, and we take for $\eta_\sM:\Id\ra p_\bullet^\sM p^*_\sM$ the
inverse of the
connection isomorphism $s^*_\sM p^*_\sM\xra{\sim}\Id$.

Next we construct $\delta_\sM$. First we define a functor
$q_{1\bullet}^\sM:\Db(\sM(\bbA^1\times Y))\ra\Db(\sM(\bbA^1\times\bbA^1
\times Y))$ in the same way as $p_\bullet^\sM$. That is, we consider the
commutative diagram
\[\xymatrix{\bbA^1\times\bbA^1\times Y\ar[d]_{q_1} & \bbA^1\times\bbA^1
\times \bbA^1\times Y 
\ar[l]_-{r_2}\ar[d]^(.3){r_1}& \bbA^1\times U\times Y\ar[l]^-J \\
\bbA^1\times Y & \bbA^1\times\bbA^1\times Y\ar[l]^-{q_1} & \bbA^1\times
\bbA^1\times Y\ar[l]^-{\Id}\ar[lu]_-{I}\ar[llu]^
(.7){\Id}|!{[l];[lu]}\hole}\]
where $r_1=\Id_{\bbA^1}\times q_1$, $r_2=\Id_{\bbA^1}\times q_2$,
$I=\Id_{\bbA^1}\times i$ and $J=\Id_{\bbA^1}\times j$, and
we set $t=\Id_{\bbA^1}\times s:\bbA^1\times Y\ra\bbA^1\times\bbA^1\times Y$.
Then the functors $\mathsf{F}'_\sP,\mathsf{G}'_\sP$ from $\Db(\sP(\bbA^1\times\bbA^1\times Y))$ to
itself defined by $\mathsf{F}'_\sP=\Id$ and $\mathsf{G}'_\sP=r_{1*}^\sP J_!^\sP J^*_\sP
r_{2\sP}^*[1]$ are $t$-exact and
we have a natural transformation $\mathsf{F}'_\sP\ra \mathsf{G}'_\sP$. As before, we can lift
these functors and transformation to endofunctors $\mathsf{F}'_\sM\ra \mathsf{G}'_\sM$
of $\Db(\sM(\bbA^1\times\bbA^1\times Y))$. We denote by $\mathsf{H}'_\sM$ the
mapping fiber of $\mathsf{F}'_\sM\ra \mathsf{G}'_\sM$, and we set $q_{1\bullet}^\sM=t^*_\sM\circ
\mathsf{H}'_\sM$.
Also, if we denote by $f':\bbA^1\times\bbA^1\times Y\ra\bbA^1$ the second
projection and by $a'$ the injection of
$\bbA^1\times\bbG_m\times Y$ into $\bbA^1\times\bbA^1\times Y$, we get as
above an invertible natural transformation from $q_{1\bullet}^\sM$
to the mapping cone of the morphism of exact functors
${}^p\mathsf{Log}^{\sM}_{f'} (a')^*_\sM\circ\mathsf{F}'_\sM\ra
{}^p\mathsf{Log}^{\sM}_{f'} (a')^*_\sM\circ\mathsf{G}'_\sM$.

Let's show that the base change isomorphism $p^*_\sP p_*^\sP\xra{\sim} q_{1*}
^\sP q_{2\sP}^*$ lifts to a morphism 
$p^*_\sM p_\bullet^\sM\ra q_{1\bullet}
^\sM q_{2\sM}^*$ (which will automatically be an isomorphism).
We have invertible natural transformations 
$\mathsf{F}'_\sP\circ q_{2\sP}^*\simeq q_{2\sP}^*\circ \mathsf{F}_\sP$ and
$\mathsf{G}'_\sP\circ q_{2\sP}^*\simeq q_{2\sP}^*\circ \mathsf{G}_\sP$. As all the functors
involved are $t$-exact up to the same shift, the transformations lift to
natural transformations $\mathsf{F}'_\sM\circ q_{2\sM}^*\simeq q_{2\sM}^*\circ\mathsf{F_\sM}$ and
$\mathsf{G}'_\sM\circ q_{2\sM}^*\simeq q_{2\sM}^*\circ\mathsf{G}_\sM$, and induce an
invertible natural transformation $\mathsf{H}'_\sM\circ q_{2\sM}^*\simeq q_{2\sM}^*\circ
\mathsf{H}_\sM$. Composing on the left with $t_\sM^*$ and using the connection
isomorphism $t_\sM^* q_{2\sM}^*\simeq p_\sM^* s_\sM^*$, we get
the desired isomorphism $p^*_\sM p_\bullet^\sM\xra{\sim} q_{1\bullet}
^\sM q_{2\sM}^*$.

Composing this isomorphism with the unit of the adjunction
$(i^*_\sM,i_*^\sM)$ and using the connection isomorphism $i^*_\sM q_{2\sM}^*
\simeq\Id$ gives a natural transformation $p^*_\sM p_\bullet^\sM\ra
q_{1\bullet}^\sM i_*^\sM$. It remains to show that the isomorphism
$q_{1*}^\sP i_*^\sP\simeq\Id$ lifts to a natural transformation
$q_{1\bullet}^\sM i_*^\sM\ra\Id$.
First we note that the functors
\begin{equation}\label{eq:funcH''}
{}^p\mathsf{Log}^{\mathsf{B}}_{f'}(a')^*_\sP r_{1*}^\sP r_2^{*\sP}i_*^\sP[1]
\end{equation}
and
\[{}^p\mathsf{Log}^{\mathsf{B}}_{f'}(a')^*_\sP r_{1*}^\sP J_!^\sP J^*_\sP
 r_2^{*\sP}i_*^\sP[1]\]
are $t$-exact and the counit of the adjunction $(J_!^\sP,J^*_\sP)$ induces a natural transformation from the second one to the first one. Hence, the functor \eqref{eq:funcH''}
 induces an exact
endofunctor $\mathsf{H}''_\sM$of $\sM(\bbA^1\times Y)$,
together with a natural transformation 
${}^p\mathsf{Log}^\sM (a')^*_\sM
\circ \mathsf{G}'_\sM\circ i_*^\sM\ra \mathsf{H}''_\sM$. 
But we also have an invertible natural transformation
of $t$-exact functors
\begin{align*}
\Id_{\Db(\sP(\bbA^1\times Y))} & \simeq t^*_\sP q_1^{*\sP}q_{1*}^\sP i_*^\sP
\mbox{ (connection isomorphisms)} \\
&\xra{\sim} t^*_\sP r_{1*}^\sP r_{2\sP}^* i_*^\sP\mbox{ (base change)}\\
& \xra{\sim} {}^p\mathsf{Log}^{\mathsf{B}}_{f'}(a')^*_\sP r_{1*}^\sP r_{2\sP}^*
i_*^\sP[1]\mbox{ (by \ref{subsec:AppPerverseMotive} as above)}
\end{align*}
and all the maps in it are defined in the categories 
$\DA_\ct(-)$, so it induces an invertible natural transformation
$\Id_{\sM(\bbA^1\times Y)}\xra{\sim}\mathsf{H}''_\sM$. Composing it with
${}^p\mathsf{Log}^\sM (a')^*_\sM
\circ\mathsf{G}'_\sM\circ i_*^\sM\ra\mathsf{H}''_\sM$
and using the isomorphism from $q_{1\bullet}^\sM$ to the mapping fiber of
${}^p\mathsf{Log}^\sM (a')^*_\sM\circ\mathsf{F}'_\sM\ra
{}^p\mathsf{Log}^\sM (a')^*_\sM\circ\mathsf{G}'_\sM$,
we finally get the desired natural
transformation $\delta_\sM$:
\[p^*_\sM p_\bullet^\sM\xra{\sim}q_{1\bullet}^\sM q_2^{*\sM}
\ra q_{1\bullet}^\sM i_*^\sM\ra {}^p\mathsf{Log}^\sM (a')^*_\sM
\circ\mathsf{G}'_\sM\circ i_*^\sM\ra\mathsf{H}''_\sM\simeq\Id_{\Db(\sM(\bbA^1\times Y))}.\]

Finally, we check that the natural transformation
\[p^*_\sM\xra{p^*_\sM\eta_\sM}p^*_\sM p_\bullet^\sM p^*_\sM
\xra{\delta_\sM p^*_\sM}p^*_\sM\]
is the identity.
The two functors $p^*_\sM p_\bullet^\sM p^*_\sM[1]$ and $p^*_\sM[1]$ are
exact and equal to the derived functor of their $\mathbf{H}^0$, and the
natural transformations 
\[(p^*_\sM\eta_\sM)^{-1},\delta_\sM p^*_\sM:
p^*_\sM p_\bullet^\sM p^*_\sM[1]\ra p^*_\sM[1]\] 
are also defined by
extending their action on the $\mathbf{H}^0$'s, so it suffices to
check that they are equal on these $\mathbf{H}^0$'s. But this follows
from the analogous result for the category of perverse sheaves.

\bigskip

\underline{Step 3}: We can now use the Brown Representability Theorem to see that the proposition is true more generally if $f$ is the projection $p:X:=E\ra Y$ of a vector bundle $E$ on $Y$. 
Indeed, given a $k$-variety $S$, let $\Ind(\mathscr M(S))$ be the abelian category of $\Ind$-objects of $\mathscr M(S)$ and consider the bounded derived category $\Db(\mathscr M(S))$ as a full subcategory of the unbounded derived category $\D(\Ind(\mathscr M(S)))$ (see e.g. Theorem~15.3.1 of Kashiwara--Schapira's book~\cite{MR2182076}). As the morphism $p:E\ra Y$ is smooth, the functor $p^*_\mathscr M$ extends to a triangulated functor $\mathsf{L}:\D(\Ind(\mathscr M(Y)))\ra\D(\Ind(\mathscr M(E)))$. By the Brown Representability Theorem (see e.g. Theorem~4.1 of Neeman's article~\cite{MR1308405} or the 
book~\cite{MR1812507} by the same author), the functor $\mathsf{L}$ admits a right adjoint $\mathsf{R}:\D(\Ind(\mathscr M(E))\ra\D(\Ind(\mathscr M(Y)))$. 

To prove that $p^*_\mathscr M$ admits a right adjoint $p^\mathscr M_*$, it suffices to check that, given $M\in\Db(\mathscr M(E))$, the object $\mathsf{R}(M)$ belongs to the subcategory $\Db(\mathscr M(Y))$. This can be checked on a finite Zariski open covering of $Y$ that trivializes $E$ and thus follows from the case of a projection $\bbA^1_Y\ra Y$ proved in step 2. That $p^\mathscr M_*$ satisfies (2) can again be checked on a finite Zariski open covering of $Y$ that trivializes $E$ and we conclude using step 2.

\bigskip

\underline{Step 4}: 
By steps 1 and 3, the proposition
is true if $f$ is an affine
morphism. Indeed, if $f$ is affine, then we can write $f=p\circ i$, where $i$ is
a closed immersion and $p:E\ra Y$ is a vector bundle on $Y$.
\bigskip

\underline{Step 5}:
We now consider the case of an arbitrary morphism $f:X\ra Y$ in
$(\Sch/k)$. By Jouanolou's trick (cf. Jouanolou's paper~\cite{Jouanolou}),
there exists a vector bundle $E\ra X$ and an affine
$E$-torsor $p:\widetilde{X}\ra X$. As $p$ is affine, we know the
proposition for $p$ by step 3. Moreover, the unit 
$\Id\ra p_*^\sM p^*_\sM$ is an isomorphism; indeed, it suffices to show
this after restricting to an open covering of $X$, so we may assume that
the morphism $p$ is isomorphic to the second projection $\bbA^n\times X\ra
X$, and then the result follows from point (3) of the proposition.
As the unit of the adjunction $(p^*_\sM,p_*^\sM)$ is an isomorphism, the
left adjoint $p^*_\sM$ is fully faithful. 

Let $g=f\circ p$. As
$\widetilde{X}$ is affine, the morphism $g$ is affine. Also, we show
as before that the unit $\Id\ra p_*^\sP p^*_\sP$ is an isomorphism, so
we get an isomorphism $f_*^\sP\iso f_*^\sP p_*^\sP p^*_\sP\simeq
g_*^\sP p^*_\sP$. We set $f_*^\sM=g_*^\sM p^*_\sM$; by the calculation we 
just did, this satisfies condition (2). It remains to show that
$f_*^\sM$ is right adjoint to $f^*_\sM$. Let $K\in\Ob\Db\sM(Y)$ and
$L\in\Ob\Db\sM(X)$. Then we have isomorphisms
\begin{align*}
\Hom_{\Db\sM(Y)}(K,f_*^\sM L)=\Hom_{\Db\sM(Y)}(K,g_*^\sM p^*_\sM L)
&\simeq\Hom_{\Db\sM(\widetilde{X})}(g^*_\sM K,p^*_\sM L)\\
&\simeq\Hom_{\Db\sM(\widetilde{X})}(p^*_\sM f^*_\sM K,p^*_\sM L),
\end{align*}
and the last group is isomorphic to $\Hom_{\Db\sM(X)}(f^*_\sM K,L)$ by the
full faithfulness of $p^*_\sM$.
\end{proof}

\begin{prop}\label{prop:leftadjoint}
For every smooth morphism $f:X\ra Y$ in $(\Sch/k)$ the functor 
\[f^*_\sM:\Db(\sM(Y))\ra\Db(\sM(X))\]
admits a left adjoint $f_\sharp$. Moreover 
\begin{enumerate}
\item  the natural transformation
\[f^\sP_\sharp\rat^\sM_X \ra \rat^\sM_Yf^\sM_\sharp\]
obtained from $\theta^\sM_f$ by adjunction, is invertible;
\item for every cartesian square
\[\xymatrix{{X'}\ar[r]^-{g'}\ar[d]^-{f'} & {X}\ar[d]^-{f}\\
{Y'}\ar[r]^-{g} & {Y}}\]
with $f$ smooth, the exchange $2$-morphism $f'^\sM_\sharp g'^*_\sM\ra g^*_\sM f^\sM_\sharp$ is invertible. 
\end{enumerate}
\end{prop}

\begin{proof}
The assertion (2) is an immediate consequence of (1) since the functor $\rat^\sM_Y$ is conservative. 

We deduce the proposition from \ref{prop:rightadjoint} using Verdier duality. Let $f:X\ra Y$ be a smooth morphisms of relative dimension $d$.  Note that $f^*_\sP$ has a left adjoint  given by 
\[f_\sharp^\sP:=\bbD^\sP_Y f^\sP_*(-d)[-2d]\bbD^\sP_X.\] 
Therefore we similarly set $f_\sharp^\sM:=\bbD^\sM_Y f^\sM_* (-d)[-2d]\bbD^\sM_X$.

Let $A$ be an object in $\Db(\sM(X))$ and $B$ be an object in $\Db(\sM(Y))$. Then, \ref{prop:rightadjoint} and \ref{prop:duality} provide isomorphisms
\begin{align*}
\Hom(f_\sharp^\sM A,B)&\simeq\Hom(\bbD^\sM_Y B,f^\sM_*(-d)[-2d]\bbD^\sM_X A)\simeq\Hom(f^*_\sM(d)[2d]\bbD^\sM_YB,\bbD_X^\sM A)\\
& \simeq\Hom(A,\bbD_X^\sM f^*_\sM(d)[2d] \bbD^\sM_YB)\simeq\Hom(A,f^*_\sM B).
\end{align*}
This shows that $(f_\sharp^\sM,f_\sM^*)$ form a pair of adjoint functors.
Note that the counit ${}^\sM\delta^*_\sharp$ of the adjunction is given by the composition
\begin{equation}\label{eq:counitstarsharp}
f^\sM_\sharp f^*_\sM\xra{\varepsilon_f^\sM} \bbD_Y^\sM f^\sM_*f^*_\sM\bbD_Y^\sM\xra{{}^\sM\eta^*_*}(\bbD^\sM_Y)^2\xra{(\varepsilon_Y^\sM)^{-1}} \Id
\end{equation}
and the unit ${}^\sM\eta^*_\sharp$ by the composition
\begin{equation}\label{eq:unitstarsharp}
\Id\xra{\varepsilon^\sM_X}(\bbD^\sM_X)^2\xra{{}^\sM\delta^*_*} \bbD^\sM_X f^*_\sM f^\sM_*\bbD^\sM_X \xra{\varepsilon_f^\sM}f^*_\sM f_\sharp^\sM.
\end{equation}
To show that the morphism
\[f^\sP_\sharp\rat^\sM_X \xra{{}^\sM\eta^*_\sharp}f^\sP_\sharp\rat^\sM_Xf^*_\sM f^\sM_\sharp\xra{(\theta^\sM_f)^{-1}}f^\sP_\sharp f^*_\sP\rat^\sM_Y f^\sM_\sharp\xra{{}^\sP\delta^*_\sharp} \rat^\sM_Yf^\sM_\sharp\]
is invertible, it is enough to check that it is equal to the morphism
\[\xymatrix{{f^\sP_\sharp\rat^\sM_X}\ar[r]^-{(\nu^\sM_X)^{-1}} & {\bbD^\sP_Yf^\sP_*(-d)[-2d]\rat^\sM_X \bbD^\sM_X}\ar[d]^-{\gamma_f^\sM} & {}\\
{} &{\bbD^\sP_Y\rat^\sM_Y f^\sM_*(-d)[-2d] \bbD^\sM_X}\ar[r]^-{\nu_Y^\sM} &{\rat^\sM_Yf^\sM_\sharp}}\]
where $\gamma_f^\sM$ is the invertible natural transformation of \ref{prop:rightadjoint}. Using the expressions of ${}^\sM\delta^*_\sharp $ and ${}^\sM\eta^*_\sharp $ given in \eqref{eq:counitstarsharp} and \eqref{eq:unitstarsharp}, this follows directly from \ref{prop:duality} (2) and \ref{prop:duality} (1), which ensure that the diagram
\[\xymatrix{{\bbD_X^\sP f^*_\sP[d]\bbD_Y^\sP\rat^\sM_Y}\ar[rr]^-{\varepsilon_f^\sP}\ar[rd]^-{\varepsilon_f^\sM} & {} &{(\bbD^\sP_X)^2 f_\sP^*(d)[d]\rat^\sM_Y}\ar[d]^-{(\varepsilon^\sP_X)^{-1}}\\
{\bbD^\sP_X f_\sP^*[d]\rat^\sM_Y\bbD^\sM_Y}\ar[d]^-{\varepsilon_f^\sP}\ar[u]_-{\nu_Y^\sM} & {f^*_\sP(d)[d](\bbD^\sP_Y)^2\rat^\sM_Y}\ar[ld]_-{\nu^\sM_Y}\ar[r]^-{(\varepsilon^\sP_Y)^{-1}} &{f_\sP^*(d)[d]\rat^\sM_Y}\\
{f_\sP^*(d)[d]\bbD^\sP_Y\rat^\sM_Y\bbD^\sM_Y}\ar[rr]^-{\nu^\sM_Y} &{} & {f_\sP^*(d)[d]\rat^\sM_Y(\bbD^\sM_Y)^2}\ar[u]_-{(\varepsilon^\sM_Y)^{-1}}}\]
is commutative.

\end{proof}

The pair $(j^*_\sM, i^*_\sM)$ is conservative, since so is the pair $(j^*_\sP, i^*_\sP)$. This follows from the existence of the isomorphisms $\theta^\sM_j$, $\theta^\sM_i$ and the fact that $\rat^\sM_X$ is a conservative functor.

To finish the proof of \ref{theo:maintheo}, it remains to check that, if $s$ is the zero section of  $p:\bbA^1_X\ra X$, then $p^\sM_\sharp s^\sM_*$ is an equivalence of categories. By construction
\[p^\sM_\sharp s^\sM_*=\bbD_{X}^\sM p_*^\sM(-1)[-2]\bbD^\sM_{\bbA^1_X} s^\sM_*.\]
Note that the isomorphism $\bbD^\sP_{\bbA^1_X}s^\sP_*\simeq s^\sP_*\bbD^\sP_X$ exists in the category of (constructible) \'etale motives. Therefore,  
 the compatibility of the Betti realization with the four operations (see \cite[Th\'eor\`eme 3.19]{AyoubBetti}) implies by the universal property of the categories of perverse motives that this isomorphism lifts to an isomorphism $\bbD^\sM_{\bbA^1_X}s^\sM_*\simeq s^\sM_*\bbD^\sM_X$. As a consequence, we get an isomorphism
\[p^\sM_\sharp s^\sM_*\simeq\bbD_{X}^\sM p_*^\sM s^\sM_*(-1)[-2]\bbD^\sM_{\bbA^1_X}\simeq \Id(-1)[-2].\]
This shows that $p^\sM_\sharp s^\sM_*$ is an equivalence of categories and concludes the proof of \ref{theo:maintheo}.

\subsection{Complement to the main theorem}

The following proposition complements \ref{theo:maintheo}.

\begin{prop}
Let $f:X\ra Y$ be a morphism of quasi-projective $k$-varieties. Then, the natural transformations 
\[\xi^\sM_f:\rat^\sM_X f^!_\sM\ra f^!_\sP\rat^\sM_Y\qquad \rho^\sM_f:f^\sP_!\rat^\sM_X\ra \rat^\sM_Y f^\sM_!\]
are invertible. 
\end{prop}

\begin{proof}
By \cite[Th\'eor\`eme 3.4]{AyoubBetti}, it just remains to check that $\xi^\sM_i$ is invertible if $i:Z\hookrightarrow X$ is a closed immersion. Let $j:U\hookrightarrow X$ be the open immersion of the complement of $Z$ in $X$. Then, we have a commutative diagram
\[
\xymatrix{{i^\sP_*i^!_\sP\rat^\sM_X}\ar[r] &{\rat^\sM_X}\ar[r]\ar@{=}[d] &{j^\sP_*j^*_\sP\rat^\sM_X}\ar[r]^-{+1} &{}\\
{i^\sP_*\rat^\sM_Z i^!_\sM}\ar[r]\ar[u]_-{\xi^\sM_i} &{\rat^\sM_X}\ar[r]\ar@{=}[d] &{j^\sP_*\rat^\sM_U j^*_\sM}\ar[r]^-{+1}\ar[u]_-{(\theta_j^\sM)^{-1}} &{}\\
{\rat^\sM_Xi^\sM_*i^!_\sM}\ar[r]\ar[u]_-{\gamma^\sM_i} &{\rat^\sM_X}\ar[r] &{\rat^\sM_X j^\sM_*j^*_\sM}\ar[r]^-{+1}\ar[u]_-{\gamma_j^\sM} &{}}
\]
which implies that the image of $\xi^\sM_i$ under $i^\sM_*$ is invertible since all the other morphisms are. Therefore $\xi^\sM_i$ is also invertible.
\end{proof}

\subsection{Some consequences}\label{subsec:consequence}

In this subsection, we draw some immediate consequences of the main theorem (\ref{theo:maintheo}). 

\subsubsection*{Geometric local systems are motivic} 

A $\bbQ$-local system $\mathscr L$ on a quasi-projective $k$-variety $X$ will be called geometric if there exists a smooth proper morphism $g:Z\ra X$ such that $\mathscr L=R^ig_*\bbQ$ for some integer $i\in\bbZ$. We will say that $\mathscr L$ is motivic if there exists an object $L$ in $\Db(\mathscr M(X))$ such that $\mathscr L$ and $\rat^{\sM}_X(L)$ are isomorphic in the category $\Db(\Perv(X))$.

\begin{coro}
A geometric $\bbQ$-local system $\mathscr L$ on a quasi-projective $k$-variety $X$ is motivic. 
\end{coro}

\begin{proof}
If the local system $\mathscr L$ is geometric, there exists a smooth proper morphism $g:Z\ra X$ such that $\mathscr L=R^ig_*\bbQ_Z$ for some integer $i\in\bbZ$. Then $\mathscr L$ is the image under the functor $\rat^\sM_X$ of the perverse motive ${}^\mathrm{c}\mathrm{H}^i(g^\sM_*\bbQ^\sM_Z)$, where ${}^{\mathrm{c}}\mathrm{H}^i$ is the cohomological functor associated with the constructible $t$-structure (see below).
\end{proof}

\begin{rema}
In this remark, we  denote by $\mathrm{H}^i$ the standard cohomology functors on the category $\Db(\sM(X))$. 
Let $\mathscr L$ be a geometric $\bbQ$-local system on a smooth quasi-projective variety of (pure) dimension $d$ and choose a smooth proper morphism $g:Z\ra X$, an integer $j\in\bbZ$ such that $\mathscr L=R^j g_*\bbQ_Z$. 
As $Z$ is smooth and $g$ is proper and smooth, the constructible sheaves
$R^i g_*\bbQ_Z$ are all $\bbQ$-local systems on $X$. Hence the complexes
$(R^i g_*\bbQ_Z)[d]$ are perverse sheaves and therefore
$(R^i g_*\bbQ_Z)[d]=\pH^{d+i}g_*\bbQ_Z$ for every $i\in\bbZ$. In particular, $\mathscr L[d]=\pH^{j+d}g_*\bbQ_Z$ and it follows that  $\mathscr L[d]$ is the image under $\rat^\sM_X$ of the perverse motivic sheaf $A:=\mathrm{H}^{j+d}(g_*^\sM\bbQ^\sM_Z)$. 
\end{rema}

\subsubsection*{Intersection cohomology} 

The four operations formalism allows the definition of a motivic avatar of intersection complexes. In particular, intersection cohomology groups with coefficients in geometric systems are motivic. More precisely:

\begin{coro}\label{coro:IC}
Let $X$ be an irreducible quasi-projective $k$-variety and $\mathscr L$ be a $\bbQ$-local system on a smooth dense open subscheme of $X$. If $\mathscr L$ is motivic (in particular if $\mathscr L$ is geometric), then the intersection cohomology group $IH^i(X,\mathscr L)$, for $i\in\bbZ$, is canonically the Betti realization of a Nori motive over $k$.
\end{coro}

\begin{proof}
Let $d$ be the dimension of $X$ and $\mathscr L$ be a $\bbQ$-local system on a smooth dense open subscheme $U$ of $X$. Since $\mathscr L$ is motivic, there exists an object $L\in\Db(\mathscr M(U))$ such that $\mathscr L$ is isomorphic to $\rat^\mathscr M_U(L)$. Since $\mathscr L[d]$ is a perverse sheaf on $U$ and $\rat^{\mathscr M}_U$ is conservative, the complex $L[d]$ is a perverse motivic sheaf on $U$ that is belongs to $\mathscr M(U)$. Then, with the notation of \ref{defi:IC}, the intersection complex
\[\mathrm{IC}_X(\mathscr L):=\Img(\pH^0j^\sP_!\mathscr L[d]\ra\pH^0j_*^\sP\mathscr L[d])\]
is canonically isomorphic to the image under $\rat^\sM_X$ of the perverse motivic sheaf $j^{\mathscr M}_{!*}L[d]:=\Img(\mathrm{H}^0(j^\sM_! L[d])\ra \mathrm{H}^0(j^\sM_*L[d]))$. This implies that 
$IH^i(X,\mathscr L):={\mathbf H}^{i-d}(X,\mathrm{IC}_X(\mathscr L))$ is the Betti realization of the Nori motive $\mathrm{H}^{i-d}(\pi^\sM_*j^{\mathscr M}_{!*}L[d])$ where $\pi:X\ra\Spec k$ is the structural morphism.

\end{proof}

This shows, in particular, that intersection cohomology groups carry a natural Hodge structure. If $X$ is a smooth projective curve, and $\mathscr L$ underlies a polarizable variation of Hodge structure, then the Hodge structure on the intersection cohomology groups was constructed by Zucker in \cite[(7.2) Theorem, (11.6) Theorem]{MR534758}. In general, it follows from Saito's work on mixed Hodge modules \cite{MHMII} and a different proof has been given by de Cataldo in~\cite{MR2877437}. We consider the weights in the next section (see \ref{theo:weightIC} and \ref{coro:weightIC}).

\subsubsection*{Leray spectral sequences}

Let $f:X\ra Y$ be a morphism of quasi-projective $k$-varieties and $\mathscr L$ be a $\bbQ$-local system on $X$. Then, we can associate with it two Leray spectral sequences in Betti cohomology: the classical one
\[{\mathbf{H}}^r(Y, R^sf_*\mathscr L)\Longrightarrow H^{r+s}(X,\mathscr L)\]
and the perverse one
\[{\mathbf{H}}^r(Y, \pH^sf_*\mathscr L)\Longrightarrow H^{r+s}(X,\mathscr L).\]
The main theorem of Arapura's~\cite{MR2178703} shows that, if $\mathscr L=\bbQ_X$ is the constant local system on $X$ and the morphism $f$ is projective, then  the classical Leray sequence is motivic, that is, it is the realization of a spectral sequence in the abelian category of Nori motives over $k$ (see precisely \cite[Theorem 3.1]{MR2178703}).

This property is still true without the projectivity assumption and also more generally if the local system $\mathscr L$ is geometric:

\begin{coro}
If the local system $\mathscr L$ is motivic (in particular if it is geometric), then the classical Leray spectral sequence and the perverse Leray spectral sequence are spectral sequences of Nori motives over $k$.
\end{coro}

\begin{proof}
The result follows from the functoriality of the direct image functors. 
\end{proof}

 In particular, the Leray spectral sequences are  spectral sequences of (polarizable) mixed Hodge structures. The compatibility of the classical Leray spectral sequence result in Hodge theory was already proved by Zucker in~\cite{MR534758} when $X$ is a curve and more generally, for both spectral sequences, by  Saito if $\mathscr L$ underlies an admissible variation of mixed Hodge structures (see \cite{MHMII}). This result has been recovered by de Cataldo and Migliorini with different techniques in~\cite{MR2680404}.

\subsubsection*{Nearby cycles} The theory developed here also shows that nearby cycles functors applied to perverse motives produce Nori motives.

\begin{coro} Let $X$ be a quasi-projective $k$-variety, $f:X\ra \bbA^1_k$ a flat morphism with smooth generic fiber $X_\eta$ and $\mathscr L$ be a $\bbQ$-local system on $X_\eta$. If 
$\mathscr L$ is motivic (in particular if it is geometric), then, for every point $x\in X_\sigma(k)$ and every integer $i\in\bbZ$, the Betti cohomology $H^i(\Psi_f(\mathscr L)_x)$ of the nearby fiber is canonically a Nori motive over $k$.
\end{coro}

\begin{proof}
The nearby cycles functor $\psi_f:=\Psi_f[-1]$ is $t$-exact for the perverse $t$-structure. Since it exists in the triangulated category of constructible \'etale motives (see \cite{AyoubII}) and the Betti realization is compatible with the nearby cycles functor by \cite[Proposition 4.9]{AyoubBetti}, the universal property ensures the existence of  an exact functor $\psi_f^\sM:\sM(X_\eta)\ra\sM(X_\sigma)$ and an invertible natural transformation $\rat^\sM_{X_\sigma}\psi_f^\sM\simeq \psi_f\rat^\sM_{X_\eta}$.

Let $d$ be the dimension of the generic fiber $X_\eta$. Since $\mathscr L$ is motivic, there exists an object $L$ in $\Db(\mathscr M(X_\eta))$ such that $\mathscr L$ and $\rat^{\sM}_{X_\eta}(L)$ are isomorphic. As $\mathscr L[d]$ is perverse and $\rat^{\sM}_{X_\eta}$ is conservative, the complex $L[d]$ belongs to $\mathscr M(X_\eta)$. So we conclude that $H^i(\Psi_f(\mathscr L)_x)$ is the Betti realization of the Nori motive $\mathrm{H}^{i+1-d}(x^*\psi_f^\sM L[d])$.

\end{proof}

\subsubsection*{Exponential motives}

The perverse motives introduced in the present paper and their stability under the four operations could be used also in the study of exponential motives as introduced in Fres\'an---Jossen's book~\cite{expo-mot}. Indeed, recall that Kontsevich and Soibelman define an exponential mixed Hodge structure as a mixed Hodge module $A$ on the complex affine line $\bbA^1_\bbC$ such that $p_*A=0$, where $p:\bbA^1_\bbC\ra\Spec(\bbC)$ is the projection (see the paper~\cite{MR2851153} of Kontsevich and Soibelman). Their definition can be mimicked in the motivic context and the abelian category of exponential Nori motives can be defined as the full subcategory of $\sM(\bbA^1_k)$ formed by the objects which have no global cohomology.

\subsubsection*{Constructible $t$-structure}

Let us conclude by a possible comparison with Arapura's construction from~\cite{MR2995668}. Let $X$ be a $k$-variety and consider the following full subcategories of $\Db(\sM(X))$
\begin{align*}
{}^c\mathrm{D}^{\leqslant 0}&:=\{A\in\Db(\sM(X)):H^k(x^*_\sM A)=0,\; \forall\; x\in X, \forall\; k>0\},\\
{}^c\mathrm{D}^{\geqslant 0}&:=\{A\in\Db(\sM(X)):H^k(x^*_\sM A)=0,\; \forall\; x\in X, \forall\; k<0\}.
\end{align*}
As in 4.6. Remarks of M. Saito's~\cite{MHMII} (see also \cite[Theorem C.0.12]{MR2995668}), we can check that these categories define a $t$-structure on $\Db(\sM(X))$. 

Let $\sMct(X)$ be the heart of this $t$-structure. Then, the functor $\rat^\sM_X$ induces a faithful exact functor from $\sMct(X)$ into the abelian category of constructible sheaves of $\bbQ$-vector spaces on $X$. Then, using the universal property of the category of constructible motives $\mathcal{M}(X,\bbQ)$ constructed by Arapura in \cite{MR2995668}, we get a faithful exact functor $\mathcal{M}(X,\bbQ)\ra\sMct(X)$. Is this functor an equivalence? If $X=\Spec k$, then
both categories are equivalent to the abelian category of Nori motives, so this functor is an equivalence.

\section{Weights}\label{sec:weights}

In this section, we will use results on motives and weight structures from Bondarko's and H\'ebert's papers~\cite{Bondarko, Hebert}. 
To apply these references directly in our context, we will make use of the fact that, if
$S$ is a Noetherian scheme of finite dimension, then Ayoub's
category $\DA_\ct(S)$ is canonically equivalent to the
category of constructible Be\u{\i}linson motives studied in Cisinski and
D\'eglise's book~\cite{CD}.
This follows from  \cite[Theorem 16.2.18]{CD} and will
henceforth be used without further comment. (Note also that, though the authors of \cite{Bondarko,Hebert} have chosen to use Be\u{\i}linson's motives, \'etale motives could have been used.)

\subsection{Continuity of the abelian hull}

Remember that, in chapter 5 of Neeman's book~\cite{MR1812507}, there are
four constructions of the abelian hull of a triangulated category.
The first one gives a lax $2$-functor from the $2$-category of triangulated
categories to that of abelian categories, but the other three constructions
give strict $2$-functors. If we use the fourth
construction, which Neeman calls $D(\matheusm S)$ (see \cite[Definition~5.2.1]{MR1812507}), then the following proposition is immediate.

\begin{prop} Let $\matheusm S$ be a triangulated category, and
suppose that we have an equivalence of triangulated categories
$\matheusm S\iso 2-\varinjlim_{i\in I}
\Scat_i$, where $I$ is a small filtered category.

Then the canonical functor $\Ab^\tr(\Scat)\ra 2-\varinjlim_{i\in I}\Ab^\tr
(\Scat_i)$ is an equivalence of abelian categories.

\label{prop_continuity_abelian_hull}
\end{prop}

\subsection{\'Etale realization and $\ell$-adic
perverse Nori motives}
\label{subsec:l-adic_perverse_motives}

Let $S$ be a Noetherian excellent scheme finite-dimensional scheme,
let $\ell$ be a prime
number invertible over $S$; we assume 
that $S$ is a $\bbQ$-scheme. 
(By Expos\'e XVIII-A of \cite{GTG}, the hypotheses above imply
Hypothesis~5.1 in Ayoub's paper \cite{AyoubEtale}.)
Under this hypothesis,
Ayoub has constructed an \'etale $\ell$-adic realization functor
on $\DA_\ct(S)$, compatible with pullbacks.

\begin{theo}(See \cite{AyoubEtale} sections 5 and 10.) 
Denote by $\Dbc(S,\bbQ_\ell)$ the category of constructible
$\ell$-adic complexes on $S$. 
Then we have a triangulated functor $\Ret_S:\DA_\ct(S)\ra
\Dbc(S,\bbQ_\ell)$ for every $S$
and, for every morphism $f:S\ra S'$, with $S'$ satisfying the same
hypotheses as $S$,
we have an
invertible natural transformation
\[\theta_f:f^*\circ\Ret_Y\ra \Ret_X\circ f^*.\]

\end{theo} 

Using results of Gabber (see \cite{Gabber_note}, and also sections 4 and 5
of Fargues's article \cite{Fargues}), we can construct an abelian category 
$\sP(S,\bbQ_\ell)$ of $\ell$-adic
perverse sheaves on $S$, satisfying
all the usual properties. In particular, we get a perverse cohomology
functor $\pH^0_\ell:=\pH^0\circ\Ret_S:\DA_\ct(S)\ra\sP(S,\bbQ_\ell)$.

\begin{defi}\label{defi:perversel} Let $S$ be as above.
The Abelian category of $\ell$-adic perverse motives on $S$
is the Abelian category
\[\sM(S)_\ell:=\Ab^\tr(\DA_\ct(S),\pH^0_\ell).\]
\end{defi}

By construction, the functor $\pH^0_\ell$ has a factorization
\[\DA_\ct(S)\xra{\pH^0_\sM}\sM(S)_\ell\xra{\rat^\sM_{S,\ell}}\sP(S,\bbQ_\ell)\]
where $\rat^\sM_{S,\ell}$ is a faithful exact functor
and $\pH^0_\sM$ is a homological functor.

By the universal property of $\sM(S)_\ell$, we also get pullback
functors between these categories as soon as the pullback functor
between the categories of $\ell$-adic complexes preserves the category
of perverse sheaves.

We will use the following important fact:
If we fix a base field $k$ of characteristic $0$
and only consider schemes
that are quasi-projective over $k$, then the main theorem (stated in
 \ref{subsec:statement_main_theorem}) stays true for the
categories $\sM(S)_\ell$. 
Of course,
we have to replace $\Db_\ct(S)$ and the Betti realization functor by
$\Dbc(S,\bbQ_\ell)$ and the \'etale realization functor
in all the statements. Indeed, the proof of the main theorem, of of the statements that
it uses, still work if we use the $\ell$-adic \'etale realization instead of the Betti realization.
The only result that requires a slightly different proof is \ref{lemm:compNCBetti}: we have to
show that the $\ell$-adic realization of Ayoub's logarithmic motive $\Log^\vee_n$ is the local system
used in Beilison's construction of the unipotent nearby cycle functor (see 1.1 and 1.2 of
Beilinson's~\cite{Beilinson-gluing} or Definition 5.2.1 of \cite{Morel-surQ}). As in the proof of 
\ref{lemm:compNCBetti}, it suffices to check this for $n=1$, and then it follows from
Lemma 11.22 of Ayoub's~\cite{AyoubEtale}.

\subsection{Mixed horizontal perverse sheaves}

Let $k$ be a field and $S$ be a $k$-scheme of finite type. 
Suppose that $k$ is finitely generated over its prime field.
We also fix
a prime number $\ell$ invertible over $S$.
The category 
$\Db_m(S,\bbQ_\ell)$ of mixed horizontal
$\bbQ_\ell$-complexes
and its perverse t-structure 
with heart $\sP_m(S,\bbQ_\ell)$ (the category of mixed horizontal
$\ell$-adic perverse sheaves on $S$)
were constructed
in Huber's article \cite{Huber-perverse} (see also  \cite[section 2]{Morel-surQ}). We recall the definition quickly and refer to \cite{Huber-perverse} and
\cite{Morel-surQ} for the details. First we consider the category
$\Dbh(S,\bbQ_\ell)$ of horizontal complexes on $S$, which is
by definition the $2$-colimit of the categories $\Dbc(\mathscr{X},\bbQ_\ell)$,
where $\mathscr{X}$ runs over all flat finite type models of $X$ over
regular subalgebras $A$ of $k$ that are of finite type over $\bbZ$ and
have $k$ as their fraction field. There is an obvious functor
$\eta^*:\Dbh(S,\bbQ_\ell)\ra\Dbc(S,\bbQ_\ell)$, which is triangulated and
conservative, and a perverse $t$-structure on $\Dbh(S,\bbQ_\ell)$
that is characterized by the fact that $\eta^*$ is $t$-exact. Also,
the functor $\eta^*$ is fully faithful on the heart of this $t$-structure
(\cite[Proposition~2.6.2]{Morel-surQ}).

We say that an object of $\Dbh(S,\bbQ_\ell)$ if it extends to a complex
$K$ on a model $\mathscr{X}$ of $X$ as before such that all the
(ordinary) cohomology sheaves of $K$ are successive extensions of
punctually pure sheaves in the sense of Deligne's~\cite{WeilII}.
The category $\Dbm(S,\bbQ_\ell)$ of mixed horizontal complexes is the
full subcategory of $\Dbh(S,\bbQ_\ell)$ whose objects are mixed complexes.
The perverse $t$-structure on $\Dbh(S,\bbQ_\ell)$ restricts to a
$t$-structure on $\Dbm(S,\bbQ_\ell)$, whose heart is the category
$\sPm(S,\bbQ_\ell)$ of mixed horizontal perverse sheaves; this last
category is a full subcategory of the heart of the perverse $t$-structure
on $\Dbh(S,\bbQ_\ell)$, so $\eta^*$ induces a fully faithful functor
$\sPm(S,\bbQ_\ell)\ra\sP(S,\bbQ_\ell)$.

Now we want to show that the realization functor $\rat^\sM_{S,\ell}:\sM(S)_\ell\ra\sP(S,\bbQ_\ell)$ factors
through the fully faithful functor $\sP_m(S,\bbQ_\ell)\ra\sP(S,\bbQ_\ell)$.

We have a continuity theorem for the categories of \'etale motives, proved by Ayoub 
in \cite[Corollaire 1.A.3]{AyoubRigid} and \cite[Corollaire 3.22]{AyoubEtale} and by Cisinski--D\'eglise in \cite[Proposition~15.1.6]{CD}).

\begin{theo}\label{theo:cont} 
Let $S$ be a Noetherian scheme of finite dimension.
Suppose that we have $S=\varprojlim_{i\in I} S_i$, where all the $S_i$
are finite-dimensional Noetherian schemes and all the transition maps
$S_i\ra S_j$ are affine. Then the canonical functor
$2-\varinjlim_{i\in I}\DA_\ct(S_i)\ra\DA_\ct(S)$ is an equivalence of
monoidal triangulated categories.

\end{theo}

Using the definition of mixed horizontal $\ell$-adic complexes, we immediately
get the following corollary.

\begin{coro} Let $S$ and $\ell$ be as in the beginning of this subsection.
Then the \'etale realization functor $\DA_\ct(S)\ra\Dbc(S,\bbQ_\ell)$ factors
through a functor $\DA_\ct(S)\ra\Dbh(S,\bbQ_\ell)$.

\end{coro}

\begin{coro}\label{coro_mixed_horizontal_real}
With the notation of the previous corollary, the essential image of the functor
$\DA_\ct(S)\ra\Dbh(S,\bbQ_\ell)$ is contained in the full subcategory
$\Dbm(S,\bbQ_\ell)$.
In particular, the perverse cohomology functor
$\pH^0_\ell:\DA_\ct(S)\ra\sP(S,\bbQ_\ell)$ factors through the subcategory $\sP_m(S,\bbQ_\ell)$.

\end{coro}

\begin{proof}
This follows from the facts that $\DA_\ct(S)$ is generated by the Tate
twists of motives of smooth $S$-schemes (see Definition~15.1.1 and
Proposition~15.1.4 of \cite{CD}) and that mixed horizontal complexes
are preserved by direct images and Tate twists (see
\cite[Proposition~3.2]{Huber-perverse} for direct images, the stability
by Tate twists is easy).

\end{proof}

\begin{coro}
The essential image of the
realization functor $\rat^\sM_{S,\ell}:\sM(S)_\ell\ra\sP(S,\bbQ_\ell)$ is contained in the
subcategory  $\sP_m(S,\bbQ_\ell)$.

\end{coro}

We will also denote the resulting faithful exact functor $\sM(S)\ra\sP_m(S,\bbQ_\ell)$ by
$\rat^\sM_{S,\ell}$.

\begin{rema}
Suppose that $k$ is not necessarily finitely generated
over its prime field. We define $\Dbm(S,\bbQ_\ell)$ as the $2$-colimit
of the categories $\Dbm(S',\bbQ_\ell)$, for $S'$ a model
of $S$ over a finitely generated subfield of $k$. This category inherits
a perverse $t$-structure from the perverse $t$-structures on the
$\Dbm(S',\bbQ_\ell)$, whose heart we denote by $\sPm(S,\bbQ_\ell)$.
The obvious functor $\sPm(S,\bbQ_\ell)\ra\sP(S,\bbQ_\ell)$ is only
exact faithful in general (not necessarily fully faithful), but the 
the perverse cohomology functor
$\pH^0_\ell:\DA_\ct(S)\ra\sP(S,\bbQ_\ell)$ 
still factors
through this functor as in Corollary~\ref{coro_mixed_horizontal_real}
(by Theorem~\ref{theo:cont}), so we get a faithful exact
realization functor $\sM(S)_\ell\ra\sPm(S,\bbQ_\ell)$.

\end{rema}

\subsection{Continuity for perverse Nori motives}

Like the triangulated category of motives, the category of
perverse Nori motives satisfies a continuity property.

\begin{prop} Let $S$ be a scheme and $\ell$ be a prime number
satisfying the conditions of 
\ref{subsec:l-adic_perverse_motives}. We assume that
$S=\varprojlim_{i\in I}S_i$, where
$(S_i)_{i\in I}$ is a directed projective system of schemes
satisfying the same conditions as $S$, and in which the transition maps
are affine. We also assume that the pullback by any transition map
$S_i\ra S_j$ preserves the category of perverse sheaves,
and that there exists $a\in\bbZ$ such that, if
$f_i:S\ra S_i$ is the canonical map, then $f^*_i[a]$ 
preserves the category of perverse sheaves for every $i\in I$.
Under these hypotheses, the functors $f_i^*[a]$ induce a functor
\[2-\varinjlim_{i\in I}\sM(S_i)_\ell\ra\sM(S)_\ell,\]
and this functor is full and essentially surjective.

If moreover
the canonical exact functor $2-\varinjlim_{i\in I}
\sP(S_i,\bbQ_\ell)\ra\sP(S,\bbQ_\ell)$ 
induced by the $f_i^*[a]$ is faithful,
then the canonical functor
\[2-\varinjlim_{i\in I}\sM(S_i)_\ell\ra\sM(S)_\ell\]
is an equivalence of abelian categories.

\label{prop_continuity_sM}
\end{prop}

\begin{proof} This follows from \ref{prop_continuity_abelian_hull} and \ref{theo:cont}.

\end{proof}

\begin{coro} Let $S$ and $\ell$
be as above, and suppose also that $S$ is integral.
Then, if
$\eta$ is the generic point of $S$, the canonical
exact functor
\[2-\varinjlim_U \sM(U)_\ell\ra\sM(\eta)_\ell,\]
where the colimit is taken over all nonempty affine open subschemes of $S$
and where the image of $K_U\in\Ob\sM(U)_\ell$ is $K_{U,\eta}[-\dim S]$,
is an equivalence of categories.

\label{cor_motifs_point_generique}
\end{coro}

\begin{proof}
By \ref{prop_continuity_sM}, it suffices to check
that the similar functor
\[2-\varinjlim_U \sP(U,\bbQ_\ell)\ra\sP(\eta,\bbQ_\ell)\]
is faithful. Let $K$ be an object of
$2-\varinjlim_U \sP(U,\bbQ_\ell)$ whose image in $\sP(\eta,\bbQ_\ell)$ is
$0$, and let $U$ be a nonempty open affine subscheme of $S$ such that
$K$ comes from an object $K'$ of $\sP(U,\bbQ_\ell)$. After shrinking
$U$ (which does not change $K$), we may assume that
$K'[-\dim S]$ is a local system. Then the condition $K'_\eta[-\dim S]=0$
implies that this local system is zero, hence that $K=0$.

\end{proof}

\subsection{Comparison of the different categories
of perverse Nori motives}

In the next proposition, we compare the $\ell$-adic definition of perverse motives with the one used previously and obtained via the Betti realization.

\begin{prop}\label{prop:independancel} Suppose that $k$ is a field 
of characteristic $0$ and that $S$ is quasi-projective over $k$.

We write
$\rho_\ell$ for the canonical exact functor
$\Ab^\tr(\DA_\ct(S))\ra\sP(S,\bbQ_\ell)$ induced by $\pH^0_\ell$.
If
$\sigma$ is an embedding of $k$ into $\bbC$, then
we also have an exact functor
$\rho_\sigma$ from $\Ab^\tr(\DA_\ct(S))$ to $\sP(S)$ induced by $\pH^0$.
Then:
\begin{enumerate}
\item If $\ell,\ell'$ are two 
prime numbers,
then $\Ker\rho_\ell=\Ker\rho_{\ell'}$. 
In particular, we get
a canonical equivalence of abelian categories $\sM(S)_\ell=\sM(S)_{\ell'}$.

\item If $\sigma:k\ra\bbC$ is
an embedding, then $\Ker\rho_\ell=\Ker\rho_\sigma$.
In particular, we get
a canonical equivalence of abelian categories $\sM(S)_\ell=\sM(S)$.

\end{enumerate}

\end{prop}

\begin{proof} 
We first treat the case $S=\Spec k$.
If $k$ can be embedded in $\bbC$, then (2) follows from Huber's
construction of mixed realizations in Huber's~\cite{Huber-real}, and (1) follows
from (2). In the general case, 
(1) follows from the case where
$k$ can be embedded in $\bbC$ and from \ref{prop_continuity_sM},
applied to the family of 
subfields of $k$ that can be embedded in $\bbC$.

Now we treat the case of a general $k$-scheme $S$.
As in the first case, (1) follows from (2)
and from \ref{prop_continuity_sM}. So suppose that we have
an embedding $\sigma:k\ra\bbC$. 
We prove the desired result by induction on the dimension of $S$.
The case $\dim S=0$ has already been treated, so we may assume that $\dim S>0$
and that the result is known for all the schemes of lower dimension.
We denote by $M\mapsto[M]$ the canonical functor $\DA_\ct(S)\ra\Ab^\tr(\DA_\ct(S))$;
as $\DA_\ct(S)$ is a triangulated category, this is a fully faithful
functor. Let $X$ be an object of $\Ab^\tr(\DA_\ct(S))$. 
By construction of $\Ab^\tr(\DA_\ct(S))$,
there
exists a morphism $N\ra M$ in $\DA_\ct(S)$ such that $X$ is the
cokernel of $[N]\ra[M]$. Then $\rho_\ell(X)$ is the cokernel
of $\pH^0_\ell(N)\ra\pH^0_\ell(M)$, so $\rho_\ell(X)=0$ if and only
$\pH^0_\ell(N)\ra\pH^0_\ell(M)$ is surjective. Similarly, $\rho_\sigma(X)=0$
if and only if $\pH^0(N)\ra\pH^0(M)$ is surjective. We can check these
conditions on a Zariski open covering of $S$, so we may assume that
$S$ is affine.
Choose a nonempty
smooth open subset $U$ of $S$ such that the restrictions to $U$ of 
$\rho_\ell(M)$, $\rho_\ell(N)$,
$\rho_\sigma(M)$ and
$\rho_\sigma(N)$ are all locally constant sheaves placed in degree
$-\dim S$.
As $S$ is affine,
after shrinking $U$, we may assume that $U$ is the complement
of the vanishing set of a nonzero function $f\in\mathcal{O}(S)$.
By Proposition~3.1 of Beilinson's~\cite{Beilinson-gluing}, we have
that $\rho_\ell(N)\ra\rho_\ell(M)$ is surjective
if and only if both 
$\rho_\ell(N)_{\mid U}\ra\rho_\ell(M)_{\mid U}$ 
and ${}^p\Phi_f^\sM\rho_\ell(N)\ra{}^p\Phi_f^\sM\rho_\ell(M)$ are, 
which is equivalent to the surjectivity of
$\rho_\ell(N_{\mid U})\ra\rho_\ell(M_{\mid U})$ 
and $\rho_\ell(\Phi_f N)\ra\rho_\ell(\Phi_f M)$. We have a similar statement
for $\rho_\sigma$. 
As $\dim(S-U)<\dim(S)$, we can use the induction hypothesis
to reduce to the case $S=U$.
It suffices to check the result on an \'etale cover of $S$, so we
may assume that $S$ has a rational point $x$. 
Let $i:x\ra S$ be the obvious inclusion.
As $\rho_\ell(N)[-\dim S]$ and
$\rho_\ell(M)[-\dim S]$ are locally constant sheaves on $S$, the morphism
$\rho_\ell(N)\ra\rho_\ell(M)$ is surjective if and only if
$\rho_\ell(i^* N[-\dim S])\ra\rho_\ell(i^* M[-\dim S])$ is, and similarly
for $\rho_\sigma$. So we are reduced to the result on 
the scheme $x$, which we have
already treated.
\end{proof}

\begin{coro}\label{coro:allfourop} Let $k$ be a field of characteristic $0$ and $S$ a quasi-projective
scheme over $k$. We have a canonical $\bbQ$-linear
abelian category of perverse Nori motives
$\sM(S)$, together with a cohomological functor $\pH^0_\sM:\DA_\ct(S)\ra
\sM(S)$, with a $\ell$-adic realization functor
$\rat^\sM_{S,\ell}:\sM(S)\ra\sP(S,\bbQ_\ell)$ for every prime number
$\ell$, with a Betti realization functor
$\rat^\sM_{S,\sigma}:\sM(S)\ra\sP(S)$ for every embedding
$\sigma:k\ra\bbC$, and it has a formalism of the 4 operations, duality, 
unipotent nearby and vanishing cycles compatible with all these operations.

\end{coro}

We fix a field $k$ of characteristic zero and a quasi-projective
scheme $S$ over $k$. We first define weights via the $\ell$-adic realizations.

\begin{defi}\label{defi:weights} Let $w\in\bbZ$.
Let $K$ be an object of $\sM(S)$. We say that $K$ is of weight $\leq w$
(resp. $\geq w$) if $\rat^\sM_{S,\ell}(K)\in\Ob(\sP_m(S,\bbQ_\ell))$ 
is of weight $\leq w$ (resp.
$\geq w$) for every 
prime number $\ell$. We say that $K$ is pure of
weight $w$ if it is both of weight $\leq w$ and of weight $\geq w$.
\end{defi}

In \ref{prop:weightsaremotivic}, we will a more intrinsic definition of weights that does not use
the realization functors.

\begin{defi} A \emph{weight filtration} on an object $K$ of $\sM(S)$ is
an increasing filtration $W_\bullet K$ on $K$ such that $W_i K=0$ for
$i$ small enough, $W_i K=K$ for $i$ big enough, and $W_{i}K/W_{i-1}K$ is
pure of weight $i$ for every $i\in\bbZ$.

\end{defi}

The next result follows immediately from the similar result in
the categories of mixed horizontal perverse sheaves (see Proposition 3.4 and Lemma 3.8 of
Huber's~\cite{Huber-perverse}).

\begin{prop} Let $K,L$ be objects of $\sM(S)$, and let $w\in\bbZ$.
\begin{enumerate}
\item If $K$ is of weight $\leq w$ (resp. $\geq w$), so is every
subquotient of $K$.
\item If $K$ is of weight $\leq w$ and $L$ is of weight $\geq w+1$,
then $\Hom_{\sM(S)}(K,L)=0$.

\end{enumerate}
\end{prop}

Recall that, if $A$ and $B$ are objects of an abelian category
endowed with increasing filtrations $(F_i A)_{i\in\bbZ}$ and
$(F_i B)_{i\in\bbZ}$, then a morphism $u:A\ra B$ is called
\emph{compatible} (resp. \emph{strictly compatible})
with the filtrations if, for every $i\in\bbZ$, we have
$u(F_i A)\subset F_i B$ (resp. $u(F_i A)=u(A)\cap F_i B$).

\begin{coro}\label{coro:strictness} A weight filtration on an object of $\sM(S)$ is unique
if it exists, and morphisms of $\sM(S)$ are strictly compatible with
weights filtrations.
In particular, if an object of $\sM(S)$ has a weight filtration, then so do
all its subquotients.

\end{coro}

\subsection{Application of Bondarko's weight structures}

Let $S$ be as in the previous subsection. We will now make use of Bondarko's
Chow weight structure on $\DA_\ct(S)$. 
Let $\Chow(S)$ be the full subcategory of $\DA_\ct(S)$ whose objects
are direct factors of finite direct sums of objects of the form
$f_!\bbQ_X(d)[2d]$, with $f:X\ra S$ a proper morphism from a smooth $k$-scheme
$X$ to $S$ and $d\in\bbZ$. Then, as shown by H\'ebert in \cite[Theorem~3.3]{Hebert}, and also by Bondarko in \cite[Theorem~2.1]{Bondarko}, there exists a unique weight structure
on $\DA_\ct(S)$ with heart $\Chow(S)$
(see \cite[Definition~1.5]{Hebert} or \cite[Definition~1.5]{Bondarko} for the definition of a weight structure).

In particular, for every object $K$ of $\DA_\ct(S)$, there exists an
exact triangle $A\ra K\ra B\xra{+1}$ (not unique) such that
$A$ (resp. $B$) is a direct factor of a successive extension of objects of
$\Chow(S)[i]$ with $i\leq 0$ (resp. $i\geq 1$).

\begin{prop} Every object of $\sM(S)$ has a weight filtration. Moreover,
if $S=\Spec k$ and $\sigma$ is an embedding of $k$ in $\bbC$, then the
notion of weights of \ref{defi:weights}
coincides with that of Section~10.2.2 of the book~\cite{BookNori} of Huber and M\"uller-Stach.

\label{cor_weight_filtration}
\end{prop}

\begin{proof} We first prove that, if $M$ is an object of $\Chow(S)$,
then $\pH^0_\sM(M)$ is pure of weight $0$ in our sense, and also in the sense
of \cite[Section~10.2.2]{BookNori} if $S=\Spec k$ with $k$ embeddable in
$\bbC$. The second statement is actually an immediate consequence of
\cite[Definition~10.2.4]{BookNori} and of the motivic Chow's lemma
(see for example \cite[Lemma~3.1]{Hebert}). To prove the first statement,
by definition of the weights on $\sP_m(S,\bbQ_\ell)$,
we may assume that $k$ is finitely generated
over $\bbQ$; then the statement follows immediately from \cite[5.1.14]{BBD}
 (see the remark on page~116 of \cite{Huber-perverse}).

Then we note that every object of $\sM(S)$ is a quotient of an
object of the form $\pH^0_\sM(M)$, for $M\in\Ob(\DA_\ct(S))$ (because
this is true for objects of $\Ab^\tr(\DA_\ct(S))$). So it suffices
to prove the result for objects in the essential image of
$\pH^0_\sM$. Let $M\in\Ob(\DA_\ct(S))$, and let $K=\pH^0_\sM(M)$.
Let $w\in\bbZ$.
By the first part of the proof,
if $M$ 
is a direct factor of a successive extension of objects
of $\Chow(S)[i]$ with $i\leq w$ (resp $i\geq w+1$), then $\pH^0_\sM(M)$
is of weight $\leq w$ (resp. $\geq w+1$) in our sense, and also in the
sense of \cite{BookNori} when this applies. In general, using the Chow weight structure of Bondarko, we can find an exact triangle
$A\ra M\ra B\xra{+1}$ such that $A$ (resp. $B$)
is a direct factor of a successive extension of objects
of $\Chow(S)[i]$ with $i\leq w$ (resp $i\geq w+1$). Applying $\pH^0_\sM$,
we get an exact sequence $\pH^0_\sM(A)\ra K\ra\pH^0_\sM(B)$,
with $\pH^0_\sM(A)$ of weight $\leq w$ and $\pH^0_\sM(B)$ of weight $\geq w+1$.
If we set $W_w K=\Img(\pH^0_\sM(A)\ra K)$, then $W_w K$ is of weight $\leq w$
and $K/W_w K$ is of weight $\geq w+1$. This defines a weight filtration
on $K$.
\end{proof}

Weights and the related weight filtration so far have been defined and constructed for perverse motives via 
the $\ell$-adic realizations. As we shall see now, we can also define weights more directly.
Let $\DA_\ct(S)_{w\leq i}$ be the full subcategory of $\DA_\ct(S)$ whose objects are direct factors of successive extensions of objects of
$\Chow(S)[w]$ with $w\leq i$ and consider the Abelian category 
\[\sM(S)_{w\leq i}:=\Ab^\ad(\DA_\ct(S)_{w\leq i},\pH^0_\ell),\]
for some 
prime number $\ell$.
It follows from \ref{prop:independancel} that this category, up to an equivalence, does not depend on $\ell$. Indeed, the universal property provides a commutative diagram (up to isomorphisms of functors)
\[\xymatrix{{\DA_\ct(S)_{w\leq i}}\ar[d]^-I\ar[r] &{\Ab^\ad(\DA_\ct(S)_{w\leq i})}\ar[d]^-{J}\ar[rd]^-{\varrho_\ell} & {}\\
{\DA_\ct(S)}\ar[r]\ar@/_2em/[rr]_-{\pH^0_\ell} &{\Ab^\tr(\DA_\ct(S))}\ar[r]^-{\rho_\ell} & {\sP_m(S,\bbQ_\ell)}}\] 
where $I$ is the inclusion and $J,\varrho_\ell$ are exact functors. As by construction $\sM(S)_{w\leq i}:=\Ab^\ad(\DA_\ct(S)_{w\leq i})/\Ker\varrho_\ell$ it suffices to show that $\Ker\varrho_\ell$ is independent on $\ell$. Let $A$ be an object in $\Ab^\ad(\DA_\ct(S)_{w\leq i})$. Since $A$ belongs to $\Ker\varrho_\ell$ if and only if $J(A)$ belongs to $\Ker\rho_\ell$ our claim follows from \ref{prop:independancel}.

The inclusion $\DA_\ct(S)_{w\leq i}\subseteq \DA_\ct(S)$ induces a faithful exact functor 
\[u_i:\sM(S)_{w\leq i}\ra \sM(S).\]
Let $K$ be an object in $\sM(S)$. Given an object $(L,\alpha:u_i(L)\ra K)$ in the slice category $\sM(S)_{w\leq i}/K $ we can consider the subobject $\Img\alpha$ of $K$ and define $\eusm W_iK$ to be the union of all such subobjects in $K$, that is, we set
\[\eusm W_iK:=\colim_{(L,\alpha)\in\sM(S)_{w\leq i}/K}\Img\alpha.\]
This construction is functorial in $K$ (and moreover using the inclusion of $\DA_\ct(S)_{w\leq i} $ in $\DA_\ct(S)_{w\leq i+1} $ it is easy to see that it defines a filtration on $K$).

\begin{prop}\label{prop:weightsaremotivic}
Let $K\in\sM(S)$. Then, $W_iK=\eusm W_i K$, for every integer $i\in\bbZ$.
\end{prop}

\begin{proof}
As observed in the proof of \ref{cor_weight_filtration}, if $M$ belongs to $\DA_\ct(S)_{w\leq i}$, then $\pH^0_\sM(M)$ is of weight $\leq i$. Hence, the functor $u_i$ takes its values in the Abelian subcategory of $\sM(S)$ formed by the objects of weight $\leq i$. As a consequence, for $(L,\alpha)$ in the slice category $\sM(S)_{w\leq i}/K$, the subobject $\Img\alpha$ of $K$ is of weight $\leq i$ and therefore $\eusm W_i K\subseteq W_iK$.

Conversely, there exists an epimorphism $e:\pH^0_\sM(M)\twoheadrightarrow K$ where $M$ belongs to $\DA_\ct(S)$. By construction 
\[W_i\pH^0_\sM(M):=\Img\big(\pH^0_\sM(A)\ra\pH^0_\sM(M)\big)\subseteq \eusm W_i\pH^0_\sM(M)\]
where $A$ is an object of $\DA_\ct(S)_{w\leq i}$ that fits in an exact triangle $A\ra M\ra B\xra{+1}$ such that $B$
is a direct factor of a successive extension of objects
of $\Chow(S)[w]$ with $w\geq i+1$.
Therefore, since the weight filtration on $K$ is the induced filtration (see \ref{coro:strictness}), we get
\[W_iK=e(W_i\pH^0_\sM(M))\subseteq e(\eusm W_i\pH^0_\sM(M))\subseteq \eusm W_i K.\]
This concludes the proof.
\end{proof}

\subsection{The intermediate extension functor}
Recall the definition of the intermediate extension functor, that already appeared in the proof of \ref{coro:IC}.
\begin{defi}\label{defi:IC} Let $j:S\ra T$ be a quasi-finite morphism between
quasi-projective $k$-schemes. We define a functor
$j_{!*}^\sM:\sM(S)\ra\sM(T)$
by
\[j_{!*}^\sM(K)=\Img(\mathrm{H}^0(j_!^\sM K)\ra\mathrm{H}^0(j_*^\sM K)).\]

\end{defi}

Note that, as $j$ is quasi-finite, the functor $j_!^\sM$ is right exact
and the functor $j_*^\sM$ is left exact. In particular, the functor
$j_{!*}^\sM$ preserves injective and surjective morphisms, but it is not
exact in general.

\begin{prop}\label{prop:purityintermediate} Let $j:S\ra T$ be an open immersion, and let
$w\in\bbZ$.
Then, if $K\in
\Ob\sM(S)$ is of weight $\leq w$ (resp. of weight $\geq w$, resp.
pure of weight $w$), so is $j_{!*}^\sM K$.

Also, the functor $j_{!*}^\sM$ is exact on the full abelian subcategory
of objects that are pure of weight $w$.

\end{prop}

\begin{proof} It suffices to show these statement for mixed $\ell$-adic
perverse sheaves. The first statement follows from \cite[5.3.2]{BBD}
(more precisely, if $j$ is not affine, 
it follows from \cite[5.1.14 and
5.3.1]{BBD} ). The second statement follows from \cite[Corollary~9.4]{Morel-surQ}.
\end{proof}

\subsection{Pure objects}

Let us start with the definition of objects with strict support on a given closed subscheme.
\begin{defi} Let $Z$ be a closed integral subscheme of $S$, and denote
the immersion $Z\ra S$ by $i$.
We say that an object $K$ of $\sM(S)$ has strict support $Z$ if
$K_{\mid S-Z}=0$ and if, for every nonempty open subset $j:U\ra Z$, the adjunction
morphism $K\ra (ij)_*^\sM (ij)_\sM^* K$ is injective and induces an
isomorphism between $K$ and
$(ij)_{!*}^\sM (ij)_\sM^* K$.

\end{defi}

\begin{rema}
For example, if $K_{\mid S-Z}=0$ and
if there exists a smooth dense open subset $j:U\ra Z$ such that
$\rat^\sM_U(K_{\mid U})[-\dim U]$ (or any $\rat^\sM_{U,\ell}(K_{\mid U})[-\dim U]$ for
some prime number $\ell$) is locally constant and $K_Z=j_{!*}^\sM(K_{\mid U})$, 
then $K$ has strict support $Z$.
Indeed, this follows from the similar result for perverse sheaves, which
follows from \cite[4.3.2]{BBD}  (note that the proof of this result does
not use the hypothesis that $\mathrm{L}$ is irreducible).

\end{rema}

\begin{prop} (Compare with \cite[5.3.8]{BBD}.) Let $K$ be an object of
$\sM(S)$, and suppose that $K$ is pure of some weight. Then we can write
$K=\bigoplus_Z K_Z$, where the sum is over all integral closed subschemes
$Z$ of $S$, each $K_Z$ is an object of $\sM(S)$ with strict support $Z$, and
$K_Z=0$ for all but finitely many $Z$. 

\label{prop_dec_support}
\end{prop}

\begin{proof} We prove the result by Noetherian induction on $S$. If
$\dim S=0$, there is nothing to prove. Suppose that $\dim S\geq 1$,
and let $j:U\ra S$ be a nonempty open affine subset of $S$. 
After shrinking $U$, we may assume that $U$ is smooth and that $\rat^\sM_S
(K)[-\dim U]$ is a locally constant sheaf on $U$. 
Let $w$ be the
weight of $K$. Then  \cite[Corollary~9.4]{Morel-surQ} implies that
$j_*^\sM j^*_\sM K/j_{!*}^\sM j^*_\sM K$ is of weight $\geq w+1$, so
the adjunction morphism 
$K\ra j_*^\sM j^*_\sM K$ factors through a morphism
$K\ra j_{!*}^\sM j^*_\sM K$. Similarly, the adjunction morphism
$j_!^\sM j^*_\sM K\ra K$ factors through a morphism
$j_{!*}^\sM j^*_\sM\ra K$. By definition of $j_{!*}^\sM$, the composition
$j_{!*}^\sM j^*_\sM K\ra K\ra j_{!*}^\sM j^*_\sM K$ is the identity of
$j_{!*}^\sM j^*_\sM K$. So we have $K=j_{!*}^\sM j^*_\sM K\oplus L$, with
$j^*_\sM L=0$. The first summand has strict support $\overline{U}$ by the remark
above, and $L_{\mid U}=0$, so the conclusion follows from the induction
hypothesis applied to $L_{\mid S-U}$.
\end{proof}

\begin{theo}  Let $S$ be as before, and let $w\in\bbZ$. Let $\sM(S)_{w}$ be
the full abelian subcategory of $\sM(S)$ whose objects
are motives that are pure of weight $w$.

Then $\sM(S)_{w}$ is semi-simple.
\end{theo}

\begin{proof} By \ref{prop_dec_support},
we may assume that $S$ is integral,
and it suffices to prove the result for the full subcategory $\sM(S)_w^0$ of
objects in $\sM(S)_w$ with strict support $S$ itself.

Let $\eta$ be the generic point of $S$. By
\ref{cor_motifs_point_generique},
we have a full and essentially surjective exact functor
(given by the restriction morphisms) $2-\varinjlim_U\sM(U)\ra\sM(\eta)$,
where the limit is over the projective system of nonempty affine open
subsets $U$ of $S$. For such a $U$, we denote by
$\sM(U)_w^0$ the full subcategory of $\sM(U)$ whose objects are
motives that are pure of weight $w$ and have strict support $U$. 
By Proposition~\ref{cor_weight_filtration}, the functor above induces a
full and essentially surjective functor
$2-\varinjlim_U\sM(U)_w\ra\sM(\eta)_w$, and,
by \ref{prop_dec_support}, this is turn gives a 
full and essentially surjective functor
$2-\varinjlim_U\sM(U)^0_w\ra\sM(\eta)_w$.
Moreover, if $j:U\ra S$ is a nonempty open subset, then the exact functor
$j^*_\sM:\sM(S)^0_w\ra\sM(U)_w^0$ is an equivalence of categories, because
it has a quasi-inverse, given by $j_{!*}^\sM$. So we deduce that the
restriction functor $\sM(S)_w^0\ra\sM(\eta)_w$ is full and essentially
surjective. But this functor is also faithful, because the analogous functor
on categories of $\ell$-adic perverse sheaves is faithful. So
$\sM(S)^0_w\ra\sM(\eta)_w$ is an equivalence of categories, which means that
we just need to show the theorem in the case $S=\eta$, i.e. if $S$ is the
spectrum of a field.

Now suppose that $S=\Spec k$. Then, by \ref{prop_continuity_sM},
$\sM(k)_w=2-\varinjlim_{k'}\sM(k')_w$,
where the limit is over all the subfields $k'$ of $k$ that are
finitely generated over $\bbQ$. So it suffices to show the theorem for
$k$ finitely generated over $\bbQ$. But then we can embed $k$ into $\bbC$,
and the conclusion follows from \cite[Theorem~10.2.7]{BookNori}.
\end{proof}

\begin{defi}
Let $K$ be an object of $\Db\sM(X)$ and $w\in\bbZ$. We say that $K$ is
of weight $\leq w$ (resp. of weight $\geq w$, resp. pure of weight
$w$) if, for every $i\in\bbZ$, the perverse motive $\Hc^i K$ is
of weight $\leq w+i$ (resp. of weight $\geq w+i$, resp. pure of weight
$w+i$).

\end{defi}

\begin{coro}\label{cor_vanishing_ext}
Let $K,L$ be objects of $\sM(S)$. If $K$ and $L$ are pure
of respective weights $i$ and $j$, then $\Ext^r_{\sM(S)}(A,B)=0$
if $i<j+r$.

\end{coro}

\begin{proof}
By Lemma~4.5 of M. Saito's~\cite{Saito-HCMMII},
this follows from the existence of the weight filtration and the
fact that it is strictly compatible with morphisms of $\sM(S)$, and
from the semi-simplicity of pure objects of $\sM(S)$.

\end{proof}

\begin{coro}
\begin{enumerate}
\item There exists a unique weight structure (see \cite[Definition~1.5]{Bondarko}) on $\Db\sM(S)$ whose heart is the full subcategory of complexes
of weight $0$. 

\item Let $K,L$ be objects of $\Db\sM(S)$ and $w\in\bbZ$. 
If $K$ is of
weight $\leq w$ and $L$ is of weight $>w$, then
$\Hom_{\Db\sM(S)}(K,L)=0$.

\item The weight structure of (1) is transversal to the
canonical $t$-structure on $\Db\sM(S)$ in the sense of
Definition~1.2.2 of Bondarko's paper~\cite{Bondarko3}.

\item If $K\in\Ob\Db\sM(S)$ is pure of some weight, then
$K\simeq\bigoplus_{i\in\bbZ}\Hc^i K[-i]$.

\end{enumerate}
\end{coro}

\begin{proof}
To prove (1), we apply part II of Theorem~4.3.2 of~\cite{Bondarko2} to
the triangulated category $\Db\sM(S)$ and the full subcategory $\sA$ of
complexes of weight $0$. This subcategory is stable by finite coproducts
and direct summands, and it generates $\Db\sM(S)$. Indeed, to prove
the second statement, it suffices to show that the triangulated
subcategory generated by $\sA$ contains $\sP(S)$; but every perverse motives
is a successive extension of pure perverse motives (by the existence of
the weight filtration), and, if $K$ is a pure perverse motives, then
some shift of $K$ is an $\sA$. By Theorem~4.3.2 of~\cite{Bondarko2}, 
there exists a weight structure on $\Db\sM(S)$ with heart $\sA$ if and only if,
for every objects $K,L$ of $\sA$ and every integer $n>0$, we have
$\Hom_{\Db\sM(S)}(K,L)=0$. As the functor $\Hom$ is cohomological in
each variable, we may assume that $K$ and $L$ are concentrated in
one degree, so that there exist objects $A$ and $B$ that are pure
of respective weights $i$ and $j$ such that $K=A[-i]$ and
$L=B[-j]$. Then $\Hom_{\Db\sM(S)}(K,L[n])=\Ext^{n+i-j}_{\sM(S)}(A,B)$
is zero by \ref{cor_vanishing_ext}.

We prove (2). We have $\Hom_{\Db\sM(S)}(K,L)=\Hom_{\Db\sM(S)}(K[-w],L[-w])$.
As $K[-w]$ is of weight $\leq 0$ and $L[-w]$ is of weight $\geq 1$,
the statement follows from Proposition~1.3.3(1) of~\cite{Bondarko2}.

Point (3) follows immediately from the existence of the weight filtration on objects of $\sM(S)$.

We prove(4). Let $w$ be the weight of $K$.
Let $i\in\bbZ$. Then $\tau_{\leq i}K$ and $\tau_{>i}K$ are pure
of weight $w$, so $\Hom_{\Db\sM(S)}(\tau_{>i}K,\tau_{\leq i}K[1])=0$ by
(3), so the exact triangle $\tau_{\leq i}K\ra K\ra\tau_{>i}K\xra{+1}$
splits. This implies the statement.
\end{proof}

\begin{theo}\label{theo:weightIC}
Let $f:X\ra S$ be a proper morphism of quasi-projective $k$-varieties with $X$ irreducible. 
Let $j:U\ra X$ be an open immersion, and
 $K$ be a perverse motive on $U$. If $K$ is pure of weight $w$, then
$\mathrm{H}^i(f_*^\sM j^\sM_{!*}K)$ is a motivic perverse sheaf that is pure of weight $w+i$.
\end{theo}

\begin{proof}
Let us say that $L\in\Db(\sM(S))$ is pure of weight $w$ if $\mathrm{H}^iL$ is pure of weight $w+i$ for every $i\in\bbZ$. For such an $L$, by \ref{coro:allfourop}, it follows from the Weil conjectures proved by Deligne in \cite{WeilII} that $f^\sM_*L$ is pure of weight $w$ (see the remark after \cite[Definition 3.3]{Huber-perverse}). Hence, \ref{prop:purityintermediate} ensures that $f_*^\sM j^\sM_{!*}K$ is pure of weight $w$. This gives the conclusion.
\end{proof}

In particular, this provides (for geometric variations of Hodge structures) an arithmetic proof of Zucker's theorem \cite[Theorem p.416]{MR534758} via reduction to positive characteristic and to the Weil conjectures \cite[Th\'eor\`eme 2]{WeilII}. More generally, in higher dimension:

\begin{coro}\label{coro:weightIC} Let $k$ be a field embedded into $\bbC$.
Let $X$ be an irreducible proper $k$-variety and $\mathscr L$ be a $\bbQ$-local system on a smooth dense open subscheme $U$ of $X$ of the form $\mathscr L=R^wg_*\bbQ_V$ where $g:V\ra U$ is a  smooth proper morphism  and $w\in\bbZ$ is an integer. Then, the intersection cohomology group $IH^i(X,\mathscr L)$, for $i\in\bbZ$, is canonically the Betti realization of a Nori motive over $k$ which is pure of weight $i+w$. In particular, $IH^i(X,\mathscr L)$ carries a canonical pure Hodge structure of weight $i+w$.
\end{coro}

\begin{proof}
Let $d$ be the dimension of $X$, $\pi:X\ra\Spec(k)$ be the structural morphism and $j$ be the inclusion
of $U$ in $X$.
As in \ref{coro:IC}, $IH^i(X,\mathscr L)$ is the Betti realization of the Nori motive $\mathrm{H}^{i-d}(\pi_*^\sM j^\sM_{!*}\mathrm{H}^{w+d}(g^\sM_*\bbQ^{\mathscr M}_V))$, which is pure of weight $w+i$ by \ref{theo:weightIC}.
\end{proof}


\bibliographystyle{alpha}
\bibliography{IvorraMorel}

\end{document}